\documentclass[reqno]{amsart}
\usepackage{amssymb,amsfonts,amstext,amsmath,amsthm} 
\usepackage{float}
\usepackage{enumerate}
\usepackage{subfigure}
	\usepackage{ulem}  
\newtheorem{thm}{Theorem}[section]
\newtheorem{cor}[thm]{Corollary}
\newtheorem{lem}[thm]{Lemma}
\newtheorem{prop}[thm]{Proposition}
\newtheorem{rem}[thm]{Remark}
\newtheorem{defi}[thm]{Definition}

\numberwithin{equation}{section}
\newcommand{\Om}{{\Omega}}

\newcommand{\R}{{\rm I}\!{\rm R}}

\allowdisplaybreaks[2] 

\begin{document}

\title[Positive steady states of an indefinite 
equation]{Positive steady states of an indefinite 
equation with a nonlinear boundary condition: existence, multiplicity, stability and asymptotic profiles}
\vspace{1cm}

\author{Humberto Ramos Quoirin}
\address{H. Ramos Quoirin \newline Universidad de Santiago de Chile, Casilla 307, Correo 2, Santiago, Chile}
\email{\tt humberto.ramos@usach.cl}

\author{Kenichiro Umezu}
\address{K. Umezu \newline Department of Mathematics, Faculty of Education, Ibaraki University, Mito 310-8512, Japan}
\email{\tt kenichiro.umezu.math@vc.ibaraki.ac.jp}

\thanks{The first author was supported by the FONDECYT grant 11121567}

\begin{abstract}
We investigate positive steady states of an indefinite superlinear reaction-diffusion equation arising from population dynamics, coupled
with a nonlinear boundary condition. Both the equation and the boundary condition depend upon a positive parameter $\lambda$, which is inversely proportional to the diffusion rate.
We establish several multiplicity results when the diffusion rate is large and analyze the asymptotic profiles and the stability properties of these steady states as the diffusion rate grows to infinity. In particular, our results show that in some cases bifurcation from zero and from infinity occur at $\lambda=0$. Our approach combines variational and bifurcation techniques. 
\end{abstract}


\maketitle


\section{Introduction and main results} \label{sec1} 
\setcounter{equation}{0}
Let $\Omega$ be a bounded and regular domain of $\mathbb{R}^N$ with $N \geq 2$. In this article we are concerned with the problem
$$ \left\{
\begin{array}{lll}
-\Delta u=\lambda(m(x)u+a(x)|u|^{p-2}u) & {\rm in } & \Omega,  \\
          \frac{\partial u}{\partial \bf{n}}=\lambda b(x)|u|^{q-2}u & {\rm on } & \partial \Omega,\\
\end{array}\right. \leqno{(P_\lambda)}
$$
where:
\begin{itemize}
\item $\Delta$ is the usual Laplacian in $\R^N$
\item $\lambda \in \R$
\item $1<q<2<p$ and if $N>2$ 
then $p < 2^*=\frac{2N}{N-2}$  
\item $m,a \in L^{\infty}(\Omega)$, $m^+ \not \equiv 0$
\item $b \in L^{\infty}(\partial \Omega)$
\item $\bf{n}$ is the outward unit normal to $\partial \Om$. 
\end{itemize}

Our main goal in this article is to carry on the study of $(P_\lambda)$, which was addressed in \cite{RQU} for the logistic case $a \leq 0$. By variational and bifurcation techniques, we established  existence and multiplicity results for non-negative solutions of $(P_\lambda)$. Moreover, the structure of the non-negative solutions set was also discussed. We intend now to deal with these issues in the case where $a$ changes sign. 

By a solution of $(P_\lambda)$ we mean a weak solution, i.e. $u \in H^1(\Om)$ satisfying
$$ 
\int_\Om \nabla u \nabla \varphi - \lambda \int_\Om m(x) u \varphi - \lambda \int_\Om a(x)|u|^{p-2}u \varphi - \lambda \int_{\partial \Om} b(x) |u|^{q-2}u \varphi = 0, \quad \forall \varphi \in H^1(\Om). 
$$
In this case, we may also say that the couple $(\lambda,u)$ is a solution of $(P_\lambda)$. As already pointed out in \cite{RQU}, solutions of $(P_\lambda)$ satisfy $u \in W^{2,r}_{\textrm{loc}}(\Omega) \cap \mathcal{C}^\theta (\overline{\Omega})$ for some $r>N$ and $0<\theta < 1$, so that by the weak maximum principle \cite[Theorem 9.1]{GT}, nontrivial non-negative solutions of $(P_\lambda)$ are strictly positive in $\Omega$.  

$(P_\lambda)$ describes the steady state of solutions of the corresponding initial boundary value problem
\begin{align}  \label{ibvp}
\begin{cases}
\dfrac{\partial u}{\partial t} = \nabla \cdot (d \nabla u) + (m(x) + a(x)|u|^{p-2})u & \mbox{in $(0, \infty)\times \Omega$}, \\ 
u(0,x) = u_0(x) \geq 0 & \mbox{in $\Omega$}, \\ 
(d\nabla u) \cdot \mathbf{n} = b(x)|u|^{q-2}u & \mbox{on $(0, \infty)\times \partial \Omega$}, 
\end{cases} 
\end{align} 
which appears as a model in population dynamics (see Cantrell and Cosner \cite{CC03}, G\'omez-Re\~nasco and L\'opez-G\'omez \cite{GRLG00}). Here the unknown function $u$ stands for the population density of some species having $m(x)$ as intrinsic growth rate and  $m(x) + a(x)|u|^{p-2}$ as extrinsic growth rate. If $a(x)\leq 0$ then the latter one is the well known logistic growth rate, with self-limitation ($a(x)<0$) or without limitation ($a(x)=0$), so that the region where $a(x)=0$ can be considered as a {\it refuge}. We can give the case $a(x)>0$ the following biological interpretation (cf. \cite{GRLG00}): in this case, the extrinsic growth rate measures the {\it symbiosis} effect due to the intraspecific cooperation whereas, in the case $a(x)<0$, it measures the crowding effect associated with competition. As for the nonlinear boundary condition, it suggests that the flux rate $(d \nabla u) \cdot \mathbf{n}$ of the population on $\partial \Omega$ is incoming or outgoing (according to the sign of $b(x)$) and depends nonlinearly on $u$ as $|u|^{q-2}u$ (cf.\ \cite{G-MM-RRS08}). From the population dynamics viewpoint, we point out that the parameter $\lambda$ appearing in $(P_\lambda)$ describes the reciprocal number of the diffusion coefficient $d>0$, and only non-negative solutions are of interest.

Elliptic problems with indefinite nonlinearities have been studied over the last 25 years, starting with the works of Bandle, Pozio and Tesei \cite{BPT}, Ouyang \cite{Ou}, Alama and Tarantello \cite{AT}, Berestycki, Capuzzo-Dolcetta and Nirenberg \cite{BCN94, BCN95}, Lopez-Gomez \cite{LG}, etc. Brown and Zhang \cite{BZ03} and Brown \cite{B07} used the Nehari manifold (or fibering) method to discuss existence, multiplicity, and non-existence of positive solutions for the problem $$ \left\{
\begin{array}{lll}
-\Delta u=\lambda m(x)u+a(x)|u|^{p-2}u & {\rm in } & \Omega,  \\
          u=0 & {\rm on } & \partial \Omega,\\
\end{array}\right.
$$
according to the position of $\lambda$. The sublinear case $1<p<2$ and the superlinear case $p>2$ were treated in  \cite{B07} and \cite{BZ03}, respectively.  We shall see in this article that the Nehari manifold method turns out to be efficient for $(P_\lambda)$ as well.
\par 
 We note that $(P_\lambda)$ is characterized by the combination of the nonlinearities $m(x)u + a(x)|u|^{p-2}u$ in $\Omega$ and $b(x)|u|^{q-2}u$ on $\partial \Omega$. Furthermore,  the signs of $m$, $a$ and $b$ may completely change the effect of these nonlinearities. For elliptic problems with such combined nonlinearities, we refer to Chipot, Fila, and Quittner \cite{CFQ91}, L\'opez-G\'omez, M\'arquez, and Wolanski \cite{L-GMW93}, Morales-Rodrigo and Su\'arez \cite{M-RS05}, Wu \cite{W06}. Besides investigating  existence and multiplicity of positive solutions, we shall analyze as well the structure of the positive solutions set. Our approach is mainly based on a detailed study of the energy functional associated to $(P_\lambda)$. Note however that unlike most of the aforementioned works, $(P_\lambda)$ lacks coercivity on its left-hand side, since the term $(\int_\Omega |\nabla u|^2)^{\frac{1}{2}}$ does not correspond to the norm  of $u$ in $H^1(\Omega)$.  

Since the set of solutions of $(P_\lambda)$ at $\lambda = 0$ is explicitly provided by the constants, we may obtain positive solutions for $|\lambda|$ small by a bifurcation analysis on the line $(\lambda, u)=(0,c)$, where $c> 0$ is a constant, cf. \cite{U12,U13}. On the other hand, the lack of continuity of the derivative of $|u|^{q-2}u$ at $u=0$ prevents the use of the bifurcation approach to obtain positive solutions bifurcating from the zero solution. Finally, let us remark that the boundary point lemma cannot be applied directly to our problem, since it seems difficult to deduce that any nontrivial non-negative solution belongs to $\mathcal{C}^1(\overline{\Omega})$ in view of the assumption $0<q-1<1$. Therefore we are not able to infer that nontrivial non-negative solutions are positive on $\overline{\Omega}$. However we shall prove in Proposition \ref{prop:positive} that if $u$ is a nontrivial non-negative solution of $(P_\lambda)$ then the set $\{ x \in \partial \Omega : u(x) = 0 \}$ has no interior points in the relative topology of $\partial \Omega$. Note that, in the one-dimensional case $N=1$, the boundary point lemma is applicable, and we can deduce that any nontrivial non-negative weak solution of $(P_\lambda)$ is positive on $\overline{\Omega}$, see Proposition \ref{p:positivedef} in Appendix \ref{sec:posi}.

Let us set the notations and conventions used in this article: 
\begin{itemize}
\item The infimum of an empty set is assumed to be $\infty$.
\item Unless otherwise stated, for any $f \in L^1(\Om)$ the integral $\int_\Omega f$ is considered with respect to the Lebesgue measure, whereas for any $g \in L^1(\partial \Om)$ the integral $\int_{\partial \Om} g$  is considered with respect to the surface measure. 
\item For $r\geq 1$ the Lebesgue norm in $L^r (\Omega)$ will be denoted by $\| \cdot \|_r$ and the usual norm of $H^1(\Omega)$ by $\|\cdot \|$.
\item We set $2^*=\frac{2N}{N-2}$ and $2_*=\frac{2(N-1)}{N-2}$ for $N>2$. 
\item The strong and weak convergence are denoted by $\rightarrow$ and $\rightharpoonup$, respectively. 
\item The positive
and negative parts of a function $u$ are defined by $u^{\pm} :=\max
\{\pm u,0\}$. 
\item If $U \subset \R^N$ then we denote the closure of $U$ by $\overline{U}$ and the interior of $U$ by $\text{int }U$.
\item The support of a measurable function $f$ is denoted by supp $f$.
\item If $\Phi$ is a functional defined on $H^1(\Omega)$, we set
$$\Phi^{\pm}:=\{u\in H^1(\Omega);~\Phi(u)\gtrless 0\}, \quad \Phi_{0}:= \Phi^{-1}(0) \quad \text{and} \quad \Phi_0^{\pm}:= \Phi^{\pm} \cup \Phi_{0}.$$ 
\end{itemize}

Recall that 
\begin{equation}
\lambda_1=\lambda_1(m):=\inf\left\{\int_{\Omega} |\nabla u|^2; u \in H^1(\Omega), \int_{\Omega}mu^2=1\right\}
\end{equation}
 is a principal and simple eigenvalue of the problem
$$ \left\{
\begin{array}{lll}
-\Delta u=\lambda m(x)u & {\rm in } & \Omega,  \\
          \frac{\partial u}{\partial \bf{n}}=0 & {\rm on } & \partial \Omega.\\
\end{array}\right. 
$$
It is well known (cf.\ Brown and Lin \cite{BL}) that $\lambda_1(m)>0$ if and only if $\int_{\Omega} m<0$, in which case $\lambda_1(m)$ is achieved by a unique non-constant eigenfunction $\varphi_1$ such that $\varphi_1>0$ on $\overline{\Omega}$. If $\int_{\Omega} m>0$ then $\lambda_1(m)=0$ is achieved by  $\varphi_1=\left(\int_{\Omega} m\right)^{-\frac{1}{2}}$.

We set 
\begin{equation}
\label{cpq}
C_{pq}=\frac{q(p-2)}{2(p-q)}\left(\frac{p(2-q)}{2(p-q)}\right)^{\frac{2-q}{p-2}}
\end{equation} 
and
\begin{equation}
\label{ek1}
K_1(m,a)=C_{pq} \frac{\left(\int_\Om m\right)^{\frac{p-q}{p-2}}}{\left(-\int_{\Omega} a\right)^{\frac{2-q}{p-2}}},
\end{equation}
whenever $\int_\Omega m > 0>\int_\Omega a$.

Let $\varphi:[0,\infty) \rightarrow \R$ be given by \begin{equation}
\label{ephi}
\varphi (t) = t^{2-q} \int_\Omega m  + t^{p-q} 
\int_{\Omega} a + \int_{\partial \Omega} b.
\end{equation} 
It is easily seen that if either
$$\int_{\Omega} m> 0>\int_{\Omega} a \quad \text{and} \quad -K_1(m,a)<\int_{\partial \Omega} b<0$$
or
$$\int_{\Omega} m< 0<\int_{\Omega} a \quad \text{and} \quad K_1(-m,-a)>\int_{\partial \Omega} b>0$$
then $\varphi$ has two zeros $c_1<c_2$, which satisfy
\begin{equation}
\label{eph}
c_1<\left(-\frac{(2-q)\int_\Omega m}{(p-q)\int_\Omega a} \right)^{\frac{1}{p-2}} < c_2. 
\end{equation}
If either $$\int_\Omega a<0<\int_{\partial \Omega} b \quad \text{or} \quad \int_\Omega a>0>\int_{\partial \Omega} b$$
then
$\varphi$ has an unique zero, denoted by $c_0$.

In order to simplify the notation, we introduce the functionals
$$E_{\lambda}(u)=\int_{\Omega}\left(|\nabla u|^2 -\lambda m(x) u^2\right), \quad A(u)=\int_{\Omega} a(x)|u|^p \quad \text{and} \quad B(u)=\int_{\partial \Omega} b(x)|u|^q,$$
defined on $H^1(\Om)$.
From the compactness of the Sobolev embeddings
$$H^1(\Om)\hookrightarrow L^r(\Om), \quad H^1(\Om) \hookrightarrow L^t(\partial \Omega),$$ 
for $r \in [1,2^*)$ and $t \in [1,2_*)$ respectively, it is straightforward that $E_\lambda$ is weakly lower semi-continuous, whereas $A$ and $B$ are weakly continuous.
\begin{rem}{\rm 
 The following results, which can be easily verified, will be used repeatedly throughout this article: 
\begin{enumerate}
\item If $(u_n)$ is a bounded sequence in $H^1(\Om)$ then we may assume that $u_n \rightharpoonup u_0$, $A(u_n) \rightarrow A(u_0)$, $B(u_n) \rightarrow B(u_0)$, and $E_\lambda(u_0) \leq \liminf E_{\lambda}(u_n)$, for some $u_0 \in H^1(\Om)$.\\
\item If $u_n \rightharpoonup 0$ in $H^1(\Om)$ and $\limsup E_{\lambda}(u_n) \leq 0$ then $u_n \rightarrow 0$ in $H^1(\Om)$.
\end{enumerate}   
}\end{rem}
Let us set
\begin{equation}
\lambda_a =\lambda_a(m):=\inf\left\{\int_{\Omega} |\nabla u|^2;    u \in A_0^+, \int_{\Omega}mu^2=1 \right\},
\end{equation} 
 \begin{equation}
\lambda_b =\lambda_b(m):=\inf\left\{\int_{\Omega} |\nabla u|^2;  u \in B_0^+, \int_{\Omega}mu^2=1 \right\},
\end{equation} 
and
\begin{equation}
\label{els}
\lambda_s=\lambda_s(m,a,b)=\inf \left\{ \int_{\Omega} |\nabla u|^2;\ u \in A_0^+ \cap B_0^+, \ S(u)=1\right\},
\end{equation}
where
\begin{equation}
\label{eqs}
S(u)=\int_{\Omega} mu^2+\left(C_{pq}^{-1} B(u)A(u)^{\frac{2-q}{p-2}}\right)^{\frac{p-2}{p-q}}
\end{equation}
is defined for $u \in A_0^+ \cap B_0^+$. One may easily check that $S$ is quadratic, i.e. $S(tu)=t^2S(u)$ for $t \in \R$ and $u \in A_0^+ \cap B_0^+$. Lastly, we set
\begin{equation}
\label{defompm}
\Omega_{\pm}:=\text{int supp } a^{\pm}.
\end{equation}

We are now in position to state our main results. To begin with, we state an existence result, which provides bifurcation from zero and from infinity at $\lambda=0$ in the following sense:
\begin{defi}{\rm
It is said that {\it bifurcation from zero occurs at $\lambda = \lambda_0 \in \R$} for $(P_\lambda)$ if there exist nontrivial nonnegative solutions $u_j$ of $(P_{\lambda_j})$ such that $\lambda_j \to \lambda_0$, and $u_j \to 0$ in $\mathcal{C}(\overline{\Omega})$ as $j\to \infty$. Similarly, 
it is said that {\it bifurcation from infinity occurs at $\lambda = \lambda_\infty \in \R$} for $(P_\lambda)$ if there exist nontrivial nonnegative solutions $u_j$ of $(P_{\lambda_j})$ such that $\lambda_j \to \lambda_\infty$, and $\| u_j \|_{\mathcal{C}(\overline{\Omega})} \to \infty$ as $j\to \infty$.}
\end{defi}
\begin{thm}
\label{t1}
\strut
\begin{enumerate}
\item If $\int_{\partial \Omega} b<0$ then $\lambda_b, \lambda_s>0$. If, in addition, $b^+ \not \equiv 0$ then $(P_\lambda)$ has a nontrivial non-negative solution $u_{0,\lambda}$ for $0<\lambda <\min\{\lambda_b,\lambda_s\}$, 
which satisfies $u_{0,\lambda} \rightarrow 0$  in $\mathcal{C}^{\theta}(\overline{\Om})$ for some $\theta \in (0,1)$ as $\lambda \to 0^+$. 
Moreover, there exists $\lambda_n \to 0^+$ such that $\lambda_n^{-\frac{1}{2-q}}u_{0,\lambda_n} \rightarrow w_0$ in $H^1(\Omega)\cap \mathcal{C}^{\theta} (\overline{\Omega})$ as $n \to \infty$, where $w_0$ is a nontrivial non-negative solution of 
\begin{equation} \label{w0}
-\Delta w = 0 \quad\mbox{in} \ \Omega, \qquad
\frac{\partial w}{\partial \mathbf{n}} = b(x)w^{q-1} \quad \mbox{on} \
\partial \Omega.
\end{equation}
Furthermore, $w_0 > 0$ in $\Omega$, the set $\{ x\in \partial \Omega : w_0 = 0 \}$ has no interior points in the relative topology of $\partial \Omega$, and it is contained in $\{ x \in \partial \Omega : b(x) \leq 0 \}$ if $b\in \mathcal{C}(\partial \Omega)$.\\ 
	\item If $\int_{\Omega} a<0$ then $\lambda_a, \lambda_s>0$. If, in addition, $a^+ \not \equiv 0$ then $(P_\lambda)$ has a nontrivial non-negative solution $u_{2,\lambda}$ for $0<\lambda <\min\{\lambda_a,\lambda_s\}$, which satisfies $u_{2,\lambda} > 0$ in $\overline{\Omega}$ for $\lambda > 0$ sufficiently small and $\displaystyle \min_{\overline{\Omega}} u_{2,\lambda} \to \infty$ as $\lambda \to 0^+$. Moreover, there exists $\lambda_n \to 0^+$ such that $\lambda_n^{\frac{1}{p-2}}u_{2,\lambda_n} \rightarrow w_{\infty}$ in $H^1(\Omega)\cap C^\theta (\overline{\Omega})$ for some $\theta \in (0,1)$ as $n \to \infty$, where $w_\infty \in W^{2,r}(\Omega)$, $r>N$, is a solution of the  problem
	$$
	-\Delta w = a(x)w^{p-1} \quad\mbox{in} \ \Omega, \qquad
	\frac{\partial w}{\partial \mathbf{n}} = 0 \quad \mbox{on} \
	\partial \Omega. 
	$$
Furthermore $w_\infty>0$ on $\overline{\Omega}$. 
%
%
\end{enumerate}
\end{thm}

\begin{rem}
\strut
{\rm 
\begin{enumerate}
\item  Theorem \ref{t1} states that if $b^+ \not \equiv 0$ and $\int_{\partial \Omega} b<0$ then $(P_\lambda)$ has, besides the trivial branch $\{(\lambda,u)=(0,c); c \text{ is a positive constant}\}$, the solution $u_{0,\lambda}$ bifurcating from zero at $\lambda=0$. In a similar way, if  $a^+ \not \equiv 0$ and $\int_{\Omega} a<0$ then $(P_\lambda)$ has, besides the trivial branch of positive constants, the solution $u_{2,\lambda}$ bifurcating from infinity at $\lambda=0$. Furthermore, the blow up of $u_{2,\lambda}$ as $\lambda \to 0^+$ occurs uniformly on $\overline{\Omega}$. \\

\item We shall see in Proposition \ref{prop:positive} that for any $x \in \partial \Omega$ such that $b(x) > 0$ there holds $u_{0,\lambda}(x)>0$ if $\lambda>0$ is sufficiently small. More precisely, if $\Sigma \subset \partial \Omega$ is such that $\overline{\Sigma} \subset \{ x \in \partial \Omega : b(x) > 0 \}$ then $\displaystyle \inf_{\Sigma}u_{0,\lambda} > 0$ for $\lambda > 0$ sufficiently small.
\end{enumerate}
}\end{rem}

Our next results provide multiplicity of non-negative solutions and are stated under the condition
\begin{equation}
\label{cm}
\int_\Omega m >0.
\end{equation}
Since we are considering $\lambda \in \R$, the case $\int_{\Omega} m<0$ reduces to \eqref{cm} after the change of variable $\lambda \mapsto -\lambda$.

\begin{thm}
\label{t1'}
Assume \eqref{cm}, $\int_\Om a<0$, and $-K_1(m,a)<\int_{\partial \Om} b <0$. Then:
\begin{enumerate}
\item  $(P_\lambda)$ has a nontrivial non-negative solution $u_{1,\lambda}$ for $0<\lambda<\min\{\lambda_a,\lambda_b\}$. Moreover
$u_{1,\lambda} \rightarrow c_2$ in $\mathcal{C}^{\theta}(\overline{\Om})$ for some $\theta \in (0,1)$ as $\lambda \to 0^+$.\\
\item $(P_\lambda)$ has two nontrivial non-negative solutions $v_{1,\lambda}$, $v_{2,\lambda}$ for $\max\{\tilde{\lambda}_1,\tilde{\lambda}_s\}<\lambda<0$, where $\tilde{\lambda}_1=-\lambda_1(-m)$ and $\tilde{\lambda}_s=-\lambda_s(-m,-a,-b)$. Furthermore, $v_{1,\lambda} \to c_1$ and $v_{2,\lambda} \to c_2$ in $\mathcal{C}^{\theta}(\overline{\Om})$ for some $\theta \in (0,1)$ as $\lambda \to 0^-$.
\end{enumerate}

\end{thm}

Under the assumptions of Theorem \ref{t1'} and stronger regularity conditions on $m,a,b$, we shall obtain, for $\lambda > 0$ sufficiently small, a positive solution of $(P_\lambda)$ converging to $c_1$ as $\lambda\to 0^+$. This will be carried out by a bifurcation argument and, as a consequence, will provide at least four nontrivial non-negative solutions of $(P_\lambda)$ for $\lambda > 0$ small enough. Since this argument requires only the existence of zeros of $\varphi$ and some regularity on $m$, $a,$ and $b$, we shall assume, in addition to \eqref{cm},
\begin{align} \label{assump:mbreg}
m,a \in \mathcal{C}^\alpha (\overline{\Omega}) \quad \text{and} \quad b \in 
\mathcal{C}^{1+\alpha}(\partial \Omega),\quad \text{for some} 
\quad 0<\alpha < 1,
\end{align}
and
\begin{align}
\int_\Omega a<0 \quad \text{and} \quad -\tilde{K}_1(m,a)< \int_{\partial \Omega} b<0, \label{assump:tildK}
\end{align}
where 
\begin{equation}
\label{ktilda}
\tilde{K}_1(m,a) = \tilde{C}_{pq} \frac{\left(\int_\Om m\right)^{\frac{p-q}{p-2}}}{\left(-\int_{\Omega} a\right)^{\frac{2-q}{p-2}}},
\end{equation}
and $$\tilde{C}_{pq} =\left(\frac{q}{2}\left( \frac{p}{2} \right)^{\frac{2-q}{p-2}} \right)^{-1}C_{pq}.$$ It is easily seen that $\tilde{C}_{pq} >C_{pq}$, so that  $\tilde{K}_1(m,a)>K_1(m,a)$.  

Hereafter, by a classical positive solution of $(P_\lambda)$ we mean $u \in \mathcal{C}^{2+\theta}(\overline{\Omega})$ for some $0< \theta < 1$ which satisfies $(P_\lambda)$ in the classical sense and is strictly positive on $\overline{\Omega}$.

\begin{thm} \label{t3b} 
Assume \eqref{assump:mbreg}. If $a$ changes sign and $\int_\Omega a < 0$, then, for $\lambda>0$ sufficiently small, the solution $u_{2,\lambda}$ obtained in Theorem \ref{t1} is a classical positive solution of $(P_\lambda)$, which is moreover unstable.
\end{thm}

\begin{rem}
\label{rw0}
{\rm Since $1<q<2$, solutions of \eqref{w0} are in $\mathcal{C}^{\theta}(\overline{\Om})$ with $\theta \in (0,1)$, so that the boundary point lemma is not applicable, and consequently we can not deduce the positivity of $w_0$ on $\partial \Omega$. If this is the case then we can prove a result similar to Theorem \ref{t3b} for $u_{0,\lambda}$, obtained in Theorem \ref{t1}. More precisely, let us assume \eqref{assump:mbreg} and $w_0>0$ on $\overline{\Omega}$. If $b$ changes sign and $\int_{\partial \Omega}b<0$ then, for $\lambda > 0$ sufficiently small, $u_{0,\lambda}$ is a classical positive solution of $(P_\lambda)$ which is moreover asymptotically stable. We include a sketch of the proof of this result in Section \ref{sec5}.
}\end{rem}


\begin{thm} \label{t3} 
Assume \eqref{cm}, \eqref{assump:mbreg} and \eqref{assump:tildK}. 
Then there exists $\underline{\lambda} >0$ such that: 
\begin{enumerate}
  \item  $u_{1,\lambda}$ is a classical positive solution of $(P_\lambda)$ which is asymptotically stable for $\lambda \in (0,\underline{\lambda})$. Moreover, $(P_\lambda)$ has a classical positive solution $u_{3,\lambda}$ for $\lambda \in (0,\underline{\lambda})$ which is unstable. These solutions are continuous in $\mathcal{C}^{2+\alpha}(\overline{\Omega})$ with respect to $\lambda$ and emanate from $(0,c_2)$ and $(0,c_1)$ respectively, i.e. $u_{1,0}=c_2$ and $u_{3,0}=c_1$. Finally, there are no other classical positive solutions of $(P_\lambda)$ converging to some positive constant in $\mathcal{C}(\overline{\Omega})$ as $\lambda \to 0^+$. \\
  \item $v_{1,\lambda}, v_{2,\lambda}$ are classical positive solutions of $(P_\lambda)$ for $\lambda \in (-\underline{\lambda},0)$, which are asymptotically stable, and unstable, respectively. Moreover, these solutions are continuous in $\mathcal{C}^{2+\alpha}(\overline{\Omega})$ with respect to $\lambda$ and emanate from $(0,c_1)$ and $(0,c_2)$ respectively, i.e. $v_{1,0}=c_1$ and $v_{2,0}=c_2$. Finally, there are no other classical positive solutions of $(P_\lambda)$ converging to some positive constant in  $\mathcal{C}(\overline{\Omega})$ as $\lambda \to 0^-$. 
\end{enumerate}
\end{thm}

\begin{rem}
{\rm Theorem \ref{t3} states, in particular, that $v_{1,\lambda}$, $v_{2,\lambda}$ are the extensions of $u_{3,\lambda}$, $u_{1,\lambda}$, respectively, to the region $\lambda < 0$ (see Figure \ref{fig1}).}
\end{rem}

Combining Theorems \ref{t1} and \ref{t1'}, we get a multiplicity result for $(P_\lambda)$ with $\lambda > 0$ small, which is obtained by  variational techniques:

\begin{cor}  \label{t2}
Under the assumptions of Theorem \ref{t1'}, assume moreover that $a^+\not\equiv 0$ and $b^+\not\equiv 0$. 
Then $(P_\lambda)$ has three nontrivial non-negative solutions $u_{j,\lambda}$, $j=0,1,2$, for $0<\lambda <\min\{\lambda_s,\lambda_a,\lambda_b\}$, which satisfy $u_{0,\lambda} \rightarrow 0$ and $u_{1,\lambda} \rightarrow c_2$ in $\mathcal{C}^{\theta}(\overline{\Om})$ for some $\theta \in (0,1)$ and $\displaystyle \min_{\overline{\Omega}} u_{2,\lambda} \to \infty$ as $\lambda \to 0^+$.
\end{cor}       

From Corollary \ref{t2} and Theorem \ref{t3}(1), we infer the following multiplicity result, which is obtained combining the variational and bifurcation approaches: 

\begin{cor}
In addition to the assumptions of Theorem \ref{t1'}, we assume \eqref{assump:mbreg}. 
Assume moreover $a^+\not \equiv 0$. Then $(P_\lambda)$ has at least  three classical positive solutions for $0<\lambda<\underline{\lambda}$, namely $u_{j,\lambda}$, $j=1,2,3$.  If, in addition, $b^+ \not \equiv 0$ then $(P_\lambda)$ has a fourth nontrivial non-negative solution for $0<\lambda<\underline{\lambda}$, namely, $u_{0,\lambda}$.
\end{cor}

\begin{rem} \label{rem:classifiP}
{\rm We shall see in Proposition \ref{prop:bif} that under \eqref{cm},  \eqref{assump:mbreg}, and \eqref{assump:tildK}, $(P_\lambda)$ has, for $|\lambda|$ sufficiently small, two classical positive solutions $U_{j, \lambda}$, $j=1,2$, which are continuous (with respect to $\lambda$) in $\mathcal{C}^{2+\alpha}(\overline{\Omega})$ and satisfy $U_{j,0}=c_j$, $j=1,2$. Furthermore, this result does not require the condition $p<2^*$.
So, under \eqref{cm}, \eqref{assump:mbreg} and the conditions
$$a,b \text{ change sign,} \quad \int_\Omega a<0, \quad \text{and} \quad  \int_{\partial \Omega} b<0,$$ 
 Theorem \ref{t1} and Proposition \ref{prop:bif} provide, according to the value of $\int_{\partial \Omega}b$,  the following result on the number of nontrivial non-negative solutions of $(P_\lambda)$:\\
\begin{enumerate}
\item If $\int_{\partial \Omega}b >-\tilde{K_1}(m,a)$ then $(P_\lambda)$ has at least four nontrivial non-negative solutions (two variational nontrivial non-negative solutions, among which one is a classical positive solution, and two bifurcating classical positive solutions) for $\lambda > 0$ sufficiently small.\\
\item If $\int_{\partial \Omega}b \leq -\tilde{K_1}(m,a)$ then $(P_\lambda)$ has at least two (variational) nontrivial non-negative solutions, among which one is a classical positive solution, for $\lambda > 0$ sufficiently small. Moreover, if $\int_{\partial \Omega}b < -\tilde{K_1}(m,a)$ then $(P_\lambda)$ has no classical positive solutions converging to a positive constant as $\lambda \to 0^+$.
\end{enumerate} 
}\end{rem}

The following condition provides an {\it a priori} bound on the values of $\lambda$ for which $(P_\lambda)$ has a nontrivial non-negative solution: 
	$$ \left\{
	\begin{array}{lll} \text{There are smooth sub-domains }D_{\pm} \Subset \Omega \text{ such that }\\ D_{\pm} \subset \Omega_{\pm}
	 \text{ and $m$ changes sign in $D_+$ and in $D_-$.} \end{array}\right. \leqno{\mathcal{(H)}}$$

	\begin{thm} 
	\label{t4}
	Assume that $a$ changes sign and $\mathcal{(H)}$ holds.  Then there exists a constant $\Lambda > 0$ such that if $u$ is a nontrivial non-negative solution of $(P_\lambda)$ then $|\lambda|\leq \Lambda$. 
	\end{thm}

As a particular case of $(P_\lambda)$, we shall consider the problem
$$ \left\{
\begin{array}{lll}
-\Delta u=\lambda m(x)(u-|u|^{p-2}u) & {\rm in } & \Omega,  \\ 
          \frac{\partial u}{\partial \bf{n}}=\lambda b(x)|u|^{q-2}u & {\rm on } & \partial \Omega,\\
\end{array}\right. \leqno{(Q_\lambda)} 
$$
where $\lambda$, $m$, $b$, $p$, $q$ are as above. Note that $(Q_\lambda)$ corresponds to $(P_\lambda)$ with $a\equiv -m$.  The nonlinearity $\lambda m(x)(u-|u|^{p-2}u)$ arises from population genetics and has already been studied under homogeneous boundary conditions in \cite{BH, F, S}.

Let 
\begin{equation*}
\lambda_r =\lambda_r(m,b):=\inf\left\{\int_{\Omega} |\nabla u|^2; \ u \in H^1(\Omega), \ \int_{\Omega}mu^2=1, \int_\Om m|u|^p \leq 0 \leq \int_{\partial \Om} b|u|^q  \right\},
\end{equation*} 
$$\lambda_t=\lambda_t(m,b):=\inf \left\{ \int_{\Omega} |\nabla u|^2;\ u \in H^1(\Omega), \ \int_\Om m|u|^p \leq 0 \leq \int_{\partial \Om} b|u|^q, \ T(u)=1\right\},$$
where
$$T(u)=\int_{\Omega} mu^2+\left(C_{pq}^{-1} \left(\int_{\partial \Om} b|u|^q \right)\left(-\int_\Om m|u|^p\right)^{\frac{2-q}{p-2}}\right)^{\frac{p-2}{p-q}},$$
is defined for $u \in H^1(\Om)$ such that $\int_\Om m|u|^p \leq 0 \leq \int_{\partial \Om} b|u|^q$. Note that $\lambda_r=\min\{\lambda_a(m),\lambda_b(m)\}$ and $\lambda_t=\lambda_s(m,a,b)$ with $a\equiv -m$. 

Finally, let 
$$\phi(t)=\left(t^{2-q}-t^{p-q}\right)\int_\Om m + \int_{\partial \Om} b.$$

The assumption of Theorem \ref{t1'}
reads now
\begin{equation}
\label{emf}\int_{\Omega} m>0 >\int_{\partial \Omega} b> -C_{pq} \int_\Om m,
\end{equation}
in which case $\phi$ has two positive zeros $c_1<c_2$.


Applying Theorems \ref{t1} and \ref{t1'} to $(Q_\lambda)$ we obtain the following: 

\begin{thm}
\label{t5}
Assume \eqref{emf} and $m$ changes sign. Then $\lambda_r,\lambda_t>0$ and:
\begin{enumerate}
\item $(Q_\lambda)$ has two nontrivial non-negative solutions $u_{1,\lambda}$, $u_{2,\lambda}$ for $0<\lambda< \min\{\lambda_1, \lambda_t\}$, which satisfy $u_{1,\lambda} \rightarrow c_2$ in $\mathcal{C}^{\theta}(\overline{\Om})$ for some $\theta \in (0,1)$ and $\displaystyle \min_{\overline{\Omega}} u_{2,\lambda} \rightarrow \infty$ as $\lambda \to 0^+$. If, in addition, $b^+ \not \equiv 0$, then $(Q_\lambda)$ has a further nontrivial non-negative solution $u_{0,\lambda}$ for $0<\lambda <\lambda_r$,
which satisfies $u_{0,\lambda} \rightarrow 0$  in $\mathcal{C}^{\theta}(\overline{\Om})$ as $\lambda \to 0^+$. \\
\item $(Q_\lambda)$ has two nontrivial non-negative solutions $v_{1,\lambda}$, $v_{2,\lambda}$ for $\max\{\tilde{\lambda}_1,\tilde{\lambda}_t\}<\lambda<0$, where $\tilde{\lambda}_1=-\lambda_1(-m)$ and $\tilde{\lambda}_t=-\lambda_t(-m,-b)$. Furthermore, $v_{1,\lambda} \to c_1$ and $v_{2,\lambda} \to c_2$ in $\mathcal{C}^{\theta}(\overline{\Om})$ for some $\theta \in (0,1)$ as $\lambda \to 0^-$.
\end{enumerate}
\end{thm}      

Lastly, we consider the case where $\phi$ has a unique positive zero $c_0$ such that $\phi'(c_0)=0$, which occurs precisely when
\begin{align} \label{assump:=-tildeK1}
\int_\Omega m > 0 > \int_{\partial \Omega}b = - \tilde{K}_1(m,-m),
\end{align}
where $\tilde{K}_1$ is defined in \eqref{ktilda}.
Under (\ref{assump:mbreg}), we obtain a smooth curve of classical positive solutions of $(Q_\lambda)$ for $|\lambda|$ sufficiently small by a bifurcation argument.


The following result asserts that if \eqref{assump:=-tildeK1} holds then $(0, c_0)$ is a turning point to the right on the smooth curve of positive solutions of $(Q_\lambda)$. 

\begin{thm} \label{thm:a=-m:bif} 
Assume \eqref{assump:mbreg}. If $(Q_\lambda)$ has a classical positive solution $u$ with $\lambda \not= 0$ such that $u_\lambda \rightarrow c$ in $\mathcal{C}(\overline{\Omega})$ as $\lambda \to 0$, where $c$ is a positive constant, then  $\phi (c) = 0$. Moreover: 
\begin{enumerate}
\item Assume \eqref{assump:tildK} with $a=-m$. Then  there exist two arbitrarily smooth maps $\lambda \mapsto U_{1,\lambda}, U_{2,\lambda} \in \mathcal{C}^{2+\alpha}(\overline{\Omega})$ for $\lambda$ close to $0$ such that $U_{1,\lambda}$, $U_{2,\lambda}$ are classical positive solutions of $(Q_\lambda)$ and satisfy $U_{1,0}=c_1$ and $U_{2,0}=c_2$. \\

\item Assume \eqref{assump:=-tildeK1}. 
Then there exist two arbitrarily smooth maps $t \mapsto \lambda(t) \in \R$ and $t \mapsto u(t)\in \mathcal{C}^{2+\alpha}(\overline{\Omega})$ for $t$ close to $c_0$ such that $(\lambda,u)=(\lambda(t),u(t))$ is a classical positive solution of $(Q_\lambda)$ with $\lambda(c_0)=\lambda'(c_0)=0$, $\lambda''(c_0)>0$, and $u(t)= t + o(1)$ as $t\to c_0$. Moreover, there exists a constant $\varepsilon > 0$ such that $u(t)$ is asymptotically stable for $c_0<t<c_0+\varepsilon$ and unstable for $c_0-\varepsilon < t < c_0$. 
\end{enumerate}
\end{thm} 

\begin{rem}{\rm \strut
\begin{enumerate}
\item Theorem \ref{thm:a=-m:bif} holds in fact without the condition $p<2^*$. 
\item Assertion (1) in Theorem \ref{thm:a=-m:bif} corresponds to Theorem \ref{t3} for $(Q_\lambda)$. 
\end{enumerate}

}\end{rem}



\begin{rem}\label{rem:classifiQ}
{\rm 
 The results on the number of nontrivial non-negative solutions from Remark \ref{rem:classifiP} can be sharpened for $(Q_\lambda)$. In this case, Theorems \ref{t1} and \ref{thm:a=-m:bif} provide that, under \eqref{assump:mbreg}, if $m$ and $b$ change sign, and $\int_\Omega m>0>    \int_{\partial \Omega} b$, then $(Q_\lambda)$ has: 
  \begin{enumerate}
  \item at least four nontrivial non-negative solutions (two variational nontrivial non-negative solutions, among which one is a classical positive solution, and two bifurcating classical positive solutions) for $\lambda>0$ sufficiently small if $\int_{\partial \Omega}b \geq -\tilde{K_1}(m,-m)$.  \\
  \item at least two (variational) nontrivial non-negative solutions, among which one is a classical positive solution, for $\lambda>0$ sufficiently small and no classical positive solutions converging to a positive constant as $\lambda \to 0^+$ if $\int_{\partial \Omega}b < -\tilde{K_1}(m,-m)$. 
  \end{enumerate}
}\end{rem}

\subsection{Suggested bifurcation diagrams}
\strut
\medskip

In view of the results stated above, we analyze now the possible bifurcations diagrams of $(P_\lambda)$ and $(Q_\lambda)$. This will be done assuming that $a$ changes sign and $(\mathcal{H})$ holds, in which case Theorem \ref{t4} ensures that $(P_\lambda)$ has no nontrivial non-negative solution for  $|\lambda|>\Lambda$.

\begin{enumerate}
\item Assume \eqref{cm}, $a$ and $b$ change sign, and $\int_\Omega a<0$. If $-K_1(m,a)<\int_{\partial \Om} b<0$ then, by Theorems \ref{t1}, \ref{t1'}, and \ref{t3}, the bifurcation diagram for $(P_\lambda)$ is suggested by Figure \ref{fig1}.
Moreover, Remark \ref{rem:classifiP} suggests that this bifurcation diagram  approaches the one shown in Figure \ref{figbinfty} as $\int_{\partial \Omega} b \to -\infty$. Indeed, note that the value $\Lambda$ provided by Theorem \ref{t4} does not depend on $b$. Furthermore, it can be shown that $\lambda_b$ and $\lambda_s$ stay bounded away from zero if $b^+$ is bounded from above (cf. Remark \ref{rbb}). By a formal observation, the nonlinear boundary condition in $(P_\lambda)$ approaches the Dirichlet boundary condition $u=0$
as $\int_{\partial \Omega}b \to -\infty$. So, after the change of variable $v = \lambda^{\frac{1}{p-2}}u$ for $u\geq 0$, the limiting problem for $(P_\lambda)$ when $\int_{\partial \Omega}b \to -\infty$ would be $$-\Delta v = \lambda m(x)v + a(x)v^{p-1} \mbox{ in } \Omega, \quad v=0  \mbox{ on } \partial \Omega.$$ This problem has been investigated by Ouyang \cite{Ou}  in the case $m \equiv 1$ and the existence of a single turning point in the positive solutions set has been proved under some conditions on $a$. We expect then that the bifurcation diagram of $(P_\lambda)$ has a unique turning point if $\int_{\partial \Omega}b \ll 0$.
\\

\item Assume \eqref{cm}, $b \leq 0$, $a$ changes sign, and $\int_\Om a<0$.\\
\begin{enumerate}
\item If $-K_1(m,a)<\int_{\partial \Om} b<0$ then the bifurcation diagram for $(P_\lambda)$ in the case $b\leq 0$ is suggested by Figure \ref{fig0sol}. This is motivated by Theorems \ref{t1}, \ref{t1'}, and \ref{t3}, and the fact that $u_{0,\lambda}$ approaches zero as $b^+$ converges to zero. More precisely, in Proposition \ref{pmn1} we prove that if $b_n \rightarrow b$ in $L^{\infty}(\partial \Omega)$ with $b_n^+ \not \equiv 0$ for every $n$, then there exists $\lambda_0>0$ such that the solution $u_{0,\lambda,b_n}$ exists for every $n$, and satisfies $u_{0,\lambda,b_n} \to 0$ in $\mathcal{C}^{\theta}(\overline{\Om})$ for some $\theta \in (0,1)$ and $0<\lambda<\lambda_0$. Note that in \cite{RQU3} we have obtained bifurcation from zero in the case $b\leq 0$ and $b\not\equiv 0$ when $a< 0$.\\
\item In a similar way, if $-\tilde{K}_1(m,a)-\varepsilon<\int_{\partial \Om} b<-\tilde{K}_1(m,a)$ with $\varepsilon>0$ sufficiently small then the bifurcation diagram for $(P_\lambda)$ would be suggested by Figure \ref{figbnonposi:large}. Lastly, as $\int_{\partial \Om} b \to -\infty$, the bifurcation diagram for $(P_\lambda)$ is expected to approach Figure \ref{figbnonposi:enoughlarge}.\\
\end{enumerate}
\item Assume that $m$ and $b$ change sign, and $\int_\Omega m>0>    \int_{\partial \Omega} b = -\tilde{K_1}(m,-m)$. By Remark \ref{rem:classifiQ}, the bifurcation diagram of $(Q_\lambda)$ is suggested by Figure \ref{fig:a=-m}.
\end{enumerate}

\begin{rem} \label{rem:bio} {\rm
From the biological viewpoint, the bifurcation diagram 
suggested in Figure \ref{fig1} would provide three possible conditional states for the population as the diffusion rate grows to infinity, namely: extinction, explosion and persistence with a spatially uniform distribution.
}\end{rem} 

	\begin{figure}[H]
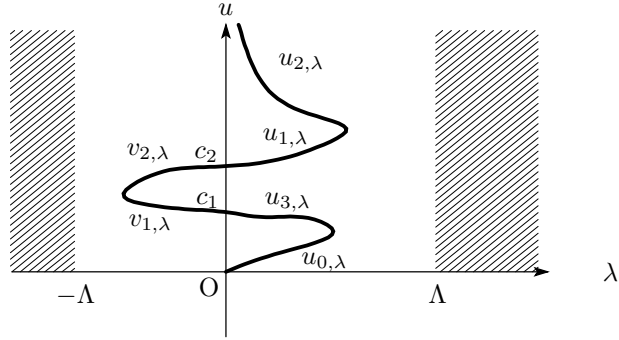
 
		\input fig14_0516ss
	  \caption{Suggested bifurcation diagram for $(P_\lambda)$ when $a$ and $b$ change sign, $\int_\Omega m > 0 > \int_\Omega a$, and $0> \int_{\partial \Omega}b > -K_1(m,a)$.} \label{fig1}
	    \end{figure} 
	   
	\begin{figure}[H]
{\unitlength 0.1in
\begin{picture}( 24.6500, 20.9400)( 20.6000,-21.9000)
%
\special{pn 8}%
\special{pa 2060 1790}%
\special{pa 4526 1790}%
\special{fp}%
\special{sh 1}%
\special{pa 4526 1790}%
\special{pa 4458 1770}%
\special{pa 4472 1790}%
\special{pa 4458 1810}%
\special{pa 4526 1790}%
\special{fp}%
%
\special{pn 8}%
\special{pa 3046 2190}%
\special{pa 3046 268}%
\special{fp}%
\special{sh 1}%
\special{pa 3046 268}%
\special{pa 3026 334}%
\special{pa 3046 320}%
\special{pa 3066 334}%
\special{pa 3046 268}%
\special{fp}%
\put(29.6700,-18.8100){\makebox(0,0){O}}%
\put(48.1400,-17.9000){\makebox(0,0){$\lambda$}}%
%
\special{pn 8}%
\special{pa 4010 1790}%
\special{pa 4010 302}%
\special{pa 4480 302}%
\special{pa 4480 1790}%
\special{pa 4010 1790}%
\special{ip}%
%
\special{pn 4}%
\special{pa 4480 1174}%
\special{pa 4010 1686}%
\special{fp}%
\special{pa 4480 1216}%
\special{pa 4010 1728}%
\special{fp}%
\special{pa 4480 1258}%
\special{pa 4010 1770}%
\special{fp}%
\special{pa 4480 1300}%
\special{pa 4030 1790}%
\special{fp}%
\special{pa 4480 1342}%
\special{pa 4068 1790}%
\special{fp}%
\special{pa 4480 1384}%
\special{pa 4106 1790}%
\special{fp}%
\special{pa 4480 1426}%
\special{pa 4146 1790}%
\special{fp}%
\special{pa 4480 1468}%
\special{pa 4184 1790}%
\special{fp}%
\special{pa 4480 1510}%
\special{pa 4222 1790}%
\special{fp}%
\special{pa 4480 1552}%
\special{pa 4262 1790}%
\special{fp}%
\special{pa 4480 1594}%
\special{pa 4300 1790}%
\special{fp}%
\special{pa 4480 1636}%
\special{pa 4338 1790}%
\special{fp}%
\special{pa 4480 1678}%
\special{pa 4378 1790}%
\special{fp}%
\special{pa 4480 1720}%
\special{pa 4416 1790}%
\special{fp}%
\special{pa 4480 1762}%
\special{pa 4454 1790}%
\special{fp}%
\special{pa 4480 1130}%
\special{pa 4010 1644}%
\special{fp}%
\special{pa 4480 1088}%
\special{pa 4010 1602}%
\special{fp}%
\special{pa 4480 1046}%
\special{pa 4010 1558}%
\special{fp}%
\special{pa 4480 1004}%
\special{pa 4010 1516}%
\special{fp}%
\special{pa 4480 962}%
\special{pa 4010 1474}%
\special{fp}%
\special{pa 4480 920}%
\special{pa 4010 1432}%
\special{fp}%
\special{pa 4480 878}%
\special{pa 4010 1390}%
\special{fp}%
\special{pa 4480 836}%
\special{pa 4010 1348}%
\special{fp}%
\special{pa 4480 794}%
\special{pa 4010 1306}%
\special{fp}%
\special{pa 4480 752}%
\special{pa 4010 1264}%
\special{fp}%
\special{pa 4480 710}%
\special{pa 4010 1222}%
\special{fp}%
\special{pa 4480 668}%
\special{pa 4010 1180}%
\special{fp}%
\special{pa 4480 626}%
\special{pa 4010 1138}%
\special{fp}%
\special{pa 4480 584}%
\special{pa 4010 1096}%
\special{fp}%
\special{pa 4480 542}%
\special{pa 4010 1054}%
\special{fp}%
\special{pa 4480 500}%
\special{pa 4010 1012}%
\special{fp}%
\special{pa 4480 458}%
\special{pa 4010 970}%
\special{fp}%
\special{pa 4480 416}%
\special{pa 4010 928}%
\special{fp}%
\special{pa 4480 374}%
\special{pa 4010 886}%
\special{fp}%
\special{pa 4480 332}%
\special{pa 4010 844}%
\special{fp}%
\special{pa 4468 302}%
\special{pa 4010 802}%
\special{fp}%
\special{pa 4428 302}%
\special{pa 4010 760}%
\special{fp}%
\special{pa 4390 302}%
\special{pa 4010 716}%
\special{fp}%
\special{pa 4352 302}%
\special{pa 4010 674}%
\special{fp}%
\special{pa 4312 302}%
\special{pa 4010 632}%
\special{fp}%
\special{pa 4274 302}%
\special{pa 4010 590}%
\special{fp}%
\special{pa 4236 302}%
\special{pa 4010 548}%
\special{fp}%
\special{pa 4196 302}%
\special{pa 4010 506}%
\special{fp}%
\special{pa 4158 302}%
\special{pa 4010 464}%
\special{fp}%
\special{pa 4120 302}%
\special{pa 4010 422}%
\special{fp}%
\special{pa 4082 302}%
\special{pa 4010 380}%
\special{fp}%
\special{pa 4042 302}%
\special{pa 4010 338}%
\special{fp}%
\put(40.1600,-19.2300){\makebox(0,0){$\Lambda$}}%
\put(30.4500,-1.7600){\makebox(0,0){$u$}}%
\put(33.6600,-5.9000){\makebox(0,0){$u_{2,\lambda}$}}%
%
\special{pn 20}%
\special{pa 3046 1790}%
\special{pa 3070 1768}%
\special{pa 3092 1744}%
\special{pa 3116 1720}%
\special{pa 3188 1654}%
\special{pa 3236 1612}%
\special{pa 3284 1572}%
\special{pa 3310 1554}%
\special{pa 3334 1538}%
\special{pa 3386 1506}%
\special{pa 3438 1478}%
\special{pa 3466 1466}%
\special{pa 3522 1438}%
\special{pa 3552 1424}%
\special{pa 3612 1392}%
\special{pa 3644 1376}%
\special{pa 3678 1358}%
\special{pa 3712 1338}%
\special{pa 3776 1294}%
\special{pa 3804 1272}%
\special{pa 3830 1248}%
\special{pa 3848 1224}%
\special{pa 3862 1200}%
\special{pa 3868 1176}%
\special{pa 3866 1152}%
\special{pa 3856 1128}%
\special{pa 3840 1104}%
\special{pa 3816 1080}%
\special{pa 3788 1058}%
\special{pa 3758 1036}%
\special{pa 3728 1016}%
\special{pa 3696 996}%
\special{pa 3664 978}%
\special{pa 3634 960}%
\special{pa 3574 926}%
\special{pa 3546 908}%
\special{pa 3462 860}%
\special{pa 3410 828}%
\special{pa 3362 792}%
\special{pa 3340 774}%
\special{pa 3318 754}%
\special{pa 3278 710}%
\special{pa 3260 688}%
\special{pa 3242 662}%
\special{pa 3226 636}%
\special{pa 3196 582}%
\special{pa 3180 554}%
\special{pa 3168 524}%
\special{pa 3154 494}%
\special{pa 3142 464}%
\special{pa 3118 402}%
\special{pa 3094 338}%
\special{pa 3084 306}%
\special{pa 3078 288}%
\special{fp}%
\put(36.0000,-15.9000){\makebox(0,0){$u_{0,\lambda}$}}%
\end{picture}}%
	  \caption{Suggested bifurcation diagram for $(P_\lambda)$ when $a$ and $b$ change sign, 
	$\int_\Omega m > 0 > \int_\Omega a$,  and $\int_{\partial \Omega}b \to -\infty$.}
	\label{figbinfty}
	 \end{figure}

	\begin{figure}[H]
	\input fig14_0702s
  \caption{Suggested bifurcation diagram for $(P_\lambda)$ when $a$ changes sign, $b \leq 0$, $\int_\Omega m > 0 > \int_\Omega a$, and  $0> \int_{\partial \Omega}b > -K_1(m,a)$.} \label{fig0sol}
    \end{figure} 

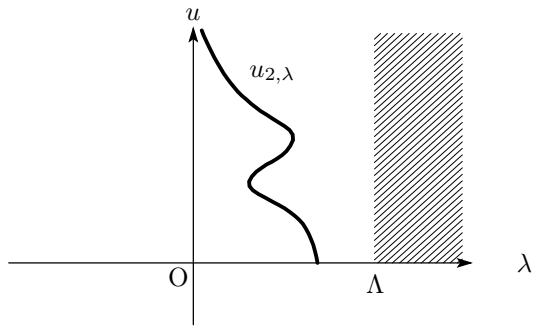
\begin{figure}[H]
{\unitlength 0.1in
\begin{picture}( 24.1800, 17.0700)( 20.6000,-18.0000)
%
\special{pn 8}%
\special{pa 2060 1476}%
\special{pa 4478 1476}%
\special{fp}%
\special{sh 1}%
\special{pa 4478 1476}%
\special{pa 4412 1456}%
\special{pa 4426 1476}%
\special{pa 4412 1496}%
\special{pa 4478 1476}%
\special{fp}%
%
\special{pn 8}%
\special{pa 3026 1800}%
\special{pa 3026 248}%
\special{fp}%
\special{sh 1}%
\special{pa 3026 248}%
\special{pa 3006 314}%
\special{pa 3026 300}%
\special{pa 3046 314}%
\special{pa 3026 248}%
\special{fp}%
\put(29.4900,-15.5100){\makebox(0,0){O}}%
\put(47.6100,-14.7600){\makebox(0,0){$\lambda$}}%
%
\special{pn 8}%
\special{pa 3974 1476}%
\special{pa 3974 274}%
\special{pa 4434 274}%
\special{pa 4434 1476}%
\special{pa 3974 1476}%
\special{ip}%
%
\special{pn 4}%
\special{pa 4434 978}%
\special{pa 3974 1392}%
\special{fp}%
\special{pa 4434 1012}%
\special{pa 3974 1426}%
\special{fp}%
\special{pa 4434 1046}%
\special{pa 3974 1460}%
\special{fp}%
\special{pa 4434 1080}%
\special{pa 3992 1476}%
\special{fp}%
\special{pa 4434 1114}%
\special{pa 4030 1476}%
\special{fp}%
\special{pa 4434 1148}%
\special{pa 4068 1476}%
\special{fp}%
\special{pa 4434 1184}%
\special{pa 4106 1476}%
\special{fp}%
\special{pa 4434 1216}%
\special{pa 4144 1476}%
\special{fp}%
\special{pa 4434 1250}%
\special{pa 4182 1476}%
\special{fp}%
\special{pa 4434 1284}%
\special{pa 4220 1476}%
\special{fp}%
\special{pa 4434 1318}%
\special{pa 4258 1476}%
\special{fp}%
\special{pa 4434 1354}%
\special{pa 4296 1476}%
\special{fp}%
\special{pa 4434 1386}%
\special{pa 4334 1476}%
\special{fp}%
\special{pa 4434 1420}%
\special{pa 4370 1476}%
\special{fp}%
\special{pa 4434 1456}%
\special{pa 4408 1476}%
\special{fp}%
\special{pa 4434 944}%
\special{pa 3974 1358}%
\special{fp}%
\special{pa 4434 910}%
\special{pa 3974 1324}%
\special{fp}%
\special{pa 4434 876}%
\special{pa 3974 1290}%
\special{fp}%
\special{pa 4434 842}%
\special{pa 3974 1256}%
\special{fp}%
\special{pa 4434 810}%
\special{pa 3974 1222}%
\special{fp}%
\special{pa 4434 774}%
\special{pa 3974 1188}%
\special{fp}%
\special{pa 4434 740}%
\special{pa 3974 1154}%
\special{fp}%
\special{pa 4434 706}%
\special{pa 3974 1120}%
\special{fp}%
\special{pa 4434 672}%
\special{pa 3974 1086}%
\special{fp}%
\special{pa 4434 638}%
\special{pa 3974 1052}%
\special{fp}%
\special{pa 4434 604}%
\special{pa 3974 1018}%
\special{fp}%
\special{pa 4434 570}%
\special{pa 3974 984}%
\special{fp}%
\special{pa 4434 536}%
\special{pa 3974 950}%
\special{fp}%
\special{pa 4434 502}%
\special{pa 3974 916}%
\special{fp}%
\special{pa 4434 468}%
\special{pa 3974 882}%
\special{fp}%
\special{pa 4434 434}%
\special{pa 3974 848}%
\special{fp}%
\special{pa 4434 400}%
\special{pa 3974 814}%
\special{fp}%
\special{pa 4434 366}%
\special{pa 3974 780}%
\special{fp}%
\special{pa 4434 332}%
\special{pa 3974 746}%
\special{fp}%
\special{pa 4434 298}%
\special{pa 3974 712}%
\special{fp}%
\special{pa 4422 274}%
\special{pa 3974 678}%
\special{fp}%
\special{pa 4384 274}%
\special{pa 3974 644}%
\special{fp}%
\special{pa 4346 274}%
\special{pa 3974 610}%
\special{fp}%
\special{pa 4308 274}%
\special{pa 3974 576}%
\special{fp}%
\special{pa 4270 274}%
\special{pa 3974 542}%
\special{fp}%
\special{pa 4232 274}%
\special{pa 3974 508}%
\special{fp}%
\special{pa 4194 274}%
\special{pa 3974 474}%
\special{fp}%
\special{pa 4156 274}%
\special{pa 3974 440}%
\special{fp}%
\special{pa 4118 274}%
\special{pa 3974 406}%
\special{fp}%
\special{pa 4080 274}%
\special{pa 3974 372}%
\special{fp}%
\special{pa 4042 274}%
\special{pa 3974 338}%
\special{fp}%
\special{pa 4006 274}%
\special{pa 3974 304}%
\special{fp}%
\put(39.7800,-15.8400){\makebox(0,0){$\Lambda$}}%
\put(30.2500,-1.7300){\makebox(0,0){$u$}}%
\put(34.4000,-4.9000){\makebox(0,0){$u_{2,\lambda}$}}%
%
\special{pn 20}%
\special{pa 3676 1476}%
\special{pa 3670 1444}%
\special{pa 3662 1412}%
\special{pa 3654 1382}%
\special{pa 3642 1352}%
\special{pa 3630 1324}%
\special{pa 3616 1298}%
\special{pa 3600 1272}%
\special{pa 3580 1250}%
\special{pa 3558 1230}%
\special{pa 3532 1210}%
\special{pa 3504 1190}%
\special{pa 3472 1172}%
\special{pa 3438 1154}%
\special{pa 3402 1134}%
\special{pa 3368 1116}%
\special{pa 3338 1096}%
\special{pa 3320 1074}%
\special{pa 3318 1050}%
\special{pa 3332 1024}%
\special{pa 3356 998}%
\special{pa 3384 978}%
\special{pa 3410 960}%
\special{pa 3438 946}%
\special{pa 3464 930}%
\special{pa 3490 912}%
\special{pa 3514 890}%
\special{pa 3534 862}%
\special{pa 3548 832}%
\special{pa 3546 804}%
\special{pa 3530 780}%
\special{pa 3504 760}%
\special{pa 3440 720}%
\special{pa 3410 700}%
\special{pa 3380 682}%
\special{pa 3328 642}%
\special{pa 3304 620}%
\special{pa 3280 600}%
\special{pa 3258 578}%
\special{pa 3238 554}%
\special{pa 3220 532}%
\special{pa 3202 508}%
\special{pa 3184 482}%
\special{pa 3152 428}%
\special{pa 3136 400}%
\special{pa 3108 344}%
\special{pa 3080 284}%
\special{pa 3070 258}%
\special{fp}%
\end{picture}}%
  \caption{Suggested bifurcation diagram for $(P_\lambda)$ when $a$ changes sign, $b\leq 0$, $\int_\Omega m > 0 > \int_\Omega a$, and $\int_{\partial
\Omega}b < -\tilde{K}_1(m,a)$.}
\label{figbnonposi:large} 
 \end{figure}

\begin{figure}[H]
      \begin{center}
      \subfigure[Bifurcation diagram with a single turning point.]{
{\unitlength 0.1in
\begin{picture}( 23.6200, 18.2400)( 22.9000,-26.8000)
%
\special{pn 8}%
\special{pa 2290 2334}%
\special{pa 4652 2334}%
\special{fp}%
\special{sh 1}%
\special{pa 4652 2334}%
\special{pa 4586 2314}%
\special{pa 4600 2334}%
\special{pa 4586 2354}%
\special{pa 4652 2334}%
\special{fp}%
%
\special{pn 8}%
\special{pa 3232 2680}%
\special{pa 3232 1016}%
\special{fp}%
\special{sh 1}%
\special{pa 3232 1016}%
\special{pa 3212 1084}%
\special{pa 3232 1070}%
\special{pa 3252 1084}%
\special{pa 3232 1016}%
\special{fp}%
\put(31.6000,-24.1200){\makebox(0,0){O}}%
\put(49.2800,-23.3300){\makebox(0,0){$\lambda$}}%
%
\special{pn 8}%
\special{pa 4160 2334}%
\special{pa 4160 1046}%
\special{pa 4608 1046}%
\special{pa 4608 2334}%
\special{pa 4160 2334}%
\special{ip}%
%
\special{pn 4}%
\special{pa 4608 1798}%
\special{pa 4160 2244}%
\special{fp}%
\special{pa 4608 1834}%
\special{pa 4160 2278}%
\special{fp}%
\special{pa 4608 1872}%
\special{pa 4160 2316}%
\special{fp}%
\special{pa 4608 1906}%
\special{pa 4178 2334}%
\special{fp}%
\special{pa 4608 1946}%
\special{pa 4214 2334}%
\special{fp}%
\special{pa 4608 1982}%
\special{pa 4250 2334}%
\special{fp}%
\special{pa 4608 2016}%
\special{pa 4286 2334}%
\special{fp}%
\special{pa 4608 2054}%
\special{pa 4326 2334}%
\special{fp}%
\special{pa 4608 2090}%
\special{pa 4360 2334}%
\special{fp}%
\special{pa 4608 2126}%
\special{pa 4400 2334}%
\special{fp}%
\special{pa 4608 2162}%
\special{pa 4436 2334}%
\special{fp}%
\special{pa 4608 2200}%
\special{pa 4472 2334}%
\special{fp}%
\special{pa 4608 2236}%
\special{pa 4512 2334}%
\special{fp}%
\special{pa 4608 2272}%
\special{pa 4548 2334}%
\special{fp}%
\special{pa 4608 2308}%
\special{pa 4582 2334}%
\special{fp}%
\special{pa 4608 1762}%
\special{pa 4160 2206}%
\special{fp}%
\special{pa 4608 1726}%
\special{pa 4160 2170}%
\special{fp}%
\special{pa 4608 1688}%
\special{pa 4160 2132}%
\special{fp}%
\special{pa 4608 1652}%
\special{pa 4160 2096}%
\special{fp}%
\special{pa 4608 1616}%
\special{pa 4160 2060}%
\special{fp}%
\special{pa 4608 1580}%
\special{pa 4160 2024}%
\special{fp}%
\special{pa 4608 1544}%
\special{pa 4160 1988}%
\special{fp}%
\special{pa 4608 1506}%
\special{pa 4160 1952}%
\special{fp}%
\special{pa 4608 1470}%
\special{pa 4160 1914}%
\special{fp}%
\special{pa 4608 1434}%
\special{pa 4160 1876}%
\special{fp}%
\special{pa 4608 1396}%
\special{pa 4160 1840}%
\special{fp}%
\special{pa 4608 1362}%
\special{pa 4160 1804}%
\special{fp}%
\special{pa 4608 1324}%
\special{pa 4160 1768}%
\special{fp}%
\special{pa 4608 1288}%
\special{pa 4160 1732}%
\special{fp}%
\special{pa 4608 1252}%
\special{pa 4160 1694}%
\special{fp}%
\special{pa 4608 1214}%
\special{pa 4160 1658}%
\special{fp}%
\special{pa 4608 1180}%
\special{pa 4160 1622}%
\special{fp}%
\special{pa 4608 1142}%
\special{pa 4160 1586}%
\special{fp}%
\special{pa 4608 1106}%
\special{pa 4160 1550}%
\special{fp}%
\special{pa 4608 1068}%
\special{pa 4160 1512}%
\special{fp}%
\special{pa 4596 1046}%
\special{pa 4160 1476}%
\special{fp}%
\special{pa 4558 1046}%
\special{pa 4160 1440}%
\special{fp}%
\special{pa 4522 1046}%
\special{pa 4160 1406}%
\special{fp}%
\special{pa 4486 1046}%
\special{pa 4160 1368}%
\special{fp}%
\special{pa 4448 1046}%
\special{pa 4160 1332}%
\special{fp}%
\special{pa 4412 1046}%
\special{pa 4160 1294}%
\special{fp}%
\special{pa 4374 1046}%
\special{pa 4160 1260}%
\special{fp}%
\special{pa 4338 1046}%
\special{pa 4160 1220}%
\special{fp}%
\special{pa 4302 1046}%
\special{pa 4160 1184}%
\special{fp}%
\special{pa 4264 1046}%
\special{pa 4160 1148}%
\special{fp}%
\special{pa 4226 1046}%
\special{pa 4160 1112}%
\special{fp}%
\special{pa 4190 1046}%
\special{pa 4160 1074}%
\special{fp}%
\put(41.6400,-24.4900){\makebox(0,0){$\Lambda$}}%
\put(32.3200,-9.3600){\makebox(0,0){$u$}}%
\put(35.4100,-12.9300){\makebox(0,0){$u_{2,\lambda}$}}%
%
\special{pn 20}%
\special{pa 3720 2334}%
\special{pa 3740 2310}%
\special{pa 3758 2284}%
\special{pa 3800 2236}%
\special{pa 3822 2212}%
\special{pa 3846 2188}%
\special{pa 3870 2162}%
\special{pa 3892 2138}%
\special{pa 3912 2112}%
\special{pa 3926 2084}%
\special{pa 3936 2056}%
\special{pa 3938 2026}%
\special{pa 3934 1996}%
\special{pa 3922 1966}%
\special{pa 3908 1934}%
\special{pa 3868 1878}%
\special{pa 3844 1852}%
\special{pa 3820 1828}%
\special{pa 3796 1808}%
\special{pa 3770 1788}%
\special{pa 3746 1768}%
\special{pa 3718 1750}%
\special{pa 3692 1734}%
\special{pa 3640 1698}%
\special{pa 3616 1680}%
\special{pa 3590 1660}%
\special{pa 3568 1638}%
\special{pa 3544 1616}%
\special{pa 3500 1568}%
\special{pa 3480 1544}%
\special{pa 3460 1518}%
\special{pa 3424 1464}%
\special{pa 3408 1436}%
\special{pa 3394 1406}%
\special{pa 3380 1378}%
\special{pa 3366 1348}%
\special{pa 3344 1288}%
\special{pa 3324 1228}%
\special{pa 3300 1134}%
\special{pa 3294 1102}%
\special{pa 3286 1072}%
\special{pa 3280 1040}%
\special{pa 3276 1026}%
\special{fp}%
\end{picture}}%
          \label{fig:turningpt}
      }
      \hfill
      \subfigure[Bifurcation diagram without a turning point.]{
{\unitlength 0.1in
\begin{picture}( 24.9100, 19.1200)( 23.8000,-26.6000)
%
\special{pn 8}%
\special{pa 2380 2296}%
\special{pa 4872 2296}%
\special{fp}%
\special{sh 1}%
\special{pa 4872 2296}%
\special{pa 4804 2276}%
\special{pa 4818 2296}%
\special{pa 4804 2316}%
\special{pa 4872 2296}%
\special{fp}%
%
\special{pn 8}%
\special{pa 3376 2660}%
\special{pa 3376 910}%
\special{fp}%
\special{sh 1}%
\special{pa 3376 910}%
\special{pa 3356 978}%
\special{pa 3376 964}%
\special{pa 3396 978}%
\special{pa 3376 910}%
\special{fp}%
\put(32.9600,-23.8000){\makebox(0,0){O}}%
\put(51.6400,-22.9600){\makebox(0,0){$\lambda$}}%
%
\special{pn 8}%
\special{pa 4350 2296}%
\special{pa 4350 942}%
\special{pa 4826 942}%
\special{pa 4826 2296}%
\special{pa 4350 2296}%
\special{ip}%
%
\special{pn 4}%
\special{pa 4826 1734}%
\special{pa 4350 2200}%
\special{fp}%
\special{pa 4826 1774}%
\special{pa 4350 2238}%
\special{fp}%
\special{pa 4826 1812}%
\special{pa 4350 2276}%
\special{fp}%
\special{pa 4826 1850}%
\special{pa 4370 2296}%
\special{fp}%
\special{pa 4826 1888}%
\special{pa 4408 2296}%
\special{fp}%
\special{pa 4826 1928}%
\special{pa 4448 2296}%
\special{fp}%
\special{pa 4826 1964}%
\special{pa 4488 2296}%
\special{fp}%
\special{pa 4826 2002}%
\special{pa 4526 2296}%
\special{fp}%
\special{pa 4826 2042}%
\special{pa 4566 2296}%
\special{fp}%
\special{pa 4826 2080}%
\special{pa 4606 2296}%
\special{fp}%
\special{pa 4826 2118}%
\special{pa 4644 2296}%
\special{fp}%
\special{pa 4826 2156}%
\special{pa 4682 2296}%
\special{fp}%
\special{pa 4826 2194}%
\special{pa 4720 2296}%
\special{fp}%
\special{pa 4826 2232}%
\special{pa 4760 2296}%
\special{fp}%
\special{pa 4826 2272}%
\special{pa 4800 2296}%
\special{fp}%
\special{pa 4826 1698}%
\special{pa 4350 2162}%
\special{fp}%
\special{pa 4826 1658}%
\special{pa 4350 2124}%
\special{fp}%
\special{pa 4826 1620}%
\special{pa 4350 2086}%
\special{fp}%
\special{pa 4826 1582}%
\special{pa 4350 2046}%
\special{fp}%
\special{pa 4826 1544}%
\special{pa 4350 2008}%
\special{fp}%
\special{pa 4826 1504}%
\special{pa 4350 1972}%
\special{fp}%
\special{pa 4826 1468}%
\special{pa 4350 1932}%
\special{fp}%
\special{pa 4826 1428}%
\special{pa 4350 1896}%
\special{fp}%
\special{pa 4826 1390}%
\special{pa 4350 1856}%
\special{fp}%
\special{pa 4826 1352}%
\special{pa 4350 1818}%
\special{fp}%
\special{pa 4826 1312}%
\special{pa 4350 1778}%
\special{fp}%
\special{pa 4826 1274}%
\special{pa 4350 1742}%
\special{fp}%
\special{pa 4826 1236}%
\special{pa 4350 1702}%
\special{fp}%
\special{pa 4826 1200}%
\special{pa 4350 1666}%
\special{fp}%
\special{pa 4826 1160}%
\special{pa 4350 1626}%
\special{fp}%
\special{pa 4826 1122}%
\special{pa 4350 1588}%
\special{fp}%
\special{pa 4826 1082}%
\special{pa 4350 1548}%
\special{fp}%
\special{pa 4826 1046}%
\special{pa 4350 1512}%
\special{fp}%
\special{pa 4826 1006}%
\special{pa 4350 1472}%
\special{fp}%
\special{pa 4826 970}%
\special{pa 4350 1436}%
\special{fp}%
\special{pa 4812 942}%
\special{pa 4350 1396}%
\special{fp}%
\special{pa 4772 942}%
\special{pa 4350 1358}%
\special{fp}%
\special{pa 4734 942}%
\special{pa 4350 1320}%
\special{fp}%
\special{pa 4696 942}%
\special{pa 4350 1282}%
\special{fp}%
\special{pa 4656 942}%
\special{pa 4350 1244}%
\special{fp}%
\special{pa 4616 942}%
\special{pa 4350 1206}%
\special{fp}%
\special{pa 4580 942}%
\special{pa 4350 1168}%
\special{fp}%
\special{pa 4540 942}%
\special{pa 4350 1128}%
\special{fp}%
\special{pa 4500 942}%
\special{pa 4350 1090}%
\special{fp}%
\special{pa 4462 942}%
\special{pa 4350 1052}%
\special{fp}%
\special{pa 4422 942}%
\special{pa 4350 1014}%
\special{fp}%
\special{pa 4384 942}%
\special{pa 4350 976}%
\special{fp}%
\put(43.5700,-24.1700){\makebox(0,0){$\Lambda$}}%
\put(33.7500,-8.2800){\makebox(0,0){$u$}}%
\put(37.0100,-12.0500){\makebox(0,0){$u_{2,\lambda}$}}%
%
\special{pn 20}%
\special{pa 4020 2296}%
\special{pa 4012 2266}%
\special{pa 3996 2202}%
\special{pa 3988 2172}%
\special{pa 3978 2142}%
\special{pa 3966 2112}%
\special{pa 3954 2084}%
\special{pa 3924 2028}%
\special{pa 3906 2000}%
\special{pa 3834 1896}%
\special{pa 3814 1870}%
\special{pa 3776 1818}%
\special{pa 3756 1790}%
\special{pa 3720 1738}%
\special{pa 3702 1710}%
\special{pa 3654 1626}%
\special{pa 3612 1540}%
\special{pa 3600 1510}%
\special{pa 3586 1480}%
\special{pa 3574 1452}%
\special{pa 3552 1392}%
\special{pa 3542 1360}%
\special{pa 3532 1330}%
\special{pa 3512 1268}%
\special{pa 3502 1238}%
\special{pa 3494 1208}%
\special{pa 3478 1146}%
\special{pa 3468 1114}%
\special{pa 3462 1084}%
\special{pa 3446 1022}%
\special{pa 3438 990}%
\special{pa 3432 958}%
\special{pa 3424 928}%
\special{pa 3422 918}%
\special{fp}%
\end{picture}}%
        \label{fig:noturningpt}
      }
      \end{center}
      \caption{Suggested bifurcation diagram for $(P_\lambda)$ when $a$ changes sign, $b\leq 0$, $\int_\Omega m > 0 > \int_\Omega a$, and $\int_{\partial
\Omega}b \to -\infty$.}
      \label{figbnonposi:enoughlarge}
    \end{figure} 
	
\begin{figure}[H]
{\unitlength 0.1in
\begin{picture}( 28.3200, 19.3400)( 23.4500,-26.0000)
%
\special{pn 8}%
\special{pa 2776 2232}%
\special{pa 5178 2232}%
\special{fp}%
\special{sh 1}%
\special{pa 5178 2232}%
\special{pa 5110 2212}%
\special{pa 5124 2232}%
\special{pa 5110 2252}%
\special{pa 5178 2232}%
\special{fp}%
%
\special{pn 8}%
\special{pa 3736 2600}%
\special{pa 3736 830}%
\special{fp}%
\special{sh 1}%
\special{pa 3736 830}%
\special{pa 3716 898}%
\special{pa 3736 884}%
\special{pa 3756 898}%
\special{pa 3736 830}%
\special{fp}%
\put(36.6000,-23.1600){\makebox(0,0){O}}%
\put(41.7300,-21.3400){\makebox(0,0){$u_{0,\lambda}$}}%
\put(54.6000,-22.3200){\makebox(0,0){$\lambda$}}%
%
\special{pn 8}%
\special{pa 4676 2232}%
\special{pa 4676 862}%
\special{pa 5134 862}%
\special{pa 5134 2232}%
\special{pa 4676 2232}%
\special{ip}%
%
\special{pn 4}%
\special{pa 5134 1664}%
\special{pa 4676 2136}%
\special{fp}%
\special{pa 5134 1702}%
\special{pa 4676 2174}%
\special{fp}%
\special{pa 5134 1742}%
\special{pa 4676 2212}%
\special{fp}%
\special{pa 5134 1780}%
\special{pa 4696 2232}%
\special{fp}%
\special{pa 5134 1820}%
\special{pa 4734 2232}%
\special{fp}%
\special{pa 5134 1858}%
\special{pa 4770 2232}%
\special{fp}%
\special{pa 5134 1896}%
\special{pa 4808 2232}%
\special{fp}%
\special{pa 5134 1936}%
\special{pa 4846 2232}%
\special{fp}%
\special{pa 5134 1974}%
\special{pa 4884 2232}%
\special{fp}%
\special{pa 5134 2014}%
\special{pa 4922 2232}%
\special{fp}%
\special{pa 5134 2052}%
\special{pa 4958 2232}%
\special{fp}%
\special{pa 5134 2090}%
\special{pa 4996 2232}%
\special{fp}%
\special{pa 5134 2130}%
\special{pa 5034 2232}%
\special{fp}%
\special{pa 5134 2168}%
\special{pa 5072 2232}%
\special{fp}%
\special{pa 5134 2208}%
\special{pa 5108 2232}%
\special{fp}%
\special{pa 5134 1626}%
\special{pa 4676 2096}%
\special{fp}%
\special{pa 5134 1586}%
\special{pa 4676 2058}%
\special{fp}%
\special{pa 5134 1548}%
\special{pa 4676 2018}%
\special{fp}%
\special{pa 5134 1510}%
\special{pa 4676 1980}%
\special{fp}%
\special{pa 5134 1470}%
\special{pa 4676 1942}%
\special{fp}%
\special{pa 5134 1432}%
\special{pa 4676 1902}%
\special{fp}%
\special{pa 5134 1392}%
\special{pa 4676 1864}%
\special{fp}%
\special{pa 5134 1354}%
\special{pa 4676 1826}%
\special{fp}%
\special{pa 5134 1316}%
\special{pa 4676 1786}%
\special{fp}%
\special{pa 5134 1276}%
\special{pa 4676 1748}%
\special{fp}%
\special{pa 5134 1238}%
\special{pa 4676 1708}%
\special{fp}%
\special{pa 5134 1198}%
\special{pa 4676 1670}%
\special{fp}%
\special{pa 5134 1160}%
\special{pa 4676 1632}%
\special{fp}%
\special{pa 5134 1122}%
\special{pa 4676 1592}%
\special{fp}%
\special{pa 5134 1082}%
\special{pa 4676 1554}%
\special{fp}%
\special{pa 5134 1044}%
\special{pa 4676 1516}%
\special{fp}%
\special{pa 5134 1004}%
\special{pa 4676 1476}%
\special{fp}%
\special{pa 5134 966}%
\special{pa 4676 1438}%
\special{fp}%
\special{pa 5134 928}%
\special{pa 4676 1398}%
\special{fp}%
\special{pa 5134 888}%
\special{pa 4676 1360}%
\special{fp}%
\special{pa 5122 862}%
\special{pa 4676 1322}%
\special{fp}%
\special{pa 5084 862}%
\special{pa 4676 1282}%
\special{fp}%
\special{pa 5048 862}%
\special{pa 4676 1244}%
\special{fp}%
\special{pa 5008 862}%
\special{pa 4676 1206}%
\special{fp}%
\special{pa 4972 862}%
\special{pa 4676 1166}%
\special{fp}%
\special{pa 4934 862}%
\special{pa 4676 1128}%
\special{fp}%
\special{pa 4896 862}%
\special{pa 4676 1088}%
\special{fp}%
\special{pa 4858 862}%
\special{pa 4676 1050}%
\special{fp}%
\special{pa 4820 862}%
\special{pa 4676 1012}%
\special{fp}%
\special{pa 4784 862}%
\special{pa 4676 972}%
\special{fp}%
\special{pa 4744 862}%
\special{pa 4676 934}%
\special{fp}%
\special{pa 4708 862}%
\special{pa 4676 896}%
\special{fp}%
\put(46.8200,-23.5400){\makebox(0,0){$\Lambda$}}%
\put(37.3500,-7.4600){\makebox(0,0){$u$}}%
%
\special{pn 20}%
\special{pa 3736 2232}%
\special{pa 3786 2192}%
\special{pa 3812 2172}%
\special{pa 3864 2136}%
\special{pa 3890 2120}%
\special{pa 3948 2092}%
\special{pa 4008 2068}%
\special{pa 4038 2060}%
\special{pa 4102 2044}%
\special{pa 4166 2030}%
\special{pa 4196 2022}%
\special{pa 4226 2008}%
\special{pa 4252 1990}%
\special{pa 4276 1966}%
\special{pa 4288 1936}%
\special{pa 4286 1908}%
\special{pa 4272 1880}%
\special{pa 4248 1854}%
\special{pa 4218 1832}%
\special{pa 4188 1814}%
\special{pa 4156 1800}%
\special{pa 4124 1790}%
\special{pa 4092 1784}%
\special{pa 4060 1780}%
\special{pa 4000 1772}%
\special{pa 3940 1760}%
\special{pa 3910 1752}%
\special{pa 3880 1742}%
\special{pa 3850 1730}%
\special{pa 3820 1716}%
\special{pa 3786 1700}%
\special{pa 3758 1682}%
\special{pa 3738 1658}%
\special{pa 3736 1630}%
\special{pa 3748 1598}%
\special{pa 3766 1568}%
\special{pa 3786 1540}%
\special{pa 3808 1516}%
\special{pa 3832 1496}%
\special{pa 3858 1476}%
\special{pa 3910 1444}%
\special{pa 3936 1430}%
\special{pa 4020 1394}%
\special{pa 4080 1368}%
\special{pa 4112 1354}%
\special{pa 4144 1338}%
\special{pa 4178 1318}%
\special{pa 4206 1298}%
\special{pa 4230 1274}%
\special{pa 4242 1248}%
\special{pa 4242 1222}%
\special{pa 4228 1194}%
\special{pa 4206 1166}%
\special{pa 4180 1142}%
\special{pa 4152 1122}%
\special{pa 4124 1104}%
\special{pa 4068 1076}%
\special{pa 4038 1064}%
\special{pa 4012 1050}%
\special{pa 3984 1034}%
\special{pa 3958 1018}%
\special{pa 3910 978}%
\special{pa 3886 956}%
\special{pa 3864 934}%
\special{pa 3842 910}%
\special{pa 3798 860}%
\special{pa 3776 834}%
\special{pa 3768 824}%
\special{fp}%
\put(31.4500,-16.2100){\makebox(0,0){$c_0 = (\lambda(0), u(0))$}}%
\put(40.3600,-9.0500){\makebox(0,0){$u_{2,\lambda}$}}%
\put(41.5600,-16.0300){\makebox(0,0){$(\lambda(t), u(t))$}}%
\end{picture}}%
  \caption{Suggested bifurcation diagram for $(Q_\lambda)$ when $m$ and $b$ change sign and $\int_\Omega m > 0 > \int_{\partial \Omega}b = -\tilde{K}_1(m,-m)$.} \label{fig:a=-m} 
    \end{figure}
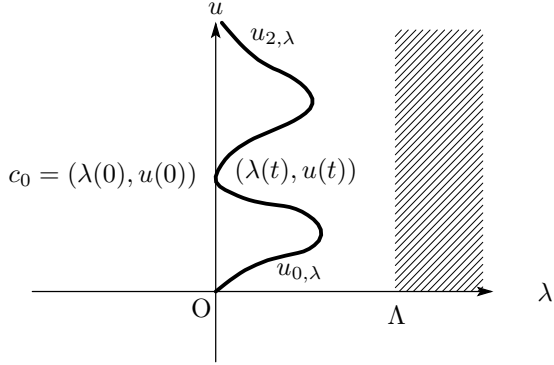 


The rest of this article is organized as follows: in Section \ref{sec2} we collect some preliminary results on $\lambda_a$, $\lambda_b$, $\lambda_s$ and $\mathcal{N}_\lambda$, the Nehari manifold associated to $(P_\lambda)$. In Section \ref{sec3}, we obtain two minimizers for $I_\lambda$ (the energy functional associated to $(P_\lambda)$) constrained to $\mathcal{N}_\lambda$. These minimizers are located in $\mathcal{N}_\lambda^+$ and correspond to local minimizers of $I_\lambda$. In Section \ref{sec4}, we obtain a further minimizer of $I_\lambda$ constrained to $\mathcal{N}_\lambda^-$, which corresponds to a minimax critical point of $I_\lambda$. Finally, in Section \ref{sec5}, we carry out a bifurcation analysis of $(P_\lambda)$ and prove the theorems stated above.


\section{Preliminaries} \label{sec2} 
Solutions of $(P_\lambda)$ shall be obtained as critical points of the functional
$$I_{\lambda}(u)=\frac{1}{2}E_{\lambda}(u) -\frac{\lambda}{p}A(u) -\frac{\lambda}{q}B(u),$$
defined on $H^1(\Omega)$. 
It is straightforward that $I_{\lambda}$ is weakly lower semicontinuous for any $\lambda \in \R$. In contrast with the case $a < 0$, for any value of $\lambda$ the functional $I_\lambda$ is not coercive. Consequently we shall restrict $I_\lambda$ to its associated Nehari manifold, which is given by
\begin{eqnarray*}
\mathcal{N}_\lambda &=&\{u \in H^1(\Omega)\setminus \{0\};\, \langle I'_{\lambda}(u), u \rangle=0\}\\
&=&\left\{u \in H^1(\Omega)\setminus \{0\}; \, E_\lambda(u)=\lambda \left[A(u)+B(u)\right]\right\}.
\end{eqnarray*}
$\mathcal{N}_\lambda$ is naturally split into $\mathcal{N}_\lambda^+$, $\mathcal{N}_\lambda^-$ and $\mathcal{N}_\lambda^0$, given by
$$\mathcal{N}_{\lambda}^{\pm}=\left\{u\in \mathcal{N}_\lambda; \,
\langle J'_\lambda(u), u \rangle \gtrless 0\right\} \quad \text{and} \quad 
\mathcal{N}_\lambda^0=\left\{u\in \mathcal{N}_\lambda; \,
\langle J'_\lambda(u), u \rangle = 0\right\},$$
where $J_\lambda(u)=\langle I'_\lambda(u),u \rangle$ for $u \in H^1(\Omega)$.
Thus
$$
\mathcal{N}_{\lambda}^{\pm}
=\left\{u\in \mathcal{N}_\lambda; \,
E_\lambda(u)\lessgtr \lambda \frac{p-q}{p-2}B(u)\right\}
=\left\{u\in \mathcal{N}_\lambda; \,
E_\lambda(u)\gtrless \lambda \frac{p-q}{2-q}A(u)\right\}.
$$
and
$$\mathcal{N}_\lambda^0
=\left\{u\in \mathcal{N}_\lambda; \,
E_\lambda(u)= \lambda \frac{p-q}{p-2}B(u)=\lambda \frac{p-q}{2-q}A(u)\right\}.$$

It is well-known that $\mathcal{N}_\lambda \setminus \mathcal{N}_\lambda^0$ is a $\mathcal{C}^1$ manifold defined by a natural constraint, i.e. 
any critical point of the restriction of $I_{\lambda}$ to $\mathcal{N}_\lambda \setminus \mathcal{N}_\lambda^0$ is a critical point of $I_{\lambda}$ (see for instance \cite[Theorem 2.3]{BZ03}).
Given $u\in H^1(\Omega) \setminus \{0\}$, we set $$j_u(t)=I_{\lambda}(tu), \quad t>0.$$
Then $tu \in \mathcal{N}_\lambda^{\pm}$ if, and only if,  $j'_u(t)=0$ and  $j''_u(t) \gtrless 0$. We shall use the map $j_u$ to deduce some properties of $\mathcal{N}_\lambda$. 
Note that $j_u(t)=t^qi_u(t)$,
where 
\begin{equation}
\label{eqi}
i_u(t)=\frac{1}{2}t^{2-q} E_\lambda(u)-\frac{\lambda}{p}t^{p-q}A(u)-\frac{\lambda}{q}B(u)
\end{equation}
for $t>0$.

\begin{rem}
\label{cn}{\rm 
It is easily seen that $c$ is a positive zero of $\varphi$ if and only if $c \in \mathcal{N}_\lambda$, for any $\lambda \in \R$. Additionally, if $\varphi'(c) \gtrless 0$ then $c \in \mathcal{N}_\lambda^{\pm}$. More precisely:
\begin{enumerate}
\item  If $\int_{\Omega} m> 0>\int_{\Omega} a$ and $-K_1(m,a)<\int_{\partial \Omega} b<0$ then $c_1 \in \mathcal{N}_\lambda^-$ and $c_2 \in \mathcal{N}_\lambda^+$.
\item If $\int_{\Omega} m< 0<\int_{\Omega} a$ and $0<\int_{\partial \Omega} b<K_1(-m,-a)$ then $c_1 \in \mathcal{N}_\lambda^+$ and $c_2 \in \mathcal{N}_\lambda^-$.
\item If $\int_{\Omega} a> 0> \int_{\partial \Omega} b$  then $c_0 \in \mathcal{N}_\lambda^-$.
\item If $\int_{\Omega} a< 0< \int_{\partial \Omega} b$ then $c_0 \in \mathcal{N}_\lambda^+$.
\end{enumerate}
}\end{rem}

The role of $\lambda_s$ in the study of $\mathcal{N}_\lambda$ becomes clear in the next result:

\begin{prop}
\label{p1}
If either $\int_{\Omega} a<0$ or $\int_{\partial \Omega} b<0$ or
\begin{equation}
\label{n1}
\int_{\Omega} m<0<\int_{\Omega} a \quad \text{and} \quad 0<\int_{\partial \Omega} b<K_1(-m,-a).
\end{equation}
then $\lambda_s$, given by \eqref{els}, is achieved. In particular, $\lambda_s>0$.
Furthermore, if $\lambda \in (0,\lambda_s)$ then for every $u \in B^+ \cap E_{\lambda}^+ \cap A^+$ there are $0<t_1<t_2$ such that $t_1 u \in \mathcal{N}_\lambda^+$ and $t_2 u \in \mathcal{N}_\lambda^-$, and there is no other $t>0$ such that $tu \in \mathcal{N}_\lambda$.
\end{prop}

\begin{proof}
Let $(u_n)$ be a minimizing sequence for $\lambda_s$, i.e.
$$A(u_n) \geq 0, \quad  B(u_n)\geq 0, \quad S(u_n)=1 \quad \text{and} \quad \int_{\Omega} |\nabla u_n|^2 \rightarrow \lambda_s.$$
If $(u_n)$ is unbounded then we set $v_n=\frac{u_n}{\|u_n\|}$. We may assume that 
$$v_n \rightharpoonup v_0, \quad A(v_n) \rightarrow A(v_0), \quad \text{and} \quad B(v_n) \rightarrow B(v_0).$$
Thus 
$$A(v_0) \geq 0, \quad B(v_0) \geq 0, \quad S(v_0)=0, \quad \text{and} \quad \int_{\Omega} |\nabla v_n|^2 \rightarrow 0.$$ Hence $v_0$ is a nonzero constant, so that 
$$\int_{\Omega} a\geq 0, \quad \int_{\partial \Omega} b\geq 0 \quad \text{and} \quad \int_{\Omega} m +\left(C_{pq}^{-1} \left(\int_{\partial \Omega} b\right) \left(\int_{\Omega} a\right)^{\frac{2-q}{p-2}}\right)^{\frac{p-2}{p-q}}=0,$$
which contradicts our assumption.
Thus $(u_n)$ is bounded and we may assume that $u_n \rightharpoonup u_0$. If $\lambda_s=0$ then $u_0$ is a nonzero constant, and from $A(u_0)\geq 0$, $B(u_0)\geq 0$ and $S(u_0)=1$ we get
$$\int_{\Omega} a\geq 0, \quad \int_{\partial \Omega} b\geq 0 \quad \text{and} \quad \int_{\Omega} m +\left(C_{pq}^{-1} \left(\int_{\partial \Omega} b\right) \left(\int_{\Omega} a\right)^{\frac{2-q}{p-2}}\right)^{\frac{p-2}{p-q}}>0,$$
which contradicts again our assumption.
Therefore $\lambda_s=\int_{\Omega} |\nabla u_0|^2>0$.

Now, let $\lambda \in (0,\lambda_s)$ and $u \in A^+ \cap B^+ \cap E_{\lambda}^+$. Then $j_u$ has two critical points if $j_u(t)>0$ for some $t>0$. In this case,  $j_u$ has a local minimum and a global maximum, i.e. there are $t_1<t_2$ such that $t_1 u \in \mathcal{N}_{\lambda}^+$ and $t_2 u \in \mathcal{N}_{\lambda}^-$. Note that $j_u(t)>0$ if and only if $i_u(t)>0$, where $i_u$ is given by \eqref{eqi}. One may easily check that $i_u$ has a global maximum point given by 
\begin{equation}
\label{eqto}
t_0(u)=\left(\frac{p(2-q)}{2(p-q)} \frac{E_\lambda(u)}{\lambda A(u)}\right)^{\frac{1}{p-2}}.
\end{equation}
We have
\begin{eqnarray*}
i_u(t_0(u))>0 &\Leftrightarrow&
\lambda B(u)<C_{pq} \frac{E_\lambda(u)^{\frac{p-q}{p-2}}}{\left(\lambda A(u)\right)^{\frac{2-q}{p-2}}} \\ 
&\Leftrightarrow &
B(u)<C_{pq} \frac{\left(\lambda^{-1}\int_{\Omega} |\nabla u|^2 -\int_{\Omega} mu^2\right)^{\frac{p-q}{p-2}}}{A(u)^{\frac{2-q}{p-2}}}\\
&\Leftrightarrow &
E_{\lambda}(u)-\lambda \left(C_{pq}^{-1} B(u)A(u)^{\frac{2-q}{p-2}}\right)^{\frac{p-2}{p-q}}>0\\
&\Leftrightarrow & \int_{\Omega} |\nabla u|^2 -\lambda S(u)>0,
\end{eqnarray*}
where $S(u)$ is defined in \eqref{eqs}. Since $$\lambda <\lambda_s=\inf \left\{ \int_{\Omega} |\nabla u|^2;\ u \in A_0^+ \cap B_0^+, \ S(u)=1\right\},$$ we get the existence of $t_1<t_2$.
Finally, from the expression of $j_u$ it is easily seen that $j_u$ can not have more than two critical points.
\end{proof}

We shall get now some kind of local coercivity for $E_\lambda$. More precisely, we shall prove that $E_\lambda$ is coercive on $A_0^+$  (respect. $B_0^+$) for $\lambda<\lambda_a$ (respect. $\lambda<\lambda_b$). To this end, we deal with the maps $\alpha_1,\beta_1:\R \mapsto \R$ given by

\begin{equation}
\alpha_1(\lambda)=\inf\{E_\lambda(u); \ u \in A_0^+,\ \|u\|_2=1\},
\end{equation}

\begin{equation}
\beta_1(\lambda)=\inf\{E_\lambda(u); \ u \in B_0^+,\ \|u\|_2=1\}.
\end{equation}
 
Note that if $\alpha_1(\lambda)>0$ (respect. $\beta_1(\lambda)>0$) then $A_0^+ \setminus \{0\} \subset E_\lambda^+$ (respect $B_0^+ \setminus \{0\} \subset E_\lambda^+$).
From \cite{RQU}, we have the following results:

\begin{prop}
\label{pal1}
\strut
\begin{enumerate}
\item $\alpha_1$ and $\beta_1$ are concave (and therefore continuous).
\item $\lambda_b>0$ if and only if either $\int_\Om m<0$ or $\int_{\partial \Om} b<0$. In this case, $\lambda_b$ is achieved and:
\begin{enumerate}
\item $\beta_1(\lambda)>0$ if $\lambda \in (0,\lambda_b)$. Moreover, if $\int_{\partial \Om} b<0$ then $\beta_1(0)>0$.
\item $\lambda_b=\max\{ \lambda>0; \ \beta_1(\lambda)\geq 0\}$.
\end{enumerate} 
\item  Assume $\int_\Om m<0$. Then $\lambda_b>\lambda_1(m)$ if and only if $\int_{\partial \Om} b\varphi_1^q<0$. 
\item $\lambda_a>0$ if and only if either $\int_\Om m<0$ or $\int_{\Om} a<0$. In this case, $\lambda_a$ is achieved and:
\begin{enumerate}
\item $\alpha_1(\lambda)>0$ if $\lambda \in (0,\lambda_a)$. Moreover, if $\int_{\Om} a<0$ then $\alpha_1(0)>0$.
\item $\lambda_a=\max\{ \lambda>0; \ \alpha_1(\lambda)\geq 0\}$.
\end{enumerate}
\item  Assume $\int_\Om m<0$. Then $\lambda_a>\lambda_1(m)$ if and only if $\int_{\Om} a\varphi_1^p<0$. 
\end{enumerate}
\end{prop}

\begin{prop}
\label{p2}
Assume $\int_{\partial \Om} b<0$.
Then for every $\lambda_* \in (0,\lambda_b)$ there exist two constants $C_0=C_0(m,b), D_0=D_0(m,b)>0$ such that:
\begin{enumerate}
\item $E_{\lambda}(u)\geq C_0 \|u\|^2$  for every $\lambda \in (0,\lambda_*)$ and $u \in B_0^+$.
\item $B(u)\leq -D_0 \|u\|^q$ for every $\lambda \in (0,\lambda_*)$ and $u \in E_{\lambda,0}^-$.
\end{enumerate}
\end{prop}

The proof of Proposition \ref{p2} can be easily adapted to obtain a similar result on $a$:

\begin{prop}
\label{p2'}
 Assume $\int_{\Om} a<0$.
Then for every $\lambda_* \in (0,\lambda_a)$ there exist two constants $C_0=C_0(m,a), D_0=D_0(m,a)>0$ such that:
\begin{enumerate}
\item $E_{\lambda}(u)\geq C_0 \|u\|^2$  for every $\lambda \in (0,\lambda_*)$ and $u \in A_0^+$.
\item $A(u)\leq -D_0	\|u\|^p$ for every $\lambda \in (0,\lambda_*)$ and $u \in E_{\lambda,0}^-$.
\end{enumerate}
\end{prop}

Propositions \ref{p2} and \ref{p2'} provide some kind of uniform coercivity (with respect to $\lambda$) to $E_\lambda$ on $A_0^+$ and $B_0^+$, respectively. In case $\lambda_1(m)>0$, we have the following:

\begin{prop}
\label{p2''}
Assume $\int_{\Omega} m<0$.
\begin{enumerate}
\item Given any $\mu<\lambda_1(m)$ there exists a constant $C_\mu>0$ such that $E_\lambda(u)\geq  \lambda C_\mu\|u\|^2$ for every $u \in H^1(\Omega)$ and every $\lambda \in (0,\mu)$.
\item If  $0<\lambda<\lambda_b$ then there exists $C_\lambda>0$ such that $E_\lambda(u)\geq C_\lambda \|u\|^2$ for every $u \in B_0^+$.
\item If $0<\lambda<\lambda_a$ then there exists $C_\lambda>0$ such that $E_\lambda(u)\geq C_\lambda \|u\|^2$ for every $u \in A_0^+$.
\end{enumerate}
\end{prop}

\begin{proof}
\strut
\begin{enumerate}
\item Assume by contradiction that there exist $\mu<\lambda_1(m)$ and two sequences $(\lambda_n) \subset (0,\mu)$ and $(v_n) \subset H^1(\Omega)$ such that
$$E_{\lambda_n}(v_n)<\frac{\lambda_n}{n}\|v_n\|^2.$$
Setting $w_n=\frac{v_n}{\|v_n\|}$, we may assume that $w_n \rightharpoonup w_0$ in $H^1(\Omega)$ and $\lambda_n \to \lambda_* \in [0,\mu]$. Hence $E_{\lambda_*}(w_0)\leq \limsup E_{\lambda_n}(w_n)\leq 0$. If $\lambda_*>0$ then $w_0\equiv 0$ and $w_n \rightarrow 0$, which is impossible. If $\lambda_*=0$ then $w_0$ is a constant. From
$$-\lambda_n \int_{\Omega} mw_n^2<\frac{\lambda_n}{n}$$ we get $\int_{\Omega} mw_0^2\geq 0$, which contradicts $\int_{\Omega} m<0$.\\
\item Assume that $(u_n) \subset B_0^+$ is such that
$$E_\lambda(u_n)<\frac{\|u_n\|^2}{n}.$$
Then $v_n=\frac{u_n}{\|u_n\|}$ is such that, up to a subsequence,
$v_n \rightharpoonup v_0$ in $H^1(\Omega)$, $B(v_n) \rightarrow B(v_0)$, and $E_\lambda(v_n) \leq \frac{1}{n}$. Then $B(v_0)\geq 0$ and $E_\lambda(v_0)\leq 0$. Finally, $v_0 \not \equiv 0$, since otherwise we would have $v_n \rightarrow 0$, which is impossible.
Therefore we get $\beta_1(\lambda)\leq 0$, which contradicts $\lambda<\lambda_b$.\\
\item The proof is similar to the previous one, so we omit it.
\end{enumerate}
\end{proof}

\section{Minimization in $\mathcal{N}_\lambda^+$}
\label{sec3}
\medskip 
\subsection{Minimization in $\mathcal{N}_\lambda^+ \cap B^+$}
\strut
\medskip

We set
\begin{equation}
\label{eov}
\overline{\lambda}=\begin{cases} \min\{\lambda_s,\lambda_1\} & \text{ if }  \eqref{n1} \text{ holds},\\
\min\{\lambda_s,\lambda_b\} & \text{ if } \int_{\partial \Omega} b<0 \text{ or } \int_\Omega m<0 \text{ and } \int_\Omega a<0. \end{cases}
\end{equation}

\begin{rem}
\label{eb}{\rm 
Note that if $0<\lambda <\overline{\lambda}$ then either $0<\lambda<\lambda_1$ or $0<\lambda<\lambda_b$, so that  $B_0^+ \setminus \{0\} 
\subset E_{\lambda}^+$.
}\end{rem}

Let us first prove that $\mathcal{N}_\lambda^+ \cap B^+$ is non-empty and bounded for $0<\lambda<\overline{\lambda}$:

\begin{prop}
\label{cun}
 Assume $b^+ \not \equiv 0$ and either $\int_{\partial \Omega} b<0$ or \eqref{n1} or $ \int_\Omega m<0$ and $ \int_\Omega a<0$. Then $\mathcal{N}_\lambda^+  \cap B^+ \neq \emptyset$  for every $\lambda \in (0,\overline{\lambda})$. Moreover:
\begin{enumerate}
\item If \eqref{n1} holds then for every $\mu<\overline{\lambda}$ there exists a constant $K=K_\mu>0$ such that $\|u\|\leq K\|b^+\|_{\infty}^{\frac{1}{2-q}}$ for every $u \in \mathcal{N}_\lambda^+$ and every $0<\lambda< \mu$.
\item If $\int_{\partial \Omega} b<0$ then for every $\mu<\overline{\lambda}$ there exists a constant $K=K_\mu>0$ such that $\|u\|\leq K\left(\lambda \|b^+\|_{\infty}\right)^{\frac{1}{2-q}}$ for every $u \in \mathcal{N}_\lambda^+ \cap B_0^+$ and every $0<\lambda< \mu$.
\item If $\int_\Omega a<0$ and  $\int_\Omega m<0$ then for every $\lambda<\overline{\lambda}$ there exists a constant $K_\lambda$ such that $\|u\|\leq K_\lambda \|b^+\|_{\infty}^{\frac{1}{2-q}}$ for every $u \in \mathcal{N}_\lambda^+ \cap B_0^+$.
\end{enumerate}

\end{prop}

\begin{proof}
First of all, note that since $b^+ \not \equiv 0$ we have $B^+ \neq \emptyset$. Moreover, from Remark \ref{eb}, if $u \in B^+$ then $u \in E_{\lambda}^+$.
If $u \in A_0^-$ then $j_u$ has a global minimum point $t>0$ such that $tu \in \mathcal{N}_\lambda^+ \cap B^+$. On the other hand, if $u \in A^+$ then, by Proposition \ref{p1}, there exists $t>0$ such that $tu \in \mathcal{N}_\lambda^+ \cap B^+$, so that $\mathcal{N}_\lambda^+ \cap B^+ \neq \emptyset$.\\
\begin{enumerate}
\item If \eqref{n1} holds then, given $0<\mu < \overline{\lambda}$ and $u \in \mathcal{N}_\lambda^+$ with $0<\lambda<\mu$, we may apply Proposition \ref{p2''}. Thus, for some $C_\mu,D>0$, we have
$$\lambda C_\mu \|u\|^2\leq E_\lambda(u)<\lambda \frac{p-q}{p-2} B(u)\leq \lambda D\|b^+\|_{\infty}\|u\|^q,$$
and consequently $$\|u\|\leq \left( \frac{D\|b^+\|_{\infty}}{C_\mu}\right)^{\frac{1}{2-q}}.$$ \\

\item If $\int_{\partial \Omega} b<0$ then, given $0<\mu < \overline{\lambda}$ and $u \in \mathcal{N}_\lambda^+$ with $0<\lambda<\mu$, we apply now Proposition \ref{p2}, so that
$E_\lambda(u)\geq C_0 \|u\|^2$ for every $\lambda \in (0,\mu)$ and $u \in B_0^+$. Thus, for $u \in \mathcal{N}_\lambda^+ \cap B_0^+$ we have $$C_0 \|u\|^2\leq E_\lambda(u)<\lambda \frac{p-q}{p-2} B(u)\leq \lambda D\|b^+\|_{\infty}\|u\|^q,$$ and consequently $$\|u\|\leq \left( \frac{\lambda D\|b^+\|_{\infty}}{C_0}\right)^{\frac{1}{2-q}}.$$\\

\item If $\int_\Omega a<0$ and  $\int_\Omega m<0$ then, by Proposition \ref{p2''}, for every $\lambda<\lambda_b$ there exists a constant $C_\lambda>0$ such that $E_\lambda(u)\geq C_\lambda \|u\|^2$ for $u \in B_0^+$. Thus, if $u \in \mathcal{N}_\lambda^+ \cap B_0^+$ then, for some $D>0$ there holds 
$$ C_\lambda \|u\|^2\leq E_\lambda(u)<\lambda \frac{p-q}{p-2} B(u)\leq \lambda D\|b^+\|_{\infty}\|u\|^q,$$
and consequently $$\|u\|\leq \left( \frac{\lambda D \|b^+\|_{\infty}}{C_\lambda}\right)^{\frac{1}{2-q}}.$$ 
\end{enumerate}
\end{proof}

\begin{rem}
\label{rr}{\rm 
From the proof of Proposition \ref{cun} (1), we may see that if $\int_\Omega m<0$ then  for every $\mu<\lambda_1$ there exists a constant $K=K_\mu>0$ such that $\|u\|\leq K \|b^+\|_{\infty}$ for every $u \in \mathcal{N}_\lambda^+$ and every $0<\lambda< \mu$, i.e. $\mathcal{N}_\lambda^+$ is uniformly bounded for $0<\lambda< \mu$ if $\mu<\lambda_1$. 
}\end{rem}

\begin{lem}
\label{lbb}
Let $b, \tilde{b} \in L^{\infty}(\partial \Omega)$ with $b \leq \tilde{b}$.
Then $\lambda_b\geq \lambda_{\tilde{b}}$ and $\lambda_s(b)\geq \lambda_s(\tilde{b}).$
\end{lem}

\begin{proof}
Let $B(u)=\int_{\partial \Omega} b|u|^q$ and $\tilde{B}(u)=\int_{\partial \Omega} \tilde{b}|u|^q$ for $u \in H^1(\Omega)$. Then $B(u)\leq \tilde{B}(u)$ for every $u \in H^1(\Omega)$, so that $B_0^+ \subset \tilde{B}_0^+$.
It follows that $\lambda_b\geq \lambda_{\tilde{b}}$. 
Let us set \begin{equation}
S_b(u)=\int_{\Omega} mu^2+\left(C_{pq}^{-1} B(u)A(u)^{\frac{2-q}{p-2}}\right)^{\frac{p-2}{p-q}}
\end{equation}
and
\begin{equation}
S_{\tilde{b}}(u)=\int_{\Omega} mu^2+\left(C_{pq}^{-1} \tilde{B}(u)A(u)^{\frac{2-q}{p-2}}\right)^{\frac{p-2}{p-q}}.
\end{equation}
Then $S_b(u)\leq S_{\tilde{b}}(u)$ for every $u \in H^1(\Omega)$.
Note that we can write 
$$
\lambda_s(b)=\inf \left\{ \frac{\int_{\Omega} |\nabla u|^2}{S(u)};\ u \in A_0^+ \cap B_0^+,\ S(u)>0\right\},
$$
From this formula, it follows that $\lambda_s(b)\geq \lambda_s(\tilde{b})$.
\end{proof}

\begin{rem}
\label{rbb}
From Lemma \ref{lbb} it follows that if $b_n \rightarrow b$ in $L^{\infty}(\partial \Omega)$ with $b^+ \not \equiv 0$ and $\int_{\partial \Omega} b <0$ then $\overline{\lambda}(b_n)$ is bounded away from zero. Indeed, we can fix $\tilde{b} \in L^{\infty}(\partial \Omega)$ such that $\tilde{b}^+ \not \equiv 0$, $\int_{\partial \Omega} \tilde{b} <0$ and  $b_n \leq \tilde{b}$ for  $n$ sufficiently large. Hence $\lambda_{b_n} \geq \lambda_{\tilde{b}}>0$ and $\lambda_s(b_n)\geq \lambda_s(\tilde{b})>0$ for $n$ sufficiently large.
\end{rem}

\begin{prop}
\label{pmn1}
 Assume $b^+ \not \equiv 0$ and either $\int_{\partial \Omega} b<0$ or \eqref{n1} or $\int_\Omega m<0$ and $\int_{ \Omega} a<0$. Then $\displaystyle \inf_{\mathcal{N}_\lambda^+ \cap B^+} I_{\lambda}$ is achieved by some $u_{0,\lambda}\geq 0$ for $0<\lambda<\overline{\lambda}$. Moreover:
\begin{enumerate}
\item If $\int_{\partial \Omega} b<0$ then $u_{0,\lambda} \rightarrow 0$  in $\mathcal{C}^{\theta}(\overline{\Om})$ for some $\theta \in (0,1)$ as $\lambda \to 0^+$.
\item If $\int_\Omega m<0<\int_{\partial \Omega} b$ and $\int_\Omega a<0$ then $u_{0,\lambda} \rightarrow c_0$  in $\mathcal{C}^{\theta}(\overline{\Om})$ for some $\theta \in (0,1)$ as $\lambda \to 0^+$.
\item If \eqref{n1} holds then $u_{0,\lambda} \rightarrow c_1$  in $\mathcal{C}^{\theta}(\overline{\Om})$ for some $\theta \in (0,1)$ as $\lambda \to 0^+$.
\end{enumerate}  
\end{prop}

\begin{proof}
Let $0<\lambda<\overline{\lambda}$. By Proposition \ref{cun}
we know that $\mathcal{N}_\lambda^+  \cap B^+$ is non-empty and bounded. We pick up a sequence $(u_n) \subset \mathcal{N}_\lambda^+  \cap B^+$ such that
$$I_{\lambda} (u_n) \rightarrow \inf_{\mathcal{N}_\lambda^+ \cap B^+} I_{\lambda}.$$
Since $(u_n)$ is bounded, we may assume that $$u_n \rightharpoonup u_0 \text{ in } H^1(\Omega)  \quad \text{and} \quad B(u_n) \rightarrow B(u_0).$$ From the shape of $j_u$ it is clear that $\displaystyle \inf_{\mathcal{N}_\lambda^+  \cap B^+} I_{\lambda}<0$. Thus $I_{\lambda}(u_0)<0$, so that $u_0 \not \equiv 0$.
Since $$E_\lambda(u_n)<\lambda \frac{p-q}{p-2}B(u_n)$$  and $u_0 \in B_0^+$ we get
$$0<E_\lambda(u_0)\leq\lambda \frac{p-q}{p-2}B(u_0),$$
i.e. $u_0 \in B^+$. From $\lambda<\lambda_s$ we deduce the existence of $t_1>0$ such that $t_1 u_0 \in \mathcal{N}_\lambda^+ \cap B^+$.
We claim that $u_n \rightarrow u_0$. Indeed, if $u_n \not \rightarrow u_0$ then $$j^{(k)}_{u_0}(t)<\liminf j^{(k)}_{u_n}(t)$$ for $k=0,1,2$ and every $t>0$. In particular, there holds 
$$0=j'_{u_0}(t_1)<j'_{u_n}(t_1)$$ for $n$ sufficiently large.
Thus $t_1>1$, since $j_{u_n}$ is decreasing in $(0,1)$. Now, since $j_{u_0}$ is decreasing in $(0,t_1)$
we get 
$$I_{\lambda}(t_1 u_0)=j_{u_0}(t_1)<j_{u_0}(1)<\liminf j_{u_n}(1)=\lim I_{\lambda}(u_n)=\inf_{\mathcal{N}_\lambda^+ \cap B^+} I_{\lambda},$$
which is a contradiction, since $t_1 u_0 \in \mathcal{N}_\lambda^+ \cap B^+$.
Therefore $u_n \rightarrow u_0$ and $t_1(u_0)=1$, so
$$I_{\lambda}(u_0)=\inf_{\mathcal{N}_\lambda^+ \cap B^+} I_{\lambda}.$$
We denote $u_0$ by $u_{0,\lambda}$.\\
\begin{enumerate}
\item If $\int_{\partial \Omega} b<0$ then, by Prop. \ref{cun}, 
$$\|u_{0,\lambda}\|\leq \left(\lambda\|b^+\|_{\infty}\right)^{\frac{1}{2-q}} K \rightarrow 0$$
as $\lambda \to 0^+$.\\
\item If $\int_\Omega m<0<\int_{\partial \Omega} b$ and $\int_\Omega a<0$, let $\lambda_n \to 0^+$ and $u_n=u_{0,\lambda_n}$. By Remark \ref{rr}, $(u_n)$ is bounded and we may assume that $u_n \rightharpoonup u_0$ in $H^1(\Omega)$.
From $$E_{\lambda_n}(u_n)=\lambda_n\left(A(u_n)+B(u_n)\right)$$
we infer that $\int_{\Omega} |\nabla u_n|^2 \rightarrow 0$, so $u_n \rightarrow u_0$ and $u_0$ is a constant.
Since $\int_\Omega m<0<\int_{\partial \Omega} b$ and $\int_\Omega a<0$, $j_{u_0}$ has a unique critical point, which is a global minimum point. Thus there is an unique constant $c$ in $\mathcal{N}_\lambda^+$. By Remark \ref{cn}, we infer that $c=c_0$, where $c_0$ is the unique zero of $\varphi$.
In particular, $c_0 \in \mathcal{N}_{\lambda_n}^+ \cap B^+$ for every $n$.
Then
$$I_{\lambda_n}(u_n)\leq I_{\lambda_n}(c_0)<0,$$
and consequently
$$-\frac{1}{2}\int_{\Omega} mu_n^2-\frac{1}{p} A(u_n)-\frac{1}{q}B(u_n)\leq -\frac{c_0^2}{2} \int_{\Omega} m-\frac{c_0^p}{p}\int_{\Omega} a -\frac{c_0^q}{q}\int_{\partial \Omega}b<0.$$
It follows that $u_n \not \rightarrow 0$, i.e. $u_0$ is a positive constant. 
Finally, since $u_n$ solves $(P_\lambda)$ for $\lambda=\lambda_n$, we have
\begin{align*}
0 & = \int_\Omega \left(\nabla u_n \nabla u_0 - \lambda_n m u_n u_0\right)  - \lambda_n \int_\Omega a u_n^{p-1} u_0 - \lambda_n \int_{\partial \Omega} b u_n^{q-1} u_0 \\
& = \lambda_n \left\{ - \int_\Omega m u_n u_0  - \int_\Omega a u_n^{p-1} u_0 - \int_{\partial \Omega} b u_n^{q-1} u_0 \right\}
\end{align*}
so, letting $n \to \infty$, we get $$u_0^{2-q} \int_{\Omega} m +u_0^{p-q} \int_{\Omega} a + \int_{\partial \Omega} b=0,$$ i.e. $u_0=c_0$.
\\ 

\item If \eqref{n1} holds then we can proceed as in the previous item to deduce that $u_n \rightarrow c$, where $c$ is a constant.
Now, by Remark \ref{cn} we infer that $c=c_1$.
\end{enumerate}
\end{proof}

\begin{rem}{\rm
Let $b_n \rightarrow b$ in $L^{\infty}(\partial \Omega)$ with $b_n^+ \not \equiv 0$ for every $n$, $b \not \equiv 0$ and $b \leq 0$. Then, from \cite[Remark 2.5]{RQU}, we have $\lambda_{b_n} \to \lambda_b>0$.
Moreover, since $b_n \rightarrow b$ in $L^{\infty}(\partial \Omega)$, we have $\int_{\partial \Omega} b_n<0$ for $n$ sufficiently large and $b_n^+ \to 0$ in $L^{\infty}(\partial \Omega)$. So
we can fix $\tilde{b} \in L^{\infty}(\partial \Omega)$ such that  $\int_{\partial \Omega} \tilde{b}<0$, $\tilde{b}^+ \not \equiv 0$, and $b_n\leq \tilde{b}$ for $n$ sufficiently large. By Remark \ref{rbb}, we have  $\lambda_s(b_n)\geq \lambda_s(\tilde{b})>0$. So $\overline{\lambda}(b_n)\geq \min\{\lambda_s(\tilde{b}), \lambda_{\tilde{b}}\}$ and
$u_{0,\lambda,b_n}$ exists for $0<\lambda<\min\{\lambda_s(\tilde{b}), \lambda_{\tilde{b}}\}$ and every $n$. Moreover, from the proof of Proposition \ref{pmn1}, for some $K_\lambda>0$ we have 
$\|u_{0,\lambda,b_n}\|\leq K_\lambda \|b_n^+\|_{\infty}^{\frac{1}{2-q}}$, and consequently $u_{0,\lambda,b_n} \to 0$ in $\mathcal{C}^{\theta}(\overline{\Om})$ for some $\theta \in (0,1)$ and $0<\lambda<\min\{\lambda_s(\tilde{b}), \lambda_{\tilde{b}}\}$.}
\end{rem}
    
\begin{lem}
\label{leqil}
 Assume $b^+ \not \equiv 0$ and $\int_{\partial \Omega} b<0$. Then, for $0<\lambda<\overline{\lambda}$, there holds
 \begin{equation}
 \label{eqil}
 I_{\lambda}(u_{0,\lambda})<-D_0\lambda^{\frac{2}{2-q}}+o(\lambda^{\frac{2}{2-q}}),
\end{equation}  
for some $D_0>0$.
\end{lem}

\begin{proof}
Recall that for $0<\lambda<\lambda_b$ we have $B^+ \subset E_{\lambda}^+$. Then there exists $C_0>0$ such that $0<E_\lambda(u)\leq C_0\|u\|^2$ for  $u \in B^+$ and $0<\lambda<\lambda_b$.

Let $u \in B^+ \cap A_0^+$.
Then $$I_\lambda(u) \leq \frac{1}{2}E_\lambda(u)-\frac{\lambda}{q}B(u)\leq \tilde{I}_\lambda(u):=\frac{C_0}{2}\|u\|^2-\frac{\lambda}{q}B(u).$$
Thus $I_\lambda(tu)\leq \tilde{I}_\lambda(tu)$ for every $t>0$.
Note that $\tilde{I}_\lambda(tu)$ has a global minimum point $t_0$ given by
$$t_0=\left(\frac{\lambda B(u)}{C_0\|u\|^2}\right)^{\frac{1}{2-q}}.$$
and $$\tilde{I}_\lambda(t_0u)=-\frac{2-q}{2q}\lambda^{\frac{2}{2-q}}B(u)\left(\frac{B(u)}{C_0\|u\|^2}\right)^{\frac{q}{2-q}}=-D_0\lambda^{\frac{2}{2-q}},$$
where $D_0=\frac{2-q}{2q}B(u)\left(\frac{B(u)}{C_0\|u\|^2}\right)^{\frac{q}{2-q}}$.
It follows that if $I_\lambda(tu)$ has a local minimum at $t_1$ then $$I_\lambda(t_1u)< -D_0\lambda^{\frac{2}{2-q}}$$
with $D_0>0$

Let now $u \in B^+ \cap A^-$. Then $$I_\lambda(u)\leq \tilde{I}_\lambda(u):=\frac{C_0}{2}\|u\|^2-\frac{\lambda}{q}B(u)-\frac{\lambda}{p}A(u)$$ and $\tilde{I}_\lambda (tu)$ has a global minimum point $t_0$ which satisfies
$$t_0C_0\|u\|^2-\lambda t_0^{q-1}B(u)-\lambda t_0^{p-1} A(u)=0.$$
Thus $$t_0C_0\|u\|^2-\lambda t_0^{q-1}B(u)<0,$$
so that $$t_0<\left(\frac{\lambda B(u)}{C_0\|u\|^2}\right)^{\frac{1}{2-q}}.$$
Hence $$\tilde{I}_\lambda (t_0u)<-D_0\lambda^{\frac{2}{2-q}}-\lambda^{\frac{p-q+2}{2-q}}\left(\frac{B(u)}{C_0\|u\|^2}\right)^{\frac{1}{2-q}}A(u)=-D_0\lambda^{\frac{2}{2-q}}+D_1\lambda^{\frac{p-q+2}{2-q}},$$
where $D_0$ is as above and $D_1=-\left(\frac{B(u)}{C_0\|u\|^2}\right)^{\frac{p}{2-q}}\frac{A(u)}{p}$. Once again, if $I_\lambda(tu)$ has a local minimum at $t_1$ then $$I_\lambda(t_1u)<-D_0\lambda^{\frac{2}{2-q}}+D_1\lambda^{\frac{p-q+2}{2-q}}.$$
Therefore we conclude that $$\inf_{\mathcal{N}_\lambda^+ \cap B^+} I_\lambda <-D_0\lambda^{\frac{2}{2-q}}+o(\lambda^{\frac{2}{2-q}})$$
for $0<\lambda<\overline{\lambda}$.
\end{proof}   

\begin{prop}
\label{pasyc}
Assume $b^+ \not \equiv 0$ and $\int_{\partial \Omega} b<0$. If $\lambda_n \to 0^+$ then, up to a subsequence, there holds $w_{0,\lambda_n}=\lambda_n^{-\frac{1}{2-q}}u_{0,\lambda_n} \rightarrow w_0$ in $H^1(\Omega)$, where $w_0$ is a nontrivial non-negative solution of 
\begin{equation}
\label{pwl}
-\Delta w = 0 \quad\mbox{in} \ \Omega, \qquad
\frac{\partial w}{\partial \mathbf{n}} = b(x)w^{q-1} \quad \mbox{on} \
\partial \Omega.
\end{equation}
\end{prop}

\begin{proof}
Let $w_n=w_{0,\lambda_n}=\lambda_n^{-\frac{1}{2-q}}u_{0,\lambda_n}$. Since $$\|u_{0,\lambda}\|\leq \left(\lambda\|b^+\|_{\infty}\right)^{\frac{1}{2-q}} K$$ it follows that $(w_n)$ is bounded in $H^1(\Omega)$. Thus, up to a subsequence, $w_n \rightharpoonup w_0$ in $H^1(\Omega)$. Furthermore,
from Lemma \ref{leqil} we have
$$I_{\lambda_n}(u_{0,\lambda_n})<-D_0\lambda_n^{\frac{2}{2-q}}++o(\lambda_n^{\frac{2}{2-q}}),$$
with $D_0>0$.
Hence
$$\frac{\lambda_n^{\frac{2}{2-q}}}{2} \int_{\Omega} |\nabla w_n|^2 -\frac{\lambda_n^{\frac{4-q}{2-q}}}{2}\int_{\Omega} mw_n^2 -\frac{\lambda_n^{\frac{2}{2-q}}}{q} B(w_n) -\frac{\lambda_n^{\frac{2+p-q}{2-q}}}{p}
A(w_n)<-D_0\lambda_n^{\frac{2}{2-q}}++o(\lambda_n^{\frac{2}{2-q}}).$$
Dividing the above inequality by $\lambda_n^{\frac{2}{2-q}}$ and letting $n \to \infty$ we get
$$\frac{1}{2}\int_{\Omega} |\nabla w_0|^2- \frac{1}{q}B(w_0)\leq -D_0<0,$$
so $w_0 \not \equiv 0$.
Taking $v=w_n-w_0$ in
\begin{equation}
\label{ewn}
\int_{\Omega} \left(\nabla w_n \nabla v -\lambda_n m(x)w_n v -\lambda_n^{\frac{p-q}{2-q}}a(x)w_n^{p-1}v\right) -\int_{\partial \Omega} b(x)w_n^{q-1}v =0 \quad \forall v \in H^1(\Omega)
\end{equation}
and letting $n \to \infty$ 
we get $\lim \int_{\Omega} \nabla w_n \nabla (w_n-w_0) =0$, so that $w_n \rightarrow w_0$ in $H^1(\Omega)$.
Finally, \eqref{ewn} also shows that $w_0$ is a solution of  \eqref{pwl}.
\end{proof}

\medskip
\subsection{Minimization in $\mathcal{N}_\lambda^+ \cap E_\lambda^-$}
\strut
\medskip
\begin{prop}
Assume $\lambda_a, \lambda_b>0$ and $\mathcal{N}_\lambda^+ \cap E_\lambda^- \neq \emptyset$. If $\displaystyle \inf_{\mathcal{N}_\lambda^+ \cap E_\lambda^-} I_{\lambda}<0$ and $0<\lambda<\min\{\lambda_a,\lambda_b\}$ then $\displaystyle \inf_{\mathcal{N}_\lambda^+ \cap E_\lambda^-} I_{\lambda}$ is achieved.
\end{prop}

\begin{proof}
First of all, from $0<\lambda<\min\{\lambda_a,\lambda_b\}$ we have $E_\lambda^- \subset A^- \cap B^-$ and by Proposition \ref{p2'}, taking $\lambda < \lambda_* <\min\{\lambda_a,\lambda_b\}$ we get a constant $D_0>0$ such that $$A(u)\leq -D_0 \|u\|^p$$ if $u \in E_\lambda^-$. Thus, from $$\lambda A(u)=E_\lambda(u)-\lambda B(u)$$ we get
$$\lambda D_0 \|u\|^p \leq -\lambda A(u)< -E_\lambda(u)< \lambda C \|m\|_\infty \|u\|^2,$$ for some $C>0$, and consequently there exists $K>0$ such that 
\begin{equation}
\label{ebe}
\|u\|\leq K^{\frac{1}{p-2}}
\end{equation}
 if $u \in \mathcal{N}_\lambda^+ \cap E_\lambda^-$.
Let now $(u_n) \subset \mathcal{N}_\lambda^+ \cap E_\lambda^-$ be such that $$I_{\lambda}(u_n) \rightarrow \inf_{\mathcal{N}_\lambda^+ \cap E_\lambda^-} I_{\lambda}.$$ Since $(u_n)$ is bounded, we may assume that $u_n \rightharpoonup u_0$ in $H^1(\Omega)$. In particular, we have $E_\lambda(u_0)\leq 0$ and
$I_{\lambda}(u_0)\leq \displaystyle \inf_{\mathcal{N}_\lambda^+ \cap E_\lambda^-} I_{\lambda}<0$, so that $u_0 \not \equiv 0$. Since $\lambda<\min\{\lambda_a,\lambda_b\}$, we have $E_{\lambda,0}^- \setminus \{0\} \subset (A^- \cap B^-) \cup \{0\}$, so 
$u_0 \in A^- \cap B^-$. Moreover, from 
$$0>I_{\lambda}(u_0)=\frac{1}{2}E_\lambda(u_0)-\frac{\lambda}{p}A(u_0) -\frac{\lambda}{q}B(u_0),$$
we infer that $E_\lambda(u_0)<0$, i.e. $u_0 \in E_\lambda^-$.
Now, as $j_{u_0}(1)=I_{\lambda}(u_0)<0$, we see, from the shape of $j_{u_0}$, that there exists $t_2>0$ such that 
$t_2 u_0 \in \mathcal{N}_\lambda^+ \cap E_\lambda^-$, i.e. $t_2$ is a global minimum point of $j_{u_0}$.
If $u_n \not \rightarrow u_0$ then 
$$j_{u_0}(t_2)\leq j_{u_0}(1)<\liminf j_{u_n}(1)=\lim I_{\lambda}(u_n)=\inf_{\mathcal{N}_\lambda^+ \cap E_\lambda^-} I_{\lambda},$$ which is a contradiction.
Therefore $u_n \rightarrow u_0$ and consequently $t_2(u_0)=1$.
Thus $u_0 \in \mathcal{N}_\lambda^+ \cap E_\lambda^-$ and $$I_{\lambda}(u_0)=\inf_{\mathcal{N}_\lambda^+ \cap E_\lambda^-} I_{\lambda}.$$
\end{proof}

\begin{prop}
$\mathcal{N}_\lambda^+ \cap E_\lambda^- \neq \emptyset$ and $\displaystyle \inf_{\mathcal{N}_\lambda^+ \cap E_\lambda^- } I_{\lambda}<0$ in the following cases:
\begin{enumerate}
\item $\int_\Om m>0>\int_\Om a$, $0>\int_{\partial \Om} b > -K_1(m,a)$ and $\lambda>0$.

\item  $\int_\Om m<0$, $\int_{\partial \Om} b\varphi_1^q<0$, 
$-\left(\frac{C_{pq}}{-\int_{\partial \Om} b\varphi_1^q}\right)^{\frac{p-2}{2-q}}<\int_{\Omega} a\varphi_1^p<0$
 and $\lambda >\lambda^*$, where
\begin{equation}
\label{l*}
\lambda^*=\lambda_1(m) \left[1-\left(C_{pq}^{-1}\left(-\int_{\partial \Om} b\varphi_1^q\right) \left(-\int_{\Omega} a\varphi_1^p\right)^{\frac{2-q}{p-2}}\right)^{\frac{p-2}{p-q}}\right]^{-1}.
\end{equation}
\end{enumerate}
\end{prop}

\begin{proof}
Let $u \in B^- \cap E_{\lambda}^- \cap A^-$. It is clear that if $ j_u(t)<0$ for some $t>0$ then $j_u$ has a local maximum followed by a global minimum, i.e. there are $t_1<t_2$ such that $t_1 u \in \mathcal{N}_{\lambda}^- \cap E_{\lambda}^- $ and $t_2 u \in \mathcal{N}_{\lambda}^+ \cap E_{\lambda}^- $. Moreover, in this case we have $j_u(t_2)<0$, so that $\displaystyle \inf_{\mathcal{N}_\lambda^+ \cap E_\lambda^- } I_{\lambda}<0$.
Note that $j_u(t)<0$ if and only if $i_u(t)<0$. Now $t_0(u)$ given by \eqref{eqto} is a global minimum point of $i_u$ and
\begin{eqnarray*}
i_u(t_0(u))<0 &\Leftrightarrow&
\lambda (-B(u))<C_{pq} \frac{(-E_\lambda(u))^{\frac{p-q}{p-2}}}{\left(-\lambda A(u)\right)^{\frac{2-q}{p-2}}} \\ 
&\Leftrightarrow &
(-B(u))<C_{pq} \frac{\left(\int_{\Omega} mu^2-\lambda^{-1}\int_{\Omega} |\nabla u|^2 \right)^{\frac{p-q}{p-2}}}{(-A(u))^{\frac{2-q}{p-2}}}\\
&\Leftrightarrow & E_{\lambda}(u)+\lambda \left(-C_{pq}^{-1} B(u)\right)^{\frac{p-2}{p-q}}(-A(u))^{\frac{2-q}{p-q}}<0.
\end{eqnarray*}
Note that if $$\int_\Om m>0>\int_\Om a, \quad 0>\int_{\partial \Om} b > -K_1(m,a) \quad \text{and}  \quad \lambda>0 $$
then $c \in  B^- \cap E_{\lambda}^- \cap A^-$ for any constant $c$. Moreover, 
$$E_{\lambda}(c)+\lambda \left(-C_{pq}^{-1} B(c)\right)^{\frac{p-2}{p-q}}(-A(c))^{\frac{2-q}{p-q}}= \lambda c^2\left( -\int_{\Omega} m + \left(-\frac{1}{C_{pq}}\int_{\partial \Om} b\right)^{\frac{p-2}{p-q}}\left(-\int_\Om a\right)^{\frac{2-q}{p-q}}\right),$$
so that $j_c(t_0(c))<0$ for $\lambda>0$.

On the other hand, if
$$\int_\Om m<0, \quad \int_{\partial \Om} b\varphi_1^q<0, \quad
-\left(\frac{C_{pq}}{-\int_{\partial \Om} b\varphi_1^q}\right)^{\frac{p-2}{2-q}}<\int_{\Omega} a\varphi_1^p<0, \quad \text{and}  \quad \lambda>\lambda_1,$$
then
$\varphi_1 \in B^- \cap E_{\lambda}^- \cap A^-.$ Furthermore 
$$E_{\lambda}(\varphi_1)+\lambda  \left(-C_{pq}^{-1} B(\varphi_1)\right)^{\frac{p-2}{p-q}}(-A(\varphi_1))^{\frac{2-q}{p-q}}=\lambda_1-\lambda \left[1-\left(-\frac{1}{C_{pq}}\int_{\partial \Om} b\varphi_1^q\right)^{\frac{p-2}{p-q}} \left(-\int_{\Omega} a\varphi_1^p\right)^{\frac{2-q}{p-q}}\right],$$
so that $j_{\varphi_1}(t_0(\varphi_1))<0$ for $\lambda>\lambda^*$.

\end{proof}

\begin{cor}
\label{c2}
 $\displaystyle \inf_{\mathcal{N}_\lambda^+ \cap E_\lambda^-} I_{\lambda}$ is achieved by some $u_{1,\lambda}\geq 0$ in the following cases:
\begin{enumerate}
\item $\int_\Om m>0>\int_\Om a$, $0>\int_{\partial \Om} b > -K_1(m,a)$ and $0<\lambda<\min\{\lambda_a,\lambda_b\}$.

\item  $\int_\Om m<0$, $\int_{\partial \Om} b\varphi_1^q<0$, 
$-\left(\frac{C_{pq}}{-\int_{\partial \Om} b\varphi_1^q}\right)^{\frac{p-2}{2-q}}<\int_{\Omega} a\varphi_1^p<0$
 and $\lambda^*<\lambda <\min\{\lambda_a,\lambda_b\}$.
\end{enumerate}
\end{cor}

\begin{rem}{\rm 
Let us show that the condition $\lambda^*<\min\{\lambda_a,\lambda_b\}$ assumed in Corollary \ref{c2} (2) may indeed hold when $\int_{\Omega} m<0$. To this end, let $a_0 \in L^{\infty}(\Omega)$ and $b \in  L^{\infty}(\partial \Omega)$ be such that
$\int_{\Omega} a_0\varphi_1^p<0$ and $\int_{\partial \Omega} b\varphi_1^q<0$. Then $\lambda_{a_0}>\lambda_1$ and $\lambda_b>\lambda_1$. Let us set $a_\varepsilon=\varepsilon a_0$ for $\varepsilon>0$. As one can easily see from the definition of $\lambda_a$, we have $\lambda_{a_\varepsilon}=\lambda_{a_0}$  for every $\varepsilon>0$. Furthermore, from \eqref{l*}, we see that
$\lambda^*(m,a_\varepsilon,b) \rightarrow \lambda_1$ as $\varepsilon \to 0$. Thus, for some $\varepsilon_0>0$ there holds $\lambda^*(m,a_\varepsilon,b)<\min\{\lambda_{a_\varepsilon}, \lambda_b\}$.
The same argument applies if we consider $b_\varepsilon=\varepsilon b_0$, where $b_0 \in  L^{\infty}(\partial \Omega)$ is such that $\int_{\partial \Omega} b_0\varphi_1^q<0$. Therefore $\lambda^*(m,a_{\varepsilon},b_{\varepsilon})<\min\{\lambda_a,\lambda_b\}$ if $0<\varepsilon<\varepsilon_0$, for some $\varepsilon_0$.
}\end{rem}

\begin{prop}
\label{p3}
If $\int_\Om m>0>\int_\Om a$ and $0>\int_{\partial \Om} b > -K_1(m,a)$  then $u_{1,\lambda} \to c_2$ in $\mathcal{C}^{\theta}(\overline{\Om})$ for some $\theta \in (0,1)$ as $\lambda \to 0^+$. 
\end{prop}

\begin{proof}
Let $\lambda_n \to 0^+$ and $u_n=u_{1,\lambda_n}$. From \eqref{ebe} we infer that $(u_n)$ is bounded. So, up to a subsequence, we have $u_n \rightharpoonup u_0$ in $H^1(\Omega)$. From
$$0\geq \liminf I_{\lambda_n}(u_n) \geq \int_{\Omega} |\nabla u_0|^2$$ we get $u_n \rightarrow u_0$ in $H^1(\Omega)$ and $u_0$ is a non-negative constant. If $u_0=0$ then we set $v_n=\frac{u_n}{\|u_n\|}$ and assume that $v_n \rightharpoonup v_0$ in $H^1(\Omega)$. Since $v_n \in E_{\lambda_n}^-$ we have $$\int_{\Omega} |\nabla v_0|^2 \leq \liminf E_{\lambda_n} (v_n)\leq 0$$ so $v_n \rightarrow v_0$ and $v_0$ is a positive constant. Moreover, from
$$\lambda_n\left(-\frac{1}{2}\int_{\Omega} mu_n^2 -\frac{1}{p}A(u_n) -\frac{1}{q} B(u_n)\right)<0$$ we get
$$\frac{1}{q} B(v_n)\geq -\frac{1}{2}\|u_n\|^{2-q} \int_{\Omega} mv_n^2 -\frac{1}{p}\|u_n\|^{p-q} A(v_n).$$
Thus $B(v_0)\geq 0$, which combined with $v_n \in B^-$ provides $v_0 \in B_0$ and contradicts $\int_{\partial \Om} b<0$. Therefore $u_0 \neq 0$, i.e. $u_0$ is a positive constant.
Since $u_n \rightarrow u_0$ and $u_n$ is a solution of $(P_{\lambda_n})$ we have
\begin{align*}
0 & = \int_\Omega \left(\nabla u_n \nabla u_0 - \lambda_n m u_n u_0\right)  - \lambda_n \int_\Omega a u_n^{p-1} u_0 - \lambda_n \int_{\partial \Omega} b u_n^{q-1} u_0 \\
& = \lambda_n \left\{ - \int_\Omega m u_n u_0  - \int_\Omega a u_n^{p-1} u_0 - \int_{\partial \Omega} b u_n^{q-1} u_0 \right\}
\end{align*}
so $$u_0^{2-q} \int_{\Omega} m +u_0^{p-q} \int_{\Omega} a + \int_{\partial \Omega} b=0,$$ i.e. $u_0$ is a positive zero of $\varphi$. Finally, from $u_n \in N_{\lambda_n}^+$, we get
$$E_{\lambda_n}(u_n)<\lambda_n \frac{p-q}{p-2}B(u_n).$$
In particular, we have $$-\lambda_n \int_\Omega mu_n^2 <\lambda_n \frac{p-q}{p-2}B(u_n),$$
so, from $u_n \rightarrow u_0$ we get
$$u_0^{2-q}\geq \frac{p-q}{p-2} \frac{\int_{\partial \Omega}b}{\left(-\int_{\Omega} m\right)}.$$ Since $u_0$ is a positive zero of $\varphi$, we get $$u_0\geq \left(-\frac{2-q}{p-q} \frac{\int_{\Omega} m}{\int_{\Omega} a}\right)^{\frac{1}{p-2}},$$ so that, by \eqref{eph}, $u_0=c_2$.
\end{proof}

\section{Minimization in $\mathcal{N}_\lambda^-$}
\label{sec4}

\begin{prop}
\label{pch}
Assume $a^+ \not \equiv 0$ and either $\int_{\Omega} a<0$ or \eqref{n1}. Then
$\mathcal{N}_\lambda^- \cap A^+ \neq \emptyset$  for every $0<\lambda  <\min\{\lambda_a, \lambda_s\}$. Moreover, for $0<\lambda  <\min\{\lambda_a, \lambda_s\}$, there holds:
\begin{enumerate}
\item $\mathcal{N}_\lambda^- \cap A^+$ is bounded away from zero, i.e. there exists $K_\lambda>0$ such that $\|u\|\geq K_\lambda$ for $u \in \mathcal{N}_\lambda^- \cap A^+$.
\item  If $(u_n) \subset \mathcal{N}_\lambda^- \cap A^+$ is a sequence such that $(I_{\lambda}(u_n))$ is bounded from above then $(u_n)$ is bounded.
\item $\displaystyle \inf_{\mathcal{N}_\lambda^- \cap A^+} I_{\lambda}>0$.
\end{enumerate}
\end{prop}

\begin{proof}
First of all, note that since $a^+ \not \equiv 0$ we have $A^+ \neq \emptyset$. Let  $u \in A^+$. If $u \in B^-$ then $j_u$ has a global maximum point $t_1>0$, so $t_1 u \in \mathcal{N}_\lambda^-$. The same conclusion holds if $u \in B_0^+$, since $\lambda<\lambda_a$ provides $u\in E_\lambda^+$, whereas $0<\lambda<\lambda_s$ yields that $j_u$ has a global maximum point, by Proposition \ref{p1}. Therefore  $\mathcal{N}_\lambda^- \cap A^+ \neq \emptyset$.
\begin{enumerate}
\item  Let us assume first $\int_{\Omega} a<0$. Given $0<\lambda < \mu<\lambda_a$ and $u \in \mathcal{N}_\lambda^- \cap A^+$, we apply Proposition \ref{p2'}. Then, for some $C_0,D>0$, we have
\begin{equation}
\label{baaaz}
C_0\|u\|^2\leq E_\lambda(u)<\lambda \frac{p-q}{2-q} A(u)\leq \lambda D\|a^+\|_{\infty}\|u\|^p,
\end{equation}
and consequently 
\begin{equation}
\label{baaz}
\|u\|\geq \left( \frac{C_0}{\lambda D\|a^+\|_{\infty}}\right)^{\frac{1}{p-2}}.
\end{equation}
Now, if \eqref{n1} holds then, since $\lambda<\lambda_a$, there is a constant $C_\lambda>0$ such that $E_\lambda(u)\geq C_\lambda \|u\|^2$ for every $u \in A^+$. Thus
$$ C_\lambda \|u\|^2\leq E_\lambda(u)<\lambda \frac{p-q}{2-q} A(u)\leq \lambda D\|a^+\|_{\infty}\|u\|^p,$$
and consequently $$\|u\|\geq \left( \frac{C_\lambda}{\lambda D\|a^+\|_{\infty}}\right)^{\frac{1}{p-2}}$$ 
if $u \in \mathcal{N}_\lambda^- \cap A^+$.\\
\item Note that
$$I_{\lambda}(u)=\frac{p-2}{2p}E_\lambda(u)-\lambda \frac{p-q}{pq}B(u)$$  if $u \in \mathcal{N}_\lambda^-$.
Hence, as $\lambda<\lambda_a$, if in addition $u \in A^+$ then 
$$I_{\lambda}(u)\geq C_\lambda \|u\|^2 - \lambda D \|u\|^q,$$ for some constants $C_\lambda,D>0$.
From the above inequality we deduce that if $(u_n) \subset \mathcal{N}_\lambda^- \cap A^+$ is such that $(I_{\lambda}(u_n))$ is bounded from above then $(u_n)$ is bounded.\\

\item If $u \in \mathcal{N}_\lambda^-$ then
\begin{equation}
\label{ej1}
I_{\lambda}(u)=\frac{p-2}{2p}E_\lambda(u)-\lambda \frac{p-q}{qp}B(u).
\end{equation}
If, in addition, $u \in A^+$ then $E_\lambda(u)\geq C_\lambda \|u\|^2$ for some constant $C_\lambda>0$, since $\lambda<\lambda_a$.
Thus, if $u \in \mathcal{N}_\lambda^- \cap A^+ \cap B_0^-$ then
$$I_{\lambda}(u)\geq D_\lambda \|u\|^2\geq \tilde{D}_\lambda>0,$$ where we used (1).
Now, if $(u_n) \subset \mathcal{N}_\lambda^- \cap A^+ \cap B^+$ then,  from $\lambda<\lambda_s$, we have $$I_{\lambda}(u_n)=j_{u_n}(1)\geq j_{u_n}(t_0(u_n))>0.$$
If, in addition $I_{\lambda}(u_n)\rightarrow 0$
then $j_{u_n}(t_0(u_n))\rightarrow 0$, and consequently either $t_0(u_n)\rightarrow 0$ or $i_{u_n}(t_0(u_n)) \rightarrow 0$.
In the first case, we get $E_\lambda(u_n)\rightarrow 0$, so that  $u_n \rightarrow 0$ in $H^1(\Omega)$, which contradicts (1).
Now, if $i_{u_n}(t_0(u_n)) \rightarrow 0$ then
$$-\lambda B(u_n)+C_{pq} \frac{E_\lambda(u_n)^{\frac{p-q}{p-2}}}{\left(\lambda A(u_n)\right)^{\frac{2-q}{p-2}}} \rightarrow 0^+.$$ 
Since $(u_n)$ is bounded, we may assume that $$u_n \rightharpoonup u_0 \text{ in } H^1(\Omega), \quad A(u_0)=\lim A(u_n)\geq 0 \quad \text{and} \quad B(u_0)=\lim B(u_n).$$
One may easily see that if either $B(u_0)=0$ or $A(u_0)=0$ then $E_\lambda(u_n) \rightarrow 0$ and  we infer again that $u_n \rightarrow 0$ in $H^1(\Omega)$, which contradicts (1).
Thus $u_0 \in A^+ \cap B^+ \cap E_\lambda^+$. From $\lambda<\lambda_s$ we have $$0<i_{u_0}(t_0(u_0))\leq \lim i_{u_n}(t_0(u_n)),$$ which is a contradiction.
Therefore we can't have $I_{\lambda}(u_n) \rightarrow 0$, so  $\displaystyle \inf_{\mathcal{N}_\lambda^- \cap A^+} I_{\lambda}>0$.
 \end{enumerate}
\end{proof}

\begin{rem}
\label{rr2}
\strut
{\rm 
\begin{enumerate}
\item If \eqref{n1} holds and $\lambda<\lambda_1$ then the conclusions of Proposition \ref{pch} remain valid. Indeed, in this case, for every $\lambda<\lambda_1$ there exists a constant $C_\lambda>0$ such that $E_\lambda(u)\geq C_\lambda \|u\|^2$ for every $u \in H^1(\Omega)$.
\item If $\int_{\Omega} a<0$ then \eqref{baaz} shows that $\mathcal{N}_\lambda^- \cap A^+$ is uniformly bounded away from zero for $\lambda \in (0,\mu)$, with $\mu <\lambda_a$.
Furthermore, the statement in (2) can be strengthened as follows: 
if $(\lambda_n) \subset (0,\mu)$ and $(u_n) \subset  \mathcal{N}_{\lambda_n}^- \cap A^+$ are such that $I_{\lambda_n}(u_n)$ is bounded then $(u_n)$ is bounded. As a matter of fact,
in this case we have
$$I_{\lambda_n}(u_n)\geq C_0\|u_n\|^2-\mu D \|u_n\|^q$$
for some constants $C_0,D>0$.
\end{enumerate} 
}\end{rem}


\begin{prop}
\label{pmn2}
Assume $a^+ \not \equiv 0$.  If either \eqref{n1} or $\int_{\Omega} a<0$ holds  then $\displaystyle \inf_{\mathcal{N}_\lambda^- \cap A^+} I_{\lambda}$ is achieved by some $u_{2,\lambda}\geq 0$ for $0<\lambda<\min\{\lambda_s,\lambda_a\}$. Moreover, if \eqref{n1} holds then $u_{2,\lambda} \rightarrow c_2$  in $\mathcal{C}^{\theta}(\overline{\Om})$ for some $\theta \in (0,1)$ as $\lambda \to 0^+$.
\end{prop}

\begin{proof}
Let $0<\lambda<\min\{\lambda_s,\lambda_a\}$ and $(u_n) \subset \mathcal{N}_\lambda^- \cap A^+$ be such that
$$I_{\lambda}(u_n) \rightarrow \inf_{\mathcal{N}_\lambda^- \cap A^+} I_{\lambda}.$$
By Proposition \ref{pch}, we know that $(u_n)$ is bounded, so we may assume that $u_n \rightharpoonup u_0$ in $H^1(\Omega)$ and $u_0 \in A_0^+$. 
From 
\begin{equation}
\label{de}
\lambda \frac{p-q}{p-2}B(u_n)<E_\lambda(u_n)< \lambda \frac{p-q}{2-q}A(u_n)
\end{equation}
 we have $$E_\lambda(u_0)\leq \liminf E_\lambda(u_n) \leq \limsup E_\lambda(u_n)\leq \lambda \frac{p-q}{2-q}A(u_0).$$
If $u_0 \equiv 0$ then the above inequalities provide $E_\lambda(u_n) \rightarrow E_\lambda(u_0)$, so $u_n \rightarrow u_0\equiv 0$. This is impossible by Proposition \ref{pch} (1). Hence $u_0 \not \equiv 0$. Moreover, if $A(u_0)=0$ then $E_\lambda(u_0)\leq 0$, with $u_0 \not \equiv 0$, which contradicts $\lambda<\lambda_a$. Thus $u_0 \in A^+$ and it is easily seen that there exists $t_2>0$ such that $t_2 u_0 \in \mathcal{N}_\lambda^- \cap A^+$.
We claim that $u_n \rightarrow u_0$. Indeed, if not then 
$$j_{u_0}(t)<\liminf j_{u_n}(t)$$ for every $t>0$. 
Hence 
$$I_{\lambda} (t_2u_0)=j_{u_0}(t_2)<\liminf j_{u_n}(t_2)\leq \liminf j_{u_n}(1)=\lim I_{\lambda}(u_n)=\inf_{\mathcal{N}_\lambda^- \cap A^+} I_{\lambda}.$$
We have then a contradiction, so $u_n \rightarrow u_0$ in $H^1(\Omega)$ and
$$I_{\lambda}(u_0)=\inf_{\mathcal{N}_\lambda^- \cap A^+} I_{\lambda}.$$
We denote $u_0$ by $u_{2,\lambda}$.

Assume now \eqref{n1}. Let $\lambda_n \to 0^+$ and $u_n=u_{2,\lambda_n}$. We claim that $(u_n)$ is bounded. Indeed, assume that $\|u_n\|\rightarrow \infty$ and set $v_n=\frac{u_n}{\|u_n\|}$. We may assume that $v_n \rightharpoonup v_0$ in $H^1(\Omega)$. Note that
since $\int_\Omega a>0$ and $0<\lambda<\lambda_s$, there is a unique positive constant $c \in \mathcal{N}_{\lambda_n}^- \cap A^+$ for every $n$. Thus
$$\frac{p-2}{2p}\lambda_n A(u_n)-\frac{2-q}{2q}\lambda_n B(u_n)=I_{\lambda_n}(u_n)\leq \lambda_n \left(-\frac{\mathcal{C}^2}{2} \int_{\Omega} m-\frac{\mathcal{C}^p}{p}\int_{\Omega} a -\frac{\mathcal{C}^q}{q}\int_{\partial \Omega}b\right).$$
It follows that $I_{\lambda_n}(u_n)\rightarrow0$ and $A(v_0)=\lim A(v_n)=0$.
If $u_n \in B_0^-$ then $$I_{\lambda_n}(u_n)\geq \frac{p-2}{2p} E_{\lambda_n}(u_n),$$
so that $E_{\lambda_n}(v_n)\rightarrow 0$, and consequently $\int_{\Omega} |\nabla v_n|^2 \rightarrow 0$, i.e. $v_0$ is a nonzero constant. This contradicts $A(v_0)=0$.
On the other hand, if $u_n \in B^+$ then $u_n \in B^+ \cap A^+$ and $j_{u_n}(t_0(u_n)) \to 0$. So either $t_0(u_n) \rightarrow 0$ or 
$$-\lambda_n B(v_n)+C_{pq} \frac{E_{\lambda_n}(v_n)^{\frac{p-q}{p-2}}}{\left(\lambda_n A(v_n)\right)^{\frac{2-q}{p-2}}} \rightarrow 0^+.$$ 
Since 
\begin{eqnarray*}
t_0(u_n)&=&\left(\frac{p(2-q)}{2(p-q)} \frac{E_{\lambda_n}(u_n)}{\lambda_n A(u_n)}\right)^{\frac{1}{p-2}}=\left(\frac{p(2-q)}{2(p-q)} \left(1+\frac{B(u_n)}{A(u_n)}\right)\right)^{\frac{1}{p-2}}\\&>&\left(\frac{p(2-q)}{2(p-q)}\right)^{\frac{1}{p-2}},
\end{eqnarray*}
 the first case is ruled out.
In the second case, it follows that $E_{\lambda_n}(v_n)\rightarrow 0$ and once again we deduce that $v_0$ is a nonzero constant, which is impossible. Therefore $(u_n)$ is bounded and we may assume that $u_n \rightharpoonup u_0$ in $H^1(\Omega)$. From $$E_{\lambda_n}(u_n)=\lambda_n\left(A(u_n)+B(u_n)\right)$$
we infer that $\int_{\Omega} |\nabla u_n|^2 \rightarrow 0$, so $u_n \rightarrow u_0$ and $u_0$ is a constant.
Since $\|u_n\|\geq K>0$, we know that $u_0 \neq 0$. Finally, proceeding as in the proof of Proposition \ref{p3}, we see that $u_0=c_2$.
\end{proof}

\begin{rem}{\rm
Let $a_n \rightarrow a$ in $L^{\infty}(\Omega)$ with $a_n^+ \not \equiv 0$ for every $n$, $a \not \equiv 0$ and $a \leq 0$. Arguing as in the proof of Lemma \ref{lbb}, we may show that if $\tilde{a} \in L^{\infty}(\Omega)$ is such that $\int_\Omega \tilde{a} <0$ and $\tilde{a} \geq a_n$ for every $n$, then $\lambda_{a_n} \geq \lambda_{\tilde{a}}$ and $\lambda_s(a_n) \geq \lambda_s(\tilde{a})$ for every $n$. So $u_{2,\lambda,a_n}$ exists for $0<\lambda<\min\{\lambda_s(\tilde{a}), \lambda_{\tilde{a}}\}$ and every $n$. Moreover, since $a_n^+ \rightarrow 0$ in $L^{\infty}(\Omega)$, from \eqref{baaz} we have $\|u_{2,\lambda,a_n}\|\to \infty$  for $0<\lambda<\min\{\lambda_s(\tilde{a}), \lambda_{\tilde{a}}\}$. Finally, getting back to \eqref{baaaz}, we deduce that $\int_\Omega a_n u_{2,\lambda,a_n}^p \rightarrow \infty$ and therefore $\|u_{2,\lambda,a_n}\|_{\mathcal{C}(\overline{\Omega})}\to \infty$ for $0<\lambda<\min\{\lambda_s(\tilde{a}), \lambda_{\tilde{a}}\}$. }
\end{rem}

\begin{prop}
\label{pewi}
Assume $a^+ \not \equiv 0$ and $\int_{\Omega} a<0$. If $\lambda_n \to 0^+$ then, up to a subsequence, there holds
$w_n:=\lambda_n^{\frac{1}{p-2}}u_{2,\lambda_n} \rightarrow w_\infty$ in $H^1(\Omega)$, where $w_\infty$ is a nontrivial non-negative solution of the  problem
\begin{equation}
\label{eqwi}
-\Delta w = a(x)w^{p-1} \quad\mbox{in} \ \Omega, \qquad
\frac{\partial w}{\partial \mathbf{n}} = 0 \quad \mbox{on} \
\partial \Omega.
\end{equation}

\end{prop}

\begin{proof}
We claim that $(w_n)$ is bounded in $H_0^1(\Omega)$. Indeed, note that 
$w_n$ minimizes $J_{\lambda_n}$ over $\mathcal{M}_{\lambda_n}^- \cap A^+$, where 
$$J_\lambda(w)=\frac{1}{2}E_\lambda(w) - \frac{1}{q}\lambda^{\frac{p-q}{p-2}} B(w)-\frac{1}{p}A(w) \quad \text{for } w\in H^1(\Omega),$$
and $\mathcal{M}_\lambda$ is the Nehari manifold associated to $J_\lambda$. If $w \in \mathcal{M}_\lambda^- \cap A^+$ and $0<\lambda<\min\{\lambda_a,\lambda_s\}$ then
$$J_\lambda(w)=\left(\frac{1}{2}-\frac{1}{p}\right)E_\lambda(w)-\left(\frac{1}{q}-\frac{1}{p}\right)\lambda^{\frac{p-q}{p-2}}B(w) \geq C_0\|w\|^2 -C_1\|w\|^q$$ for some constants $C_0,C_1>0$. If we prove that $J_{\lambda_n}(w_n)$ is bounded from above then we deduce that $(w_n)$ is bounded. We have
$$J_{\lambda_n}(w_n)=\inf_{\mathcal{M}_{\lambda_n}^- \cap A^+} J_{\lambda_n} \leq \inf_{\mathcal{M}_{\lambda_n}^- \cap A^+ \cap H_0^1(\Omega)} J_{\lambda_n}.$$
If $w \in H_0^1(\Omega)$ then $B(w)=0$, so $J_\lambda(w)=\frac{1}{2}E_\lambda(w) -\frac{1}{p}A(w)$, and it can be shown that $\displaystyle \inf_{\mathcal{M}_{\lambda}^- \cap A^+ \cap H_0^1(\Omega)} J_{\lambda}$ is achieved  for $\lambda \in (0,\lambda_1^D(m))$, where $\lambda_1^D(m)$ is the first positive eigenvalue of  $$-\Delta u =\lambda m(x) u \text{ in } \Omega, \quad u=0 \text{ on }\partial \Omega.$$
Finally, we claim that the latter infimum is bounded from above for $\lambda \in (0,\lambda_1^D(m))$, which yields the conclusion.
This claim follows from the inequality
$$J_\lambda(w) \leq L(w):=\frac{C}{2}\|w\|^2 - \frac{1}{p}A(w),$$ which holds for $w \in H_0^1(\Omega)$, $\lambda \in (0,\lambda_1^D(m))$ and some $C>0$. Thus, given $w \in H_0^1(\Omega) \cap A^+$, if $J_\lambda(tw)$ achieves its global maximum at $t_0>0$ then $$J_\lambda(t_0w)\leq L(t_0w)\leq \left(\frac{1}{2}-\frac{1}{p}\right) \left(\frac{C\|w\|^2}{A(w)}\right)^{\frac{1}{p-2}}C\|w\|^2.$$
Therefore, fixing a $w_0 \in H_0^1(\Omega) \cap A^+$, we obtain
$$\inf_{\mathcal{M}_{\lambda}^- \cap A^+ \cap H_0^1(\Omega)} J_{\lambda}\leq K:=\left(\frac{1}{2}-\frac{1}{p}\right) \left(\frac{C\|w_0\|^2}{A(w_0)}\right)^{\frac{1}{p-2}}C\|w_0\|^2,$$ for $\lambda \in (0,\lambda_1^D(m))$, as claimed.
Thus $(w_n)$ is bounded in $H^1(\Omega)$, so up to a subsequence we have $w_n \rightharpoonup w_\infty$ in $H^1(\Omega)$. Taking $v=w_n-w_\infty$ in
\begin{equation}
\label{ewn2}
\int_{\Omega} \left(\nabla w_n \nabla v -\lambda_n m(x)w_n v -a(x)w_n^{p-1}v\right) -\lambda_n^{\frac{p-q}{p-2}}\int_{\partial \Omega} b(x)w_n^{q-1}v =0 \quad \forall v \in H^1(\Omega)
\end{equation}
and letting $n \to \infty$ 
we get $\lim \int_{\Omega} \nabla w_n \nabla (w_n-w_\infty) =0$, so that $w_n \rightarrow w_\infty$ in $H^1(\Omega)$. Furthermore,
since $$C_0\|w_n\|^2\leq E_{\lambda_n}(w_n)<\frac{p-q}{2-q}A(w_n)\leq C_1\|w_n\|^p$$ for some $C_0,C_1>0$, we get $\|w_n\|\geq C^{\frac{1}{p-2}}$ for some $C>0$, so that $w_\infty \not \equiv 0$.
Finally, \eqref{ewn2} also shows that $w_\infty$ is a solution of  \eqref{eqwi}.
\end{proof}

\bigskip

\section{Proofs of the main results}
\label{sec5}

Before proceeding to the proofs of our main results, we prove a partial positivity result on the boundary for nontrivial non-negative solutions of $(P_\lambda)$:

\begin{prop} \label{prop:positive}
\strut
\begin{enumerate}
\item Let $u_\lambda$ be a nontrivial non-negative solution of $(P_\lambda)$ for $\lambda > 0$. Then the set $\{ x \in \partial \Omega : u(x) = 0 \}$ has no interior points in the relative topology of $\partial \Omega$, and it is contained in $\{ x \in \partial \Omega : b(x) \leq 0 \}$ if $b\in \mathcal{C}(\partial \Omega)$.\\ 
\item Let $w_0$ be a nontrivial non-negative solution of \eqref{w0}. Then $w_0 > 0$ in $\Omega$, the set $\{ x\in \partial \Omega : w_0 = 0 \}$ has no interior points in the relative topology of $\partial \Omega$, and it is contained in $\{ x \in \partial \Omega : b(x) \leq 0 \}$ if $b\in \mathcal{C}(\partial \Omega)$.
\end{enumerate}
\end{prop}

\begin{proof}
\strut
\begin{enumerate}
\item  Assume by contradiction that $x_0$ is an interior point of $\partial \Omega$ with $u_\lambda (x_0)=0$. Then, there exists $\rho_0 > 0$ such that $u_\lambda (x)=0$ for $x \in \Gamma_1 :=  B_{\rho_0} (x_0) \cap \partial \Omega$. Let $D$ be a subdomain of $\Omega$ with smooth boundary $\partial D$ such that $\Gamma_1 \subset \partial D$ and $\Gamma_0 := \partial D \setminus \overline{\Gamma_1} = \partial D \cap \Omega$. Consider the following mixed problem 
\begin{align} \label{p.v}
\begin{cases}
-\Delta u = \lambda (-m_\infty u  - a_\infty u^{p-1}) & \mbox{in $D$}, \\ 
\frac{\partial u}{\partial \mathbf{n}} = \lambda K_\infty u & \mbox{on $\Gamma_1$}, \\
u=0 & \mbox{on $\Gamma_0$}, 
\end{cases}
\end{align}
where $m_\infty = \Vert m \Vert_\infty$, $a_\infty = \Vert a \Vert_\infty > 0$, and $K_\infty > 0$ is a constant to be determined.   Arguing as in the proof of \cite[Theorem 1]{GRS09}, we can prove that if $K_\infty$ is sufficiently large then \eqref{p.v}  has a unique nontrivial non-negative weak solution $v_\lambda \in H_{\Gamma_0}^1(D)$. Here, $H_{\Gamma_0}^1(D)$ is defined as the closure of $\mathcal{C}_c^{\infty}\left(D \cup (\partial D \setminus \overline{\Gamma_0})\right)$ with respect to the $H^1(D)$ norm.
We remark that $v_\lambda \in \mathcal{C}^2(D \cup \Gamma_1) \cap \mathcal{C}(\overline{D})$ \cite{ADN,Stam}, so that $v_\lambda > 0$ in $D \cup \Gamma_1$ by the strong maximum principle and the boundary point lemma. On the other hand, 
we have $u_\lambda > 0$ on $\Gamma_0$, and for any $\varphi \in H_{\Gamma_0}^1(D)$ satisfying $\varphi \geq 0$ there holds
\begin{align*}
& \int_D \nabla u_\lambda \nabla \varphi -\lambda \int_D \left( -m_\infty u_\lambda - a_\infty u_\lambda^{p-1} \right) \varphi - \lambda \int_{\Gamma_1} K_\infty u_\lambda \varphi \geq 0, 
\end{align*}
since $u_\lambda = 0$ on $\Gamma_1$. Hence, by Proposition \ref{app:prop:comparison}, we deduce that $v_\lambda \leq u_\lambda$ in $\overline{D}$. Thus $u_\lambda (x_0) = 0 < v_\lambda (x_0)$, and a contradiction follows. 

The second assertion can be verified in a similar way. We assume that $u_\lambda (x_0)=0$ but $b(x_0) > 0$ for some $x_0 \in \partial \Omega$. Then there exist $\rho_0, b_0 > 0$ such that $b(x)\geq b_0$ for $x \in \Gamma_1 :=  B_{\rho_0} (x_0) \cap \partial \Omega$. Setting $D, \Gamma_0$ as above, we consider the following mixed problem 
\begin{align} \label{p.vv}
\begin{cases}
-\Delta u = \lambda (-m_\infty u  - a_\infty u^{p-1}) & \mbox{in $D$}, \\ 
\frac{\partial u}{\partial \mathbf{n}} = \lambda b_0 u^{q-1} & \mbox{on $\Gamma_1$}, \\
u=0 & \mbox{on $\Gamma_0$}. 
\end{cases}
\end{align}
By direct computations, we have 
\begin{align*}
\int_D \nabla u_\lambda \nabla \varphi - \lambda \int_D \left( 
-m_\infty u_\lambda - a_\infty u_\lambda^{p-1} \right) \varphi 
- \lambda \int_{\Gamma_1} b_0 u_\lambda^{q-1} \varphi \geq 0
\end{align*}
for any $\varphi\in H_{\Gamma_0}^1(D)$ satisfying $\varphi \geq 0$. Moreover, 
we have $u_\lambda > 0$ in $\Gamma_0 \cup D$, and $u_\lambda \in \mathcal{C}^\theta (\overline{D})$ for some $\theta \in (0,1)$. On the other hand, associated with \eqref{p.vv}, we consider the following eigenvalue 
problem. 
\begin{align} \label{p.e}
\begin{cases}
-\Delta \phi = \lambda (-m_\infty) \phi  + \sigma \phi & \mbox{in $D$}, \\ 
\frac{\partial \phi}{\partial \mathbf{n}} = \lambda K \phi & \mbox{on $\Gamma_1$}, \\
\phi = 0 & \mbox{on $\Gamma_0$}. 
\end{cases}
\end{align}
We note that  if  $K>0$ is sufficiently large then for every $\lambda > 0$ the above problem has a negative first eigenvalue $\sigma_1$, cf. \cite{GRS09}. Let $\phi_1$ be the positive eigenfunction associated to $\sigma_1$ with $\Vert \phi_1 \Vert_\infty = 1$. Since $\phi_1 \in \mathcal{C}^2(\Omega \cup \Gamma_1) \cap \mathcal{C}(\overline{D})$, by the strong maximum principle and the boundary point lemma, we have $\phi_1 > 0$ in $D \cup \Gamma_1$. By direct computations, if $\varepsilon > 0$ is a constant then, for any $\varphi \in H_{\Gamma_0}^1(D)$ satisfying $\varphi \geq 0$, we have 
$$
\int_D \nabla (\varepsilon \phi_1) \nabla \varphi 
- \lambda \int_D \left( -m_\infty \varepsilon \phi_1 - a_\infty (\varepsilon \phi_1)^{p-1} \right) \varphi 
- \lambda \int_{\Gamma_1} b_0 (\varepsilon \phi_1)^{q-1} \varphi \leq 0,
$$ provided that 
\begin{align*}
0< \varepsilon \leq \min \left\{ \left( \frac{-\sigma_1}{\lambda a_\infty} \right)^{\frac{1}{p-2}}, \ \left( \frac{b_0}{K} \right)^{\frac{1}{2-q}} \right\}. 
\end{align*}
Applying Proposition \ref{app:prop:comparison} to \eqref{p.vv}  with $u=\varepsilon \phi_1$ and $v=u_\lambda$, we obtain $\varepsilon \phi_1 \leq u_\lambda$ in $\overline{D}$. However, we have $u_\lambda (x_0)=0< \varepsilon \phi_1(x_0)$, which is a contradiction.\\

\item First of all, by the weak maximum principle, we have $w_0>0$ in $\Omega$. We argue now as in the previous item to deduce the positivity result on $\partial \Omega$.  As a matter of fact, it suffices to consider \eqref{p.v} replaced by the problem 
\begin{align} \label{prob:0}
\begin{cases}
-\Delta u = \lambda (w_0 u - u^2) & \mbox{in $D$}, \\
\frac{\partial u}{\partial \mathbf{n}} = 0 & \mbox{on $\Gamma_1$}, \\
u= 0 & \mbox{on $\Gamma_0$}, 
\end{cases}
\end{align}
and note that \eqref{prob:0} has a unique nontrivial non-negative weak solution for $\lambda > 0$ large. Note also that if $b$ is continuous and $b(x_0)>0=w(x_0)$ then we can apply the same argument to reach a contradiction, so that $\{ x\in \partial \Omega : w_0 = 0 \} \subset \{ x \in \partial \Omega : b(x) \leq 0 \}$. 
\end{enumerate}
\end{proof}

\begin{rem}
If $u_\lambda$ is a nontrivial non-negative solution of $(P_\lambda)$ for $\lambda < 0$ then the assertions of Proposition \ref{prop:positive} (1) hold true replacing $\{ x \in \partial \Omega : b(x) \leq 0 \}$ by $\{ x \in \partial \Omega : b(x) \geq 0 \}$. Indeed, if $\lambda < 0$ then, by the change of variables $\mu = -\lambda$, $(P_\lambda)$ reduces to 
\begin{align*}
\begin{cases}
-\Delta u= \mu \left( (-m) u + (-a)|u|^{p-2}u \right) 
& \mbox{in $\Omega$},  \\ 
\frac{\partial u}{\partial \bf{n}}= \mu (-b) |u|^{q-2}u 
& \mbox{on $\partial \Omega$}. 
\end{cases}
\end{align*}
with $\mu>0$.
\end{rem}
\bigskip

We deduce now our existence results using the fact that local minimizers of $I_\lambda$ constrained to $\mathcal{N} \setminus \mathcal{N}_0$ are critical points of $I_\lambda$ and, therefore, solutions of $(P_\lambda)$. It is clear that $A^+$, $B^+$ and $E_{\lambda}^-$ are open sets, so that, whenever achieved, the infima of $I_\lambda$ constrained to $\mathcal{N}_\lambda^+ \cap B^+$, $\mathcal{N}_\lambda^+ \cap E_{\lambda}^-$ and $\mathcal{N}_\lambda^- \cap A^+$ provide solutions of $(P_\lambda)$.\\

\noindent {\it Proof of Theorem \ref{t1}:}\\

By Propositions \ref{p1} and \ref{pal1}, we have $\lambda_b, \lambda_s>0$ if $\int_{\partial \Omega} b<0$ and $\lambda_a, \lambda_s>0$ if $\int_{\Omega} a<0$. Moreover, if $\int_{\partial \Omega} b<0$ then Proposition \ref{pmn1} yields that $\displaystyle \inf_{\mathcal{N}_\lambda^+ \cap B^+} I_{\lambda}$ is achieved by  $u_{0,\lambda}\geq 0$ for $0<\lambda<\min\{\lambda_b,\lambda_s\}$ and $u_{0,\lambda} \rightarrow 0$  in $\mathcal{C}^{\theta}(\overline{\Om})$ for some $\theta \in (0,1)$ as $\lambda \to 0^+$. Likewise, if $\int_{\Omega} a<0$ then, by Proposition \ref{pmn2}, $\displaystyle \inf_{\mathcal{N}_\lambda^- \cap A^+} I_{\lambda}$ is achieved by $u_{2,\lambda}\geq 0$ for $0<\lambda<\min\{\lambda_a,\lambda_s\}$. Furthermore, by Propositions \ref{pasyc} and \ref{pewi}, the asymptotic profiles of $u_{0,\lambda}$ and $u_{2,\lambda}$ are given by $\lambda^{\frac{1}{q-2}} w_0$ and $\lambda^{\frac{1}{p-2}} w_\infty$ as $\lambda \to 0^+$, where $w_0$ and $w_\infty$  are nontrivial non-negative solutions of \eqref{pwl} and \eqref{eqwi}, respectively.

By a standard bootstrap argument, we obtain $w_\infty \in W^{2,r}(\Omega)$, with $r>N$. The strong maximum principle and boundary point lemma yield $w_\infty > 0$ in $\overline{\Omega}$. Setting $w_\lambda = \lambda^{\frac{1}{p-2}} u_{2,\lambda}$, we have that $w_\lambda$ is bounded in $H^1(\Omega)$ and $w_\lambda$ is a weak solution of the problem
$$
\begin{cases}
-\Delta w = \lambda m w + a w^{p-1} & \text{in} \ \ \Omega, \\ 
\frac{\partial w}{\partial \mathbf{n}} = \lambda^{\frac{p-q}{p-2}}b w^{q-1} & \text{on} \ \ \partial \Omega. 
\end{cases}
$$
Rossi's bootstrap argument \cite{R} yields that $w_\lambda$ is bounded in $\mathcal{C}^\nu (\overline{\Omega})$ for some $\nu \in (0,1)$. By the compact embedding $\mathcal{C}^\nu (\overline{\Omega}) \subset \mathcal{C}^{\theta}(\overline{\Omega})$, $\theta<\nu$, we may obtain that $w_\lambda$ converges to some $w^*$ in $\mathcal{C}^{\theta}(\overline{\Omega})$. Since $w_\lambda \to w_\infty$ in $L^2(\Omega)$, we have $w^* = w_\infty$. From $w_\infty > 0$ in $\overline{\Omega}$, we have $u_{2,\lambda} > 0$ in $\overline{\Omega}$ for $\lambda > 0$ close to $0$, and $\min_{\overline{\Omega}}u_{2,\lambda} \to \infty$ as $\lambda \to 0^+$. It can be verified in a similar way that $\lambda^{-\frac{1}{2-q}}u_{0,\lambda} \to w_0$ in $C^{\tilde{\theta}}(\overline{\Omega})$ for some $\tilde{\theta}\in (0,1)$. Indeed, it suffices to note that $\lambda^{-\frac{1}{2-q}}u_{0,\lambda}$ is a weak solution of the problem
\begin{align*}
\begin{cases}
-\Delta w = \lambda m w + \lambda^{\frac{p-q}{2-q}} a w^{p-1} & \text{in} \ \ \Omega, \\ 
\frac{\partial w}{\partial \mathbf{n}} = b w^{q-1} & \text{on} \ \ \partial \Omega. 
\end{cases}
\end{align*}
Finally, Proposition \ref{prop:positive} provides the positivity properties of $w_0$.\qed \\



\noindent {\it Proof of Theorem \ref{t1'}:}\\

By Corollary \ref{c2},  $\displaystyle \inf_{\mathcal{N}_\lambda^+ \cap E_\lambda^-} I_{\lambda}$ is achieved by $u_{1,\lambda}\geq 0$ for $0<\lambda<\min\{\lambda_a,\lambda_b\}$. Moreover, by Proposition \ref{p3}, $u_{1,\lambda} \to c_2$ in $\mathcal{C}^{\theta}(\overline{\Om})$ for some $\theta \in (0,1)$ as $\lambda \to 0^+$. 

Now, if $\lambda<0$ then we change the signs of  $\lambda$, $m$, $a$ and $b$. Since
$$\int_{\Omega} (-m)<0<\int_{\Omega} (-a) \quad \text{and} \quad 0<\int_{\partial \Omega} (-b)<K_1(m,a),$$
Propositions \ref{pmn1} and \ref{pmn2} yield the existence of two non-negative solutions $u_{0,-\lambda}$, $u_{2,-\lambda}$ for $0<-\lambda<\min\{\tilde{\lambda}_1, \tilde{\lambda}_s\}$, which satisfy $u_{0,-\lambda} \to c_1$ and $u_{2,-\lambda} \to c_2$ in $\mathcal{C}^{\theta}(\overline{\Om})$ for some $\theta \in (0,1)$ as $\lambda \to 0^-$. We set then $v_{1,\lambda}=u_{0,-\lambda}$ and $v_{2,\lambda}=u_{2,-\lambda}$.\qed \\

\noindent {\it Proof of Theorem \ref{t3b}:}\\  

First of all, by a standard bootstrap argument we infer that $u_{2,\lambda}$ is a classical positive solution of $(P_\lambda)$, since $u_{2,\lambda} > 0$ in $\overline{\Omega}$. In order to prove that $u_{2,\lambda}$ is unstable, we consider the following linearized eigenvalue problem at $u_{2,\lambda}$ with an eigenvalue $\gamma$:
\begin{align}  \label{eigenprob}
\begin{cases}
\mathcal{L}_\lambda \psi = \lambda (p-1) a u_{2,\lambda}^{p-2}\psi + \gamma \psi & \text{in} \ \ \Omega, \\
\frac{\partial \psi}{\partial \mathbf{n}} = \lambda (q-1) b u_{2,\lambda}^{q-2} \psi + \gamma \psi & \text{on} \ \ \partial \Omega, 
\end{cases}
\end{align}
where $\mathcal{L}_\lambda = -\Delta - \lambda m$. Let us denote by $\gamma_1$ its smallest eigenvalue and by $\psi_1 \in \mathcal{C}^{2+\alpha}(\overline{\Omega})$ an eigenfunction associated to $\gamma_1$ which is positive in $\overline{\Omega}$. We claim that $\gamma_1 < 0$. To this end, we use Picone's identity \cite{BCN95}. By direct computations, we have
\begin{align*}
\int_\Omega \left( \frac{u_{2,\lambda}}{\psi_1} \right) \sum_j \frac{\partial}{\partial x_j} \psi_1^2 \frac{\partial}{\partial x_j} \left( \frac{u_{2,\lambda}}{\psi_1}\right) &= \int_\Omega \left( \frac{u_{2,\lambda}}{\psi_1} \right) (-\mathcal{L}_\lambda u_{2,\lambda}\psi_1 + u_{2,\lambda} \mathcal{L}_\lambda \psi_1) \\
& = \lambda (p-2)A(u_{2,\lambda}) + \gamma_1 \int_\Omega u_{2,\lambda}^2. 
\end{align*}
On the other hand, by Green's formula we have 
\begin{align*}
\int_\Omega \left( \frac{u_{2,\lambda}}{\psi_1} \right) \sum_j \frac{\partial}{\partial x_j} \psi_1^2 \frac{\partial}{\partial x_j} \left( \frac{u_{2,\lambda}}{\psi_1}\right) & = - \int_\Omega \psi_1^2 \left\vert \nabla \left( \frac{u_{2,\lambda}}{\psi_1} \right)\right\vert^2 + \int_{\partial \Omega} \left( \frac{u_{2,\lambda}}{\psi_1} \right) \psi_1^2 \frac{\partial}{\partial \mathbf{n}} \left( \frac{u_{2,\lambda}}{\psi_1} \right) \\
& =  - \int_\Omega \psi_1^2 \left\vert \nabla \left( \frac{u_{2,\lambda}}{\psi_1} \right)\right\vert^2 + \lambda (2-q) B(u_{2,\lambda}) - \gamma_1 \int_{\partial \Omega} u_{2,\lambda}^2. 
\end{align*}
Hence, 
\begin{align*}
\gamma_1 = \frac{- \int_\Omega \psi_1^2 \left\vert \nabla \left( \frac{u_{2,\lambda}}{\psi_1} \right)\right\vert^2 - \lambda (p-2) A(u_{2,\lambda}) + \lambda (2-q) B(u_{2,\lambda})}{\int_\Omega u_{2,\lambda}^2 + \int_{\partial \Omega} u_{2,\lambda}^2}. 
\end{align*}
Since $u_{2,\lambda} \in \mathcal{N}_\lambda$, we have $\lambda B(u_{2,\lambda}) = E_\lambda (u_{2,\lambda}) - \lambda A(u_{2,\lambda})$. So, it follows that 
\begin{align*}
\gamma_1 = \frac{- \int_\Omega \psi_1^2 \left\vert \nabla \left( \frac{u_{2,\lambda}}{\psi_1} \right)\right\vert^2 + (2-q) E_\lambda (u_{2,\lambda}) - \lambda (p-q) A(u_{2,\lambda})}{\int_\Omega u_{2,\lambda}^2 + \int_{\partial \Omega} u_{2,\lambda}^2}. 
\end{align*}
Since $u_{2,\lambda} \in \mathcal{N}_\lambda^-$, we have $E_\lambda (u_{2,\lambda})< \lambda \left( \frac{p-q}{2-q}\right) A(u_{2,\lambda})$, and hence, $\gamma_1 < 0$, as desired. The proof of Theorem \ref{t3b} is complete. \qed \\


\noindent {\it Sketch of the proof of Remark \ref{rw0}:}\\

Let us assume \eqref{assump:mbreg} and that $w_0$ is a classical positive solution of \eqref{w0}. In the same way as in the proofs of Theorems \ref{t1} and \ref{t3b},  we infer that, for $\lambda > 0$ sufficiently small, $u_{0,\lambda}$ is a classical positive solution of $(P_\lambda)$. In order to discuss the stability of $u_{0,\lambda}$, we replace $u_{2,\lambda}$ by $u_{0,\lambda}$ in \eqref{eigenprob} and analyze the sign of $\gamma_1 = \gamma_1(\lambda)$. Let $\psi_1 = \psi_1(\lambda)$ be the unique positive eigenfunction associated to $\gamma_1$ and satisfying $\int_{\partial \Omega} \psi_1^2 + \int_\Omega \psi_1^2 = 1$. Setting $\xi_\lambda (x) = \lambda m(x) + \lambda (p-1) a(x) u_{0,\lambda} (x)^{p-2}$ and $\eta_\lambda (x) = \lambda (q-1) b(x) u_{0,\lambda} (x)^{q-2}$, we observe from Theorem \ref{t1} that $\xi_\lambda \to 0$ in $L^2(\Omega)$ and $\eta_\lambda \to (q-1) b w_0^{q-2}$ in $H^1(\Omega)$ as $\lambda \to 0^+$, where $b$ is understood as an extension to $\mathcal{C}^{1+\alpha}(\overline{\Omega})$. By the continuity of $\gamma_1, \psi_1$ with respect to $\lambda$, we get to a limiting eigenvalue problem as $\lambda \to 0^+$, namely: 
\begin{align*}
\begin{cases} 
-\Delta \psi_1(0) = \gamma_1(0) \psi_1(0) & \mbox{in} \ \ \Omega, \\ 
\frac{\partial \psi_1(0)}{\partial \mathbf{n}} = (q-1)b(x) w_0^{q-2} \psi_1(0) + \gamma_1(0) \psi_1(0) & \mbox{on} \ \ \partial \Omega. 
\end{cases}
\end{align*}
By Green's formula, we have 
\begin{align} \label{eqlam0}
\int_\Omega |\nabla \psi_1(0)|^2 - \int_{\partial \Omega} (q-1)b w_0^{q-2} \psi_1(0)^2 = \gamma_1(0). 
\end{align}

Now, we claim that $\gamma_1(0) > 0$. Once this is verified, by the continuity of $\gamma_1$ we conclude that $\gamma_1(\lambda) > 0$ for $\lambda > 0$ sufficiently small, and the proof is complete. Using Green's formula again, we see that
\begin{align*}
0 = \int_\Omega (-\Delta w_0) \frac{\psi_1(0)^2}{w_0} = - \int_\Omega \left\vert \frac{\psi_1(0)}{w_0} \nabla w_0 - \nabla \psi_1(0) \right\vert^2 + \int_\Omega |\nabla \psi_1(0)|^2 - \int_{\partial \Omega}b w_0^{q-2}\psi_1(0)^2,  
\end{align*}
which combined with \eqref{eqlam0} yields
\begin{align*}
\gamma_1(0) & = (q-1) \int_\Omega \left\vert \frac{\psi_1(0)}{w_0} \nabla w_0 - \nabla \psi_1(0) \right\vert^2 + (2-q) \int_\Omega |\nabla \psi_1(0)|^2 \\
& \geq (2-q) \int_\Omega |\nabla \psi_1(0)|^2 > 0, 
\end{align*}
since $\psi_1(0)$ is not a constant. \qed \\

\noindent {\it Proof of Theorem \ref{t4}:}\\

Let 
\begin{align*}
\mu_{1,\pm} = \inf\left\{ \int_{D_{\pm}} \left( |\nabla u|^2 - \lambda m u^2 \right) ; \, u \in H^1_0(D_{\pm}), \, \int_{D_\pm} u^2 = 1 \right\}
\end{align*}
be the unique positive principal eigenvalues of the Dirichlet eigenvalue 
problems 
\begin{align*}
\begin{cases}
-\Delta u = \lambda m u + \mu_{\pm} u & \mbox{in} \ D_{\pm}, \\
u = 0 & \mbox{on} \ \partial D_{\pm}, 
\end{cases}
\end{align*}
and let $u_{1,\pm}$ denote the corresponding positive eigenfunctions in $H^1_0(D_{\pm})$, respectively. By a standard regularity argument and the strong maximum principle, it follows that $u_{1,\pm} \in W^{2,r}(D_{\pm})$ for any $r>N$, and $u_{1,\pm} > 0$ in $D_{\pm}$. Then, by Green's formula, we deduce 
\begin{align*}
\int_{D_{\pm}} \nabla u_{1,\pm} \nabla v - \int_{\partial D_{\pm}} \frac{\partial u_{1,\pm}}{\partial \mathbf{n}} v = \lambda \int_{D_{\pm}} mu_{1,\pm}v + 
\mu_{1,\pm} \int_{D_{\pm}} u_{1,\pm}v \quad \mbox{for all} \ v \in 
\mathcal{C}^1(\overline{D_{\pm}}). 
\end{align*}
On the other hand, for any nontrivial non-negative solution $u \in H^1(\Omega)$ of $(P_\lambda)$, we have 
\begin{align*}
\int_\Omega \nabla u \nabla w - \lambda \int_\Omega muw - \lambda \int_\Omega au^{p-1}w - \lambda \int_{\partial \Omega}bu^{q-1}w = 0 \quad \forall \ w \in H^1(\Omega),
\end{align*}
and recall that $u \in \mathcal{C}^\alpha (\overline{\Omega}) \cap W^{2,r}_{\rm loc}(\Omega)$, $r>N$, and $u>0$ in $\Omega$. 

Now, we consider $v=u$ and $w=\widetilde{u_{1,\pm}}$, where  
\begin{align*}
\widetilde{u_{1,\pm}} = \left\{ 
\begin{array}{ll}
u_{1,\pm} & \mbox{in} \ D_{\pm}, \\
0, & \mbox{otherwise}, 
\end{array} \right. 
\end{align*}
and then observe that 
\begin{align*}
0< -\int_{\partial D_{\pm}} \frac{\partial u_{1,\pm}}{\partial \mathbf{n}}u = \mu_{1,\pm} \int_{D_{\pm}} uu_{1,\pm} - \lambda \int_{D_{\pm}} a u^{p-1}u_{1,\pm}, 
\end{align*}
since $u_{1,\pm}\in \mathcal{C}^1(\overline{D_{\pm}})$,  $u_{1,\pm}>0$ in $D_{\pm}$, and $\frac{\partial u_{\pm,1}}{\partial \mathbf{n}}< 0$ from the boundary point lemma (cf. \cite{V}). Hence, it follows that $\mu_{1,+} > 0$ if $\lambda > 0$, and also that $\mu_{1,-} > 0$ if $\lambda < 0$. Since $m$ changes sign in $D_{\pm}$, we have $\mu_{1,\pm} < 0$ for $|\lambda|>\Lambda$ if $\Lambda$ is sufficiently large. Thus, we obtain $|\lambda|\leq \Lambda$. \qed

\bigskip

\subsection{Bifurcating solutions} \label{sec6}
\strut
\medskip


In this final subsection we prove Theorems \ref{t3} and \ref{thm:a=-m:bif} by a bifurcation technique. Since this technique does not require a variational structure for $(P_\lambda)$, the next results hold under the condition 
\begin{align} 
1<q<2<p. \label{1<q<2<p} 
\end{align}
We use the usual orthogonal decomposition $L^2(\Omega) = \R \oplus V$, where $$V = \left\{ v \in L^2(\Omega) : \int_\Omega v = 0 \right\},$$ and the projection $Q : L^2(\Omega) \to V$ given by $$v = Qu = u - \frac{1}{|\Omega|}\int_\Omega u.$$ In this way we reduce the problem of finding a classical positive solution to $(P_\lambda)$ under (\ref{assump:mbreg}) to the following two problems 
	\begin{align}  \label{1stpro}
	\begin{cases}
	-\Delta v + \frac{\lambda}{|\Omega|} 
	\int_{\partial \Omega} b(x) g(t + v) 
	= \lambda Q f(x, t + v) & \mbox{in $\Omega$}, \\
	\dfrac{\partial v}{\partial \mathbf{n}}= \lambda b(x) g(t + v) & 
	\mbox{on $\partial \Omega$}, 
	\end{cases} 
	\end{align}
	\begin{equation} \label{2ndpro}
	\lambda \left( \int_\Omega f(x, t + v) 
	+ \int_{\partial \Omega} b g(t + v) \right) = 0, 
	\end{equation}
where $$t=\frac{1}{|\Omega|}\int_\Omega u, \quad v = Qu = u - t,\quad 
 f(x, u) = m(x)u + a(x) u^{p-1}, \quad \text{and} \quad g(u) = u^{q-1}.$$
First, to solve (\ref{1stpro}) in the H\"older space 
$\mathcal{C}^{2+\alpha}(\overline{\Omega})$, we set 
$$
X = \left\{ v \in \mathcal{C}^{2+\alpha}(\overline{\Omega}) 
: \int_\Omega v = 0 \right\}
$$ 
and introduce the nonlinear mapping $F: \R \times \R \times 
X \to Z$ given by  
	$$
	F(\lambda, t, v) = \left( -\Delta v - \lambda Q f(x, t + v) + \frac{\lambda}{|\Omega|} \int_{\partial \Omega} b g(t + v), \; 
	\frac{\partial v}{\partial \mathbf{n}} - \lambda b g(t + v) \right), 
	$$
where 
$$
Z = \left\{ (\phi, \psi) \in \mathcal{C}^{\alpha}(\overline{\Omega}) \times 
\mathcal{C}^{1+\alpha}(\partial \Omega) : 
\int_\Omega \phi + \int_{\partial \Omega} \psi = 0 \right\}.  
$$
The Fr\'echet derivative of $F$ with respect to $v$ at $(0,c,0)$ is given by $F_v (0, c, 0)v = (-\Delta v, \frac{\partial v}{\partial \mathbf{n}})$, where $c > 0$ is a constant.  From Banach's closed graph theorem it follows that $F_v (0, c, 0)$ is a homeomorphism. By the implicit function theorem, the set $F(\lambda, t, v) = 0$ consists exactly of an unique $\mathcal{C}^\infty$ function $v = v(\lambda, t)$ in a neighbourhood of $(\lambda, t) = (0,c)$, satisfying $v(0,c) = 0$. 
\par

Now, plugging $v(\lambda, t)$ in (\ref{2ndpro}), we obtain the bifurcation equation 
$$
\lambda \left( \int_\Omega f(x, t + v(\lambda, t)) + \int_{\partial \Omega} b g(t + v(\lambda, t)) \right) = 0. 
$$
From this equation we deduce that $\lambda = 0$ corresponds to the trivial solution $(\lambda, u)=(0, d)$ with $c-\varepsilon < d < c+\varepsilon$ for some $\varepsilon > 0$. 

Hence, the study of the set of non-trivial solutions 
for $(\lambda, u)$ close to $(0, c)$ is reduced to the consideration of the equation 
	\begin{equation}  \label{Phi:def}
	\Phi (\lambda, t) 
	:= \int_\Omega f(x, t + v(\lambda, t)) + 
	\int_{\partial \Omega} b g(t + v(\lambda, t)) = 0 
	\end{equation}
for $(\lambda, t)$ close to $(0, c)$.  

Recall that under the condition that $\int_\Omega m > 0 > \int_\Omega a$ and $\int_{\partial \Omega}b < 0$, the assumption $\int_{\partial \Omega} b > - \tilde{K}_1(m,a)$ made in (\ref{assump:tildK}) is equivalent to the existence of two positive zeros $c_1 < c_2$ of $\varphi$ in (\ref{ephi}).

Theorems \ref{t3} and \ref{thm:a=-m:bif} (1) are direct consequences of the following result. 
\begin{prop} \label{prop:bif} 
Assume  \eqref{cm}, \eqref{assump:mbreg}, \eqref{assump:tildK}, and \eqref{1<q<2<p}. Then the following two assertions hold. 
\begin{enumerate}
  \item[(1)]  $(P_\lambda)$ has two classical positive solutions $U_{j, \lambda}$, $j=1,2$, for $\lambda$ close to $0$, given by 
\begin{align*}
U_{j,\lambda} = t_j (\lambda) + v (\lambda, t_j(\lambda)).
\end{align*}
Here $t_j$ is a $\mathcal{C}^1$ function of $\lambda$ such that
$\lambda \mapsto U_{j,\lambda} \in \mathcal{C}^{2+\alpha}(\overline{\Omega})$ is a $\mathcal{C}^1$ map, $t_j(0) = c_j$, and $v(0,c_j) = 0$, $j=1,2$. Moreover, $U_{1,\lambda}$ is unstable (respect. asymptotically stable),  whereas $U_{2,\lambda}$ is asymptotically stable (respect. unstable) if $\lambda > 0$ (respect. $\lambda < 0$).
  \item[(2)] If $(P_\lambda)$ has a classical positive solution $u_\lambda$ with $\lambda \not= 0$ such that $u_\lambda \rightarrow c$ in $\mathcal{C}(\overline{\Omega})$ as $\lambda \to 0$, where $c>0$ is a constant, then  $\varphi (c) = 0$. 
\end{enumerate}
\end{prop}
\begin{proof} 
\strut
\begin{enumerate}
\item Differentiating $\Phi$ in \eqref{Phi:def} with respect to $t$ we find
\begin{align*}
\Phi_t (\lambda, t) 
= \int_\Omega f_u(x, t + v) (1 + v_t) 
+ \int_{\partial \Omega} b g'(t + v) (1 + v_t).  
\end{align*}
From (\ref{1stpro}) we have
\begin{align*}
\begin{cases}
-\Delta v_t + \frac{\lambda}{|\Omega|} \int_{\partial \Omega} b g'(t 
+ v) (1 + v_t) = \lambda Q [f_u (x, t + v) (1 + v_t)] & 
\mbox{in $\Omega$}, \\
\dfrac{\partial v_t}{\partial \mathbf{n}} = \lambda b g'(t + v) (1 + v_t) 
& \mbox{on $\partial \Omega$}. 
\end{cases}
\end{align*}
Putting $\lambda = 0$ and $t = c_j$ we get
\begin{align*}
\begin{cases}
-\Delta v_t (0,c_j)  = 0 & \mbox{in $\Omega$} \\
\dfrac{\partial v_t (0, c_j)}{\partial \mathbf{n}} = 0 & \mbox{on $\partial \Omega$}. 
\end{cases}
\end{align*}
Since $v_t (0,c_j) \in V$, we have $v_t (0,c_j) = 0$. 
It follows that 
\begin{align*}
\Phi_t (0,c_j) 
& = \int_\Omega f_u (x, c_j) + \int_{\partial \Omega} b g'(c_j) \\
& = \int_\Omega m + (p-1)c_j^{p-2} \int_\Omega a 
+ (q-1)c_j^{q-2} \int_{\partial \Omega}b. 
\end{align*}
Therefore, 
\begin{align} \label{13082338}
\frac{c_j^{2-q}}{q-1} \Phi_t (0, c_j) 
= \frac{c_j^{2-q}}{q-1} \int_\Omega m + \left( \frac{p-1}{q-1} 
\right) c_j^{p-q} \int_\Omega a + \int_{\partial \Omega}b. 
\end{align}
Now, we claim that $\Phi_t (0, c_j) \not= 0$. Once this is verified, we end the proof of Proposition \ref{prop:bif} (1) by the use of the implicit function theorem. Let $c_0 \in (c_1, c_2)$ be the global maximum point of $\varphi$. 
This one is given explicitly by 
\begin{align} \label{c0}
c_0 = \left( \frac{(2-q)\int_\Omega m}{(p-q)(-\int_\Omega a)}\right)
^{\frac{1}{p-2}}. 
\end{align}
From (\ref{13082338}) and the fact that $\varphi (c_j)=0$, 
we deduce that 
\begin{align*} 
\frac{c^{2-q}}{q-1} \Phi_t (0,c_j) 
= \left( \frac{2-q}{q-1} \right) c_j^{2-q} \int_\Omega m 
+ \left( \frac{p-q}{q-1} \right) c_j^{p-q} \int_\Omega a. 
\end{align*}
It follows from (\ref{c0}) that 
\begin{align*}
\frac{1}{q-1} \Phi_t (0,c_j) 
&= \left( \frac{2-q}{q-1} \right) \int_\Omega m 
+ \left( \frac{p-q}{q-1} \right) c_j^{p-2} \int_\Omega a 
\\
& \left\{ \begin{array}{ll}
> \left( \frac{2-q}{q-1} \right) \int_\Omega m 
+ \left( \frac{p-q}{q-1} \right) c_0^{p-2} \int_\Omega a = 0, & \quad \text{for } j=1, \\ 
< \left( \frac{2-q}{q-1} \right) \int_\Omega m 
+ \left( \frac{p-q}{q-1} \right) c_0^{p-2} \int_\Omega a = 0, & \quad \text{for } j=2. 
\end{array} \right. 
\end{align*}
The conclusion follows. 

\bigskip 

We prove now the stability results of $U_{j,\lambda}$. We recall from \eqref{eigenprob} the linearized eigenvalue problem at $U_{j,\lambda}$:
\begin{align} \label{leigenU}
\begin{cases}
-\Delta \psi = \lambda m \psi + \lambda (p-1)U_{j,\lambda}^{p-2}\psi + \gamma \psi & \mbox{in $\Omega$}, \\
\frac{\partial \psi}{\partial \mathbf{n}}= \lambda (q-1)U_{j,\lambda}^{q-2}\psi + \gamma \psi & \mbox{on $\partial \Omega$}. 
\end{cases}
\end{align}
Let $\gamma_1 = \gamma_1(\lambda)$ be the smallest positive eigenvalue of this problem and  $\psi_1 = \psi_1(\lambda)$ be the unique positive eigenfunction associated to $\gamma_1$, satisfying $\int_\Omega \psi_1^2 + \int_{\partial \Omega} \psi_1^2 = 1$. It is easy to see that $\gamma_1(0)=0$ and $\psi_1(0)=\left( \frac{1}{|\Omega|+|\partial \Omega|}\right)^{1/2}$. 
We differentiate \eqref{leigenU} with respect to $\lambda$ and let $\lambda = 0$ to obtain
\begin{align*}
\begin{cases}
-\Delta \psi_1'(0) = m \psi_1(0) + (p-1)aU_{j,0}^{p-2}\psi_1(0) + \gamma_1'(0) \psi_1(0) & \mbox{in $\Omega$}, \\
\frac{\partial \psi_1'(0)}{\partial \mathbf{n}}= (q-1)bU_{j,0}^{q-2}\psi_1(0) + \gamma_1'(0) \psi_1(0) & \mbox{on $\partial \Omega$}. 
\end{cases}
\end{align*}
By Green's formula, we have
\begin{align*}
& \int_\Omega \left( m \psi_1(0) + (p-1)aU_{j,0}^{p-2}\psi_1(0) + \gamma_1'(0) \psi_1(0) \right) \\ 
& \hspace*{3cm} + \int_{\partial \Omega}\left( (q-1)bU_{j,0}^{q-2}\psi_1(0) + \gamma_1'(0) \psi_1(0) \right) = 0. 
\end{align*}
Since $U_{j,0}=c_j$ and $\varphi (c_j) = 0$ from \eqref{ephi}, we deduce 
\begin{align} \label{gamm1U}
\gamma_1'(0) = -\frac{c_j^{q-2} \left\{ (2-q)c_j^{2-q}\int_\Omega m + (p-q)c_j^{p-q}\int_\Omega a \right\}}{|\Omega|+|\partial \Omega|}. 
\end{align}
By a direct computation, we see that $$c_j\varphi'(c_j)=(2-q)c_j^{2-q}\int_\Omega m + (p-q)c_j^{p-q}\int_\Omega a.$$
Since $\varphi (c_1)>0>\varphi_1(c_2)$, we deduce from \eqref{gamm1U} that $\gamma_1'(0)<0$ for $j=1$ and $\gamma_1'(0)>0$ for $j=2$, which combined with $\gamma_1(0)=0$ provides the desired conclusion. \\ 

\item Let $u$ be a classical positive solution of $(P_\lambda)$ with $\lambda \not= 0$. By Green's formula it follows that 
$$
\int_\Omega -\Delta u = - \lambda \int_{\partial \Omega} b u^{q-1}. 
$$
Hence we have 
$$
\lambda \left( \int_\Omega (mu +  a u^{p-1}) 
+ \int_{\partial \Omega} b u^{q-1} \right) = 0, 
$$
i.e.
$$
\int_\Omega (mu + a u^{p-1}) 
+ \int_{\partial \Omega} b u^{q-1} = 0. 
$$
Since $u \to c$ in $\mathcal{C}(\overline{\Omega})$, where $c$ is a positive constant, we obtain the desired conclusion. 
\end{enumerate}
\end{proof} 

Next, we recall that under the conditions $\int_\Omega m > 0 > \int_\Omega a$ and $\int_{\partial \Omega}b < 0$, the assumption $\int_{\partial \Omega} b = - \tilde{K}_1(m,a)$ is equivalent to the existence of a unique positive zero $c_0$ of $\varphi$, given by (\ref{c0}). In this case, $\varphi (c_0) = \varphi'(c_0) = 0$, and $c_0$ is the global maximum point of $\varphi$. 

\begin{lem} \label{lem:Phi}
Assume \eqref{assump:mbreg}, \eqref{1<q<2<p}, $\int_\Omega m > 0 > \int_\Omega a$, and $\int_{\partial \Omega}b = -\tilde{K}_1(m,a)$. Then, for $\Phi$ defined in \eqref{Phi:def}, we have the following:
\begin{enumerate}
\item $\Phi (0,c_0) = 0$.
\item $\Phi_t (0,c_0) = 0$.
\item $\Phi_{tt}(0,c_0) = - (2-q)(p-2) 
c_0^{-1} \int_\Omega m < 0$.
\item $\Phi_\lambda (0, c_0) = \int_\Omega mv_\lambda (0,c_0) + (p-1)c_0^{p-2} 
\int_\Omega a v_\lambda(0,c_0) 
+ (q-1)c_0^{q-2} \int_{\partial \Omega} 
b v_\lambda (0,c_0). $
\end{enumerate}
In particular, if $a=-km$ for some positive constant $k$, then 
\begin{align} \label{Philamposi}
\Phi_\lambda (0, c_0) = (q-1)c_0^{-1} \int_\Omega |\nabla v_\lambda(0,c_0)|^2 > 0. 
\end{align}
\end{lem} 
\begin{proof}
\strut
\begin{enumerate}
\item It is straightforward from
\begin{align*}
\Phi (0,c_0) &= \int_\Omega (mc_0 + a c_0^{p-1}) 
+ \int_{\partial \Omega} b c_0^{q-1} = c_0^{q-1} \varphi (c_0) = 0. 
\end{align*}\\
\item  We differentiate $\Phi$ with respect to $t$ to get
$$
 \Phi_t = \int_\Omega f_u(x, t + v)(1+v_t) + 
\int_{\partial \Omega} b g'(t + v) (1 + v_t). 
$$
Thus 
\begin{equation} \label{Phialpha0express}
\Phi_t (0,c_0) = \int_\Omega f_u(x,c_0)(1+v_t(0,c_0)) + \int_{\partial \Omega} b g'(c_0) (1 + v_t(0,c_0)). 
\end{equation}
Let us show how we derive $v_t (0,c_0)$ from (\ref{1stpro}). 
Differentiating (\ref{1stpro}) (with $v=v(\lambda, t)$) 
with respect to $t$, we obtain
\begin{align*}
\begin{cases}
-\Delta v_t + \frac{\lambda}{|\Omega|} \int_{\partial \Omega} b g'(t + v) (1 + v_t) = \lambda Q [f_u(x,t + v)(1+v_t)] & \mbox{in $\Omega$}, \\
\dfrac{\partial v_t}{\partial \mathbf{n}} 
= \lambda b g'(t + v)(1+v_t) & \mbox{on $\partial \Omega$}. 
\end{cases}
\end{align*}
Taking $(\lambda, t)=(0,c_0)$ we get
\begin{align*} 
\begin{cases}
-\Delta v_t(0,c_0) = 0 & \mbox{in $\Omega$}, \\
\dfrac{\partial v_t(0,c_0)}{\partial \mathbf{n}} = 0 & \mbox{on $\partial \Omega$}. 
\end{cases}
\end{align*}
Since $v_t (0,c_0) \in V$, i.e. $\int_\Omega v_t (0,c_0) = 0$, we have 
\begin{align} \label{eq140829vt0}
v_t (0,c_0)=0. 
\end{align}
Hence, from (\ref{Phialpha0express}) and $\varphi (c_0) = \varphi'(c_0)= 0$, it follows 
that 
\begin{align*}
\Phi_t (0,c_0) 
& = \int_\Omega (m + (p-1)ac_0^{p-2}) 
+ \int_{\partial \Omega} b (q-1) c_0^{q-2} \\
& = \int_\Omega (m + (p-1)a c_0^{p-2}) 
+ (q-1) c_0^{q-2} \int_{\partial \Omega} b \\
& = \int_\Omega (m + (p-1)ac_0^{p-2}) 
+ (q-1) c_0^{q-2} \left( -c_0^{p-q}\int_\Omega a - c_0^{2-q} \int_\Omega m 
\right) \\
& = (2-q) \int_\Omega m + (p-q) c_0^{p-2} \int_\Omega a = 0. 
\end{align*}\\

\item Differentiating $\Phi$ once more with respect to $t$, we have
\begin{align*}
\Phi_{tt}(0,c_0) &= \int_\Omega \{ f_{uu}(x,c_0) 
(1 + v_t (0,c_0))^2 + f_u(x,c_0) v_{tt}(0,c_0) \} \\
& \quad + \int_{\partial \Omega} b \{ g''(c_0) (1 + v_t (0,c_0))^2 
+ g'(c_0) v_{tt}(0,c_0) \}.
\end{align*}
In the same way as $v_t (0,c_0)=0$, we get $v_{tt}(0,c_0) = 0$ from 
(\ref{1stpro}). It follows that 
\begin{eqnarray*}
 \Phi_{tt}(0,c_0) 
&=& (p-1)(p-2) c_0^{p-3} \int_\Omega a + (q-1)(q-2) c_0^{q-3} 
\int_{\partial \Omega} b \\
&=& (p-1)(p-2) c_0^{p-3} \int_\Omega a + (q-1)(q-2) c_0^{q-3} 
\left( -c_0^{p-q} \int_\Omega a - c_0^{2-q} \int_\Omega m \right) \\
& =& \{ (p-1)(p-2)-(q-1)(q-2)\} c_0^{p-3}\int_\Omega a - (q-1)(q-2) c_0^{-1} 
\int_\Omega m \\
& =& c_0^{-1} \left( \{ (p-1)(p-2)-(q-1)(q-2) \} c_0^{p-2} \int_\Omega a 
- (q-1)(q-2) \int_\Omega m \right) \\
& =& c_0^{-1} \left( \{ (p-1)(p-2)-(q-1)(q-2) \} 
\frac{(2-q)\int_\Omega m}{(p-q)(-\int_\Omega a)} \int_\Omega a
- (q-1)(q-2) \int_\Omega m \right) \\
& =&c_0^{-1}\int_\Omega m \left( \frac{2-q}{p-q} \right) 
\{ (q-1)(q-2) + (p-1)(p-2) - (q-1)(p-q) \} \\
& = &c_0^{-1}\int_\Omega m\left( \frac{2-q}{p-q} \right) 
(p-2)(q-p) =
-(2-q)(p-2) c_0^{-1}\int_\Omega m  < 0,
\end{eqnarray*}
where we have used again that $\varphi(c_0)=\varphi'(c_0)=0$. \\

\item From the formula 
\begin{align*}
\Phi_\lambda = \int_\Omega f_u(x,t + v)v_\lambda 
+ \int_{\partial \Omega} b g'(t + v) v_\lambda,  
\end{align*}
it follows that 
\begin{align*}
\Phi_\lambda (0,c_0) 
&= \int_\Omega f_u(x, c_0)v_\lambda (0, c_0) 
+ \int_{\partial \Omega} b g'(c_0) v_\lambda (0,c_0) \nonumber \\
&= \int_\Omega (m + (p-1) a c_0^{p-2}) v_\lambda (0, c_0) 
+ \int_{\partial \Omega} b (q-1) c_0^{q-2} v_\lambda (0,c_0) \nonumber \\
&= \int_\Omega m v_\lambda (0,c_0) + (p-1) c_0^{p-2} \int_\Omega a 
v_\lambda (0,c_0) + (q-1)c_0^{q-2} \int_{\partial \Omega} 
b v_\lambda (0,c_0). 
\end{align*}
\\

\item From (\ref{1stpro}) we get 
\begin{align*}
\begin{cases}
-\Delta v_\lambda + \frac{1}{|\Omega|} \int_{\partial \Omega} b 
[g(t + v) + \lambda g'(t + v) v_\lambda] 
= Q[f(x, t + v)] + \lambda Q[f_u(x, t + v) v_\lambda] & 
\mbox{in $\Omega$}, \\
\dfrac{\partial v_\lambda}{\partial \mathbf{n}} = b [g(t + v) + \lambda  g'(t + v) v_\lambda] & \mbox{on $\partial \Omega$}. 
\end{cases}
\end{align*}
Put $\lambda = 0$, $t = c_0$, and $v(0,c_0)=0$, to obtain
\begin{align*}
\begin{cases}
-\Delta v_\lambda (0,c_0) 
+ \frac{1}{|\Omega|} \int_{\partial \Omega} b g(c_0) 
= Q[f(x, c_0)] & \mbox{in $\Omega$}, \\ 
\dfrac{\partial v_\lambda (0,c_0)}{\partial \mathbf{n}} 
= b g(c_0) & \mbox{on $\partial \Omega$}, 
\end{cases}
\end{align*}
where $$Q[f(x,c_0)] = mc_0 + a c_0^{p-1} - \frac{1}{|\Omega|} \int_\Omega 
(mc_0 + a c_0^{p-1}).$$ It follows that 
\begin{align*}
\begin{cases}
-\Delta v_\lambda (0,c_0) 
+ \frac{1}{|\Omega|} \{ \int_\Omega (mc_0 +a c_0^{p-1}) + 
\int_{\partial \Omega} b c_0^{q-1} \} 
= mc_0 + a c_0^{p-1} & \mbox{in $\Omega$}, \\
\dfrac{\partial v_\lambda (0,c_0)}{\partial \mathbf{n}} 
= b c_0^{q-1} & \mbox{on $\partial \Omega$},  
\end{cases}
\end{align*}
and consequently 
\begin{align*}
\begin{cases}
-\Delta v_\lambda (0,c_0) 
= mc_0 +a c_0^{p-1} & \mbox{in $\Omega$}, \\ 
\dfrac{\partial v_\lambda (0,c_0)}{\partial \mathbf{n}} 
= b c_0^{q-1} & \mbox{on $\partial \Omega$},  
\end{cases}
\end{align*}
since $\varphi(c_0)=0$. Hence
\begin{align*}
\int_\Omega |\nabla v_\lambda (0,c_0)|^2 - c_0^{q-1} \int_{\partial \Omega} b 
v_\lambda (0,c_0) = \int_\Omega (mc_0 + ac_0^{p-1}) v_\lambda (0,c_0). 
\end{align*}
From (4) we get
\begin{eqnarray*}
 c_0\Phi_\lambda(0,c_0)  &=& c_0 \int_\Omega mv_\lambda(0,c_0) + (p-1)c_0^{p-1}\int_\Omega a v_\lambda(0,c_0) + (q-1)c_0^{q-1}\int_{\partial \Omega}bv_\lambda(0,c_0) \\
& =& (q-1)\int_\Omega |\nabla v_\lambda (0,c_0)|^2 + \int_\Omega \left\{ 
(2-q)c_0 m + (p-q)c_0^{p-1} a \right\} v_\lambda(0,c_0).  
\end{eqnarray*}
Since $k$ is a positive constant and $a=-km$, we have
\begin{align*}
(2-q)c_0 m + (p-q)c_0^{p-1} a &= mc_0 \left\{ (2-q) - (p-q)c_0^{p-2}k 
\right\} \\ 
& = mc_0 \left\{ 
(2-q) - (p-q) \frac{(2-q)\int_\Omega m}{(p-q)(-\int_\Omega a)}k \right\}
\\ 
& = 0. 
\end{align*}
Therefore
\begin{align*}
c_0\Phi_\lambda(0,c_0) = (q-1) \int_\Omega |\nabla v_\lambda(0,c_0)|^2. 
\end{align*}
Moreover, since $b\not\equiv 0$, $v_\lambda(0,c_0)$ is not a constant, so that $\int_\Omega |\nabla v_\lambda(0,c_0)|^2 > 0$. 
\end{enumerate} 

The proof of Lemma \ref{lem:Phi} is now complete. 
\end{proof}

Theorem \ref{thm:a=-m:bif} (2) is then a direct consequence of the following result:
\begin{prop} \label{prop:turn}
Assume  \eqref{assump:mbreg}, \eqref{1<q<2<p}, $a=-km$ for some positive constant $k$ and $$\int_\Omega m > 0 > \int_{\partial \Omega}b = -\tilde{K}_1(m,a).$$ Then there exists a constant $\varepsilon > 0$ and a $\mathcal{C}^3$ function $\lambda:(c_0-\varepsilon,c_0+\varepsilon)\rightarrow \R
$ satisfying $\lambda (c_0) = \lambda'(c_0) = 0$ and $\lambda''(c_0) > 0$, such that the set 
\begin{align*}
\{ (\lambda (t), \ t + v(\lambda (t), t)) : 
t \in (c_0-\varepsilon,c_0+\varepsilon) \}
\end{align*}
is contained in the positive solutions set of $(P_\lambda)$. Moreover, the positive solution $t + v(\lambda (t), t)$ of $(P_{\lambda (t)})$ is asymptotically stable for $c_0 < t < c_0 + \varepsilon$ and unstable for $c_0 - \varepsilon < t < c_0$. 
\end{prop}

\begin{rem}{\rm 
From \eqref{c0} note that if $a=-km$ then $c_0 = c_0(k) = \left( \frac{2-q}{(p-q)k}\right)^{\frac{1}{p-2}}$. It follows that $k \mapsto c_0(k)$ is decreasing, $\lim_{k \to 0^+} c_0(k)=\infty$ and $\lim_{k \to \infty} c_0(k)=0$.
}\end{rem}

\begin{proof} 
%
By (\ref{Philamposi}) and the implicit function theorem,  we deduce that there exists a $\mathcal{C}^3$ function $t \mapsto \lambda (t)$  such that 
\begin{align*}
\Phi (\lambda, t) = 0 \ \ \mbox{for $(\lambda, t)$ close to $(0,c_0)$} \Longleftrightarrow 
(\lambda, t) = (\lambda (t), t) \ \ \mbox{for $t$ close to $c_0$}.
\end{align*}
From Lemma \ref{lem:Phi} (1) we have $\lambda (c_0) = 0$, whereas  from Lemma \ref{lem:Phi} (2), (3), and (\ref{Philamposi}), we have
\begin{align*}
\lambda'(c_0) = - \frac{\Phi_t (0,c_0)}{\Phi_\lambda (0,c_0)} = 0, 
\quad 
\lambda''(c_0) = -\frac{\Phi_{tt} (0,c_0)}{\Phi_\lambda (0,c_0)} 
>0, 
\end{align*}
as desired. 
 
We prove now the stability result. Recall $c_0$ is the unique zero of $\varphi$, given by \eqref{ephi}, as well as its global maximum point. So we have
\begin{align} \label{eq:140828c0i}
& c_0^{q-2}\varphi (c_0) =  \int_\Omega m + c_0^{p-2} \int_\Omega a + c_0^{q-2}\int_{\partial \Omega}b=0, \\
& c_0^{q-1}\varphi' (c_0)= (2-q)\int_\Omega m + (p-q)c_0^{p-2} \int_\Omega a = 0. \label{eq:140828c0ii}
\end{align} 
We let $w(t):= t + v(\lambda (t), t)$, and consider the stability of $(\lambda (t), w(t))$, $|t-c_0|<\varepsilon$. To this end, we study the linearized eigenvalue problem at $(\lambda (t), w(t))$, which is given by
\begin{align} \label{prob:eigen140829}
\begin{cases} 
-\Delta \psi = \lambda m \psi + \lambda a (p-1) w^{p-2}\psi + \gamma \psi & \mbox{in $\Omega$}, \\ 
\partial_{\mathbf{n}}\psi = \lambda b (q-1) w^{q-2}\psi + \gamma \psi & \mbox{on $\partial \Omega$}.
\end{cases}
\end{align}
Let $\gamma_1 = \gamma_1(t)$ be its smallest eigenvalue, which is simple, and $\psi_1 = \psi_1(t)$ be the positive eigenfunction associated to $\gamma_1$ satisfying $ \int_\Omega \psi_1^2 + \int_{\partial \Omega} \psi_1^2=1$. 

First, we claim that
\begin{align} \label{val140829gamm10}
\gamma_1(c_0) = 0, \quad \mbox{and} \ \ \psi_1(c_0)\equiv \left( \frac{1}{|\Omega|+|\partial \Omega|}\right)^{1/2}. 
\end{align}
Indeed, putting $t=c_0$ in \eqref{prob:eigen140829}, we have 
\begin{align*}
\begin{cases}
-\Delta \psi_1(c_0) = \gamma_1(c_0) \psi_1(c_0) & \mbox{in $\Omega$}, \\ 
\partial_{\mathbf{n}}\psi_1 = \gamma_1(c_0) \psi_1(c_0) & \mbox{on $\partial \Omega$}. 
\end{cases}
\end{align*}
By uniqueness, it follows that $\gamma_1(c_0)=0$ and 
$$
\psi_1(c_0)\equiv \left( \frac{1}{|\Omega|+|\partial \Omega|}\right)^{1/2}, 
$$
as claimed. 

Second, we show that
\begin{align} \label{val140829gamm11}
\gamma_1'(c_0) =0, \quad \mbox{and} \ \ \psi_1'(c_0)= 0.
\end{align}
To this end, we differentiate \eqref{prob:eigen140829} with respect to $t$ (with $\gamma=\gamma_1$ and $\psi=\psi_1$) to obtain 
\begin{align} \label{prob:140828psi1pri}
\begin{cases}
-\Delta \psi_1' = m\left(\lambda' \psi_1 + \lambda \psi_1' \right) +(p-1)a\left( \lambda' w^{p-2}\psi_1 + \lambda (p-2) w^{p-3} w'\psi_1  + \lambda  w^{p-2} \psi_1'\right) \\
\hspace*{2cm} + \gamma_1' \psi_1 + \gamma_1 \psi_1' & \mbox{in $\Omega$}, \\ 
\partial_{\mathbf{n}}\psi_1' = b (q-1)\left(\lambda' w^{q-2}\psi_1 + \lambda  (q-2)w^{q-3}w' \psi_1 + \lambda  w^{q-2}\psi_1' \right)
+ \gamma_1' \psi_1 + \gamma_1 \psi_1' & \mbox{on $\partial \Omega$}. 
\end{cases}
\end{align}
Taking $t=c_0$  in \eqref{prob:140828psi1pri} we get
\begin{align*}
\begin{cases}
-\Delta \psi_1'(c_0) = \gamma_1'(c_0) \psi_1(c_0) & \mbox{in $\Omega$}, \\ 
\partial_{\mathbf{n}}\psi_1'(c_0) = \gamma_1'(c_0) \psi_1(c_0) & \mbox{on $\partial \Omega$}. 
\end{cases}
\end{align*}
By Green's formula, we deduce that
\begin{align*}
& \int_\Omega (\gamma_1'(c_0) \psi_1(c_0)) + \int_{\partial \Omega} (\gamma_1'(c_0) \psi_1(c_0)) = 0, 
\end{align*}
so that $\gamma_1'(c_0) \psi_1(c_0) (|\Omega|+|\partial \Omega|)=0$. Hence, we have $\gamma_1'(c_0) = 0$, and consequently $\psi_1'(c_0)$ is a constant. On the other hand, differentiating $\int_\Omega \psi_1^2 + \int_{\partial \Omega} \psi_1^2 = 1$ with respect to $t$ we obtain 
\begin{align} \label{eq:140828psipsipri}
\int_\Omega \psi_1 \psi_1' + \int_{\partial \Omega} \psi_1 \psi_1' = 0. 
\end{align}
Thus we have $\psi_1(c_0) \psi_1'(c_0) (|\Omega|+|\partial \Omega|) = 0$, which implies $\psi_1'(c_0) = 0$, as desired.

Third, we verify that
\begin{align} \label{val140829gamm12}
\gamma_1''(c_0)=0.
\end{align}
Differentiating \eqref{prob:140828psi1pri} with respect to $t$, we obtain
\begin{align} \label{prob140828psi13}
\begin{cases}
-\Delta \psi_1'' = m\left( \lambda''  \psi_1 + 2\lambda'  \psi_1' + \lambda  \psi''\right) +(p-1)a\left(  \lambda'' w^{p-2}\psi_1 + \lambda' (w^{p-2}\psi_1)' \right)& \\ 
\hspace*{1.5cm} + \left[\lambda a (p-1)\left((p-2) w^{p-3} w' \psi_1 +  w^{p-2} \psi_1'\right)\right]' + \gamma_1'' \psi_1 + 2 \gamma_1' \psi_1' + \gamma_1 \psi_1'' & \mbox{in $\Omega$}, \\ 
\partial_{\mathbf{n}}\psi_1'' = b (q-1)\left[\lambda''  w^{q-2}\psi_1 + \lambda'  (w^{q-2}\psi_1)' \right]  + \left[\lambda b (q-1)\left( (q-2)w^{q-3}w' \psi_1 +  w^{q-2}\psi_1'\right)\right]' & \\
\hspace*{1.5cm} + \gamma_1'' \psi_1 + 2\gamma_1' \psi_1' + \gamma_1 \psi_1'' & \mbox{on $\partial \Omega$}. 
\end{cases}
\end{align}
Taking $t=c_0$ we get
\begin{align*}
\begin{cases}
-\Delta \psi_1''(c_0) = \lambda''(c_0) m \psi_1(c_0) + \lambda''(c_0) a (p-1) 
(w(c_0))^{p-2} \psi_1(c_0) + \gamma_1''(c_0) \psi_1(c_0) & \mbox{in $\Omega$}, \\ 
\partial_{\mathbf{n}}\psi_1''(c_0) = \lambda''(c_0) b (q-1) (w(c_0))^{q-2}\psi_1(c_0) + \gamma_1''(c_0) \psi_1(c_0) & \mbox{on $\partial \Omega$}. 
\end{cases}
\end{align*}
Since $\psi_1(c_0)$ is a positive constant, it follows by Green's formula that 
\begin{align*}
& \lambda''(c_0)\left( 
\int_\Omega m + (p-1) \int_\Omega a(w(c_0))^{p-2} + (q-1)\int_{\partial \Omega}
b(w(c_0))^{q-2} \right) \\
& \hspace*{1.5cm}+ \gamma_1''(c_0) (|\Omega|+|\partial \Omega|) = 0.
\end{align*}
Note that $w(c_0)=c_0$, which combined with \eqref{eq:140828c0i} and \eqref{eq:140828c0ii} implies 
\begin{align} \label{eq1408281842}
\int_\Omega m + (p-1) \int_\Omega a(w(c_0))^{p-2} + (q-1)\int_{\partial \Omega}
b(w(c_0))^{q-2} = 0,
\end{align}
and \eqref{val140829gamm12} follows. 

Finally, we verify that
\begin{align} \label{140829gamm13posi} 
\gamma_1'''(c_0) > 0. 
\end{align}
We differentiate once more \eqref{prob140828psi13} with respect to $t$ and take $t=c_0$. Since $\lambda (c_0)=\lambda'(c_0)=0$ and $\psi_1'(c_0)=0$ we deduce that
\begin{align*}
\begin{cases}
-\Delta \psi_1'''(c_0) = \lambda'''(c_0)m\psi_1(c_0)  + \lambda'''(c_0)a(p-1) (w(c_0))^{p-2} \psi_1(c_0) & \\
\hspace*{2cm} + 3\lambda''(c_0)a(p-1)(p-2) (w(c_0))^{p-3}w'(c_0)\psi_1(c_0) + \gamma_1'''(c_0)\psi_1(c_0) & \mbox{in $\Omega$}, \\
\partial_{\mathbf{n}}\psi_1'''(c_0) = \lambda'''(c_0)b(q-1)(w(c_0))^{q-2}\psi_1(c_0)  + 3 \lambda''(c_0)b(q-1)(q-2)(w(c_0))^{q-3}w'(c_0)\psi_1(c_0) 
& \\
\hspace{2cm} + \gamma_1'''(c_0)\psi_1(c_0) & \mbox{on $\partial \Omega$}. 
\end{cases}
\end{align*}
Since $w(c_0)=c_0$ and $\psi_1(c_0)$ is a positive constant, by Green's formula, we deduce that
\begin{align*}
0&= \gamma_1'''(c_0) (|\Omega|+|\partial \Omega|) + \lambda'''(c_0) \left\{ \int_\Omega \left( m + (p-1)a c_0^{p-2} \right) + (q-1) c_0^{q-2} \int_{\partial \Omega}b \right\} \\
&+ 3\lambda''(c_0) \left\{ (p-1)(p-2)c_0^{p-3} \int_\Omega a w'(c_0) + (q-1)(q-2) c_0^{q-3} \int_{\partial \Omega}b w'(c_0) \right\}. 
\end{align*}
Using \eqref{eq1408281842}, we see that the second term on the right hand side vanishes. In order to deal with the third term, we consider $w'(c_0)$. Note that, from the definition of $w$, we have $w'(t) = 1 + v_\lambda \lambda' + v_t$. Since $\lambda'(c_0)=0$, we have then $w'(c_0) = 1 + v_t(0,c_0) = 1$, where we have used \eqref{eq140829vt0}. From $a=-km$, it follows that 
$$
 \gamma_1'''(c_0) (|\Omega|+|\partial \Omega|) \\
 = 3\lambda''(c_0) \left\{ k(p-1)(p-2)c_0^{p-3} \int_\Omega m + (q-1)(2-q) c_0^{q-3} \int_{\partial \Omega}b \right\}.
$$

Now, we recall the assumption
\begin{align} \label{eq140829bm} 
\int_{\partial \Omega}b = -\tilde{K}_1(m,-km),
\end{align}
which yields $\int_{\partial \Omega}b = -\tilde{C}_{pq}k^{\frac{q-2}{p-2}}\int_\Omega m$, where 
\begin{align*}
\tilde{C}_{pq} = \left( \frac{q}{2} \left( \frac{p}{2} \right)^{\frac{2-q}{p-2}} \right)^{-1} \frac{q(p-2)}{2(p-q)} \left( \frac{p(2-q)}{2(p-q)} \right)^{\frac{2-q}{p-2}} = \frac{p-2}{p-q} \left( \frac{2-q}{p-q}\right)^{\frac{2-q}{p-2}}.  
\end{align*}
Moreover, from \eqref{eq:140828c0ii} we also have
\begin{align*}
c_0^{p-q} = k^{\frac{q-p}{p-2}}\left( \frac{2-q}{p-q}\right)^{\frac{p-q}{p-2}}. \end{align*}
Thus, from \eqref{eq140829bm} we deduce that
\begin{eqnarray*}
 \gamma_1'''(c_0) (|\Omega|+|\partial \Omega|)  
& =& 3\lambda''(c_0) c_0^{q-3} \left\{ (p-1)(p-2)c_0^{p-q}\int_\Omega m + (q-1)(2-q) \int_{\partial \Omega}b \right\}  \\
& =& 3\lambda''(c_0) c_0^{q-3} \left( \int_\Omega m \right) k^{\frac{q-2}{p-2}}\left( \frac{2-q}{p-q}\right)^{\frac{2-q}{p-2}} (p-2)(2-q). 
\end{eqnarray*}
From $\lambda''(c_0)>0$ and $\int_\Omega m > 0$, we infer \eqref{140829gamm13posi}. 

Summing up, from \eqref{val140829gamm10}, \eqref{val140829gamm11}, \eqref{val140829gamm12}, and \eqref{140829gamm13posi}, the desired conclusion follows. 

The proof of Proposition \ref{prop:turn} is now complete. 
\end{proof} 

\bigskip

\appendix 

\section{A comparison principle for mixed Dirichlet and Neumann nonlinear boundary conditions}\label{sec:a} 
In this Appendix, we provide a variant of the comparison principle 
proved by Ambrosetti, Brezis and Cerami \cite[Lemma 3.3]{ABC} to mixed Dirichlet and Neumann nonlinear boundary conditions. We consider the general boundary value problem with mixed nonlinear boundary conditions 
\begin{align} \label{gmp}
\begin{cases}
-\Delta u = f(x,u) & \mbox{in $D$}, \\
\frac{\partial u}{\partial \mathbf{n}} = g(x,u) & \mbox{on $\Gamma_1$}, \\
u=0 & \mbox{on $\Gamma_0$}, 
\end{cases}
\end{align}
where:
\begin{itemize}
\item $D$ is a bounded domain of $\R^N$ with smooth boundary $\partial D$.
\item $\Gamma_0, \Gamma_1 \subset \partial D$ are disjoint, open, and smooth $(N-1)$ dimensional surfaces of $\partial D$.
\item $\overline{\Gamma_0}, \overline{\Gamma_1}$ are compact manifolds with $(N-2)$ dimensional closed boundary $\gamma = \overline{\Gamma_0} \cap \overline{\Gamma_1}$ such that $\partial D = \Gamma_0 \cup \gamma \cup \Gamma_1$.
\item $f: \overline{\Omega}\times [0,\infty) \to \R$ and $g: \overline{\Gamma_1} \times [0,\infty) \to \R$ are continuous.\\
\end{itemize}

\begin{prop} \label{app:prop:comparison}
Under the above conditions, assume that for every $x \in D$, $t \mapsto \frac{f(x,t)}{t}$ is decreasing in $(0,\infty)$, and for every $x \in \Gamma_1$, $t \mapsto \frac{g(x,t)}{t}$ is non-increasing in $(0,\infty)$. Let $u,v \in H^1 (D) \cap \mathcal{C}(\overline{D})$ be non-negative functions satisfying $u= 0\leq v$ on $\Gamma_0$, and 
\begin{align*}
\int_D \nabla u \nabla \varphi - \int_D f(x,u) \varphi - \int_{\Gamma_1} g(x,u) \varphi \leq \ 0, \quad \forall \varphi \in H_{\Gamma_0}^1(D) \ \ \mbox{such that} \ \ \varphi \geq 0, 
\end{align*}
\begin{align*}
\int_D \nabla v \nabla \varphi - \int_D f(x,v) \varphi - \int_{\Gamma_1} g(x,v) \varphi \geq \ 0, \quad \forall \varphi \in H_{\Gamma_0}^1(D) \ \ \mbox{such that} \ \ \varphi \geq 0.  
\end{align*}
If $u, v > 0$ in $D$, then $u\leq v$ in $\overline{D}$. 
\end{prop}


\begin{proof} 
Let $\theta: \R \rightarrow \R$,  be a nonnegative nondecreasing smooth function such that $\theta (t)=0$ for $t\leq 0$ and $\theta (t) = 1$ for $t\geq 1$. For $\varepsilon > 0$ we set 
$\theta_\varepsilon (t) = \theta (t/\varepsilon)$. Since $u-v \leq 0$ on $\Gamma_0$, we have $v \theta_\varepsilon (u-v) \in H_{\Gamma_0}^1(D)$, so that 
\begin{align} \label{u:ineq}
\int_D \nabla u \nabla (v \theta_\varepsilon (u-v)) - \int_D f(x,u)v \theta_\varepsilon (u-v) - \int_{\Gamma_1} g(x,u) v \theta_\varepsilon (u-v) \leq 0. 
\end{align}
Likewise, since $u \theta_\varepsilon (u-v) \in H_{\Gamma_0}^1(D)$, 
we have 
\begin{align} \label{v:ineq}
\int_D \nabla v \nabla (u \theta_\varepsilon (u-v)) - \int_D f(x,v)u \theta_\varepsilon (u-v) - \int_{\Gamma_1} g(x,v) u \theta_\varepsilon (u-v) \geq 0. 
\end{align}
Let $\Gamma_1^+ = \{ x \in \Gamma_1 : u, v > 0 \}$. 
Since $t \mapsto \frac{g(x,t)}{t}$ is non-increasing in $(0,\infty)$, we have $g(x,0)\geq 0$, which combined with \eqref{u:ineq} and \eqref{v:ineq} yields
\begin{align*}
& \int_D  u \theta_\varepsilon^\prime (u-v) \nabla v (\nabla u - \nabla v) 
- \int_D  v \theta_\varepsilon^\prime (u-v) \nabla u (\nabla u - \nabla v) 
\\
& \geq \int_D uv \left( \frac{f(x,v)}{v} - \frac{f(x,u)}{u} \right) 
\theta_\varepsilon (u-v) 
+ \int_{\Gamma_1^+} uv \left( \frac{g(x,v)}{v} - \frac{g(x,u)}{u} \right) \theta_\varepsilon (u-v) \\
& \geq \int_D uv \left( \frac{f(x,v)}{v} - \frac{f(x,u)}{u} \right) 
\theta_\varepsilon (u-v). 
\end{align*}
From $-\int_D u \theta_\varepsilon^\prime (u-v) |\nabla (u-v)|^2 \leq 0$, it follows that  
\begin{align} \label{ineq140826}
\int_D  (u-v) \theta_\varepsilon^\prime (u-v) \nabla u\nabla (u-v) 
\geq \int_D uv \left( \frac{f(x,v)}{v} - \frac{f(x,u)}{u} \right) 
\theta_\varepsilon (u-v). 
\end{align}

Now, we introduce $\gamma_\varepsilon (t) = \int_0^t s \theta_\varepsilon^\prime (s) ds$ for $t\in \R$. We have then $0\leq \gamma_\varepsilon (t) \leq \varepsilon$, $t \in \R$. Note that $\nabla (\gamma_\varepsilon (u-v)) = (u-v) \theta_\varepsilon^\prime (u-v) \nabla (u-v)$. Hence, from \eqref{ineq140826} we deduce that 
\begin{align*}
\int_D \nabla u \nabla (\gamma_\varepsilon (u-v)) 
\geq \int_D uv \left( \frac{f(x,v)}{v} - \frac{f(x,u)}{u} \right) 
\theta_\varepsilon (u-v). 
\end{align*}
Now, since $\gamma_\varepsilon (u-v) \in H_{\Gamma_0}^1(D)$ and 
$\gamma_\varepsilon (u-v) \geq 0$, we note that
\begin{align*}
\int_D \nabla u \nabla (\gamma_\varepsilon (u-v)) 
- \int_D f(x,u) \gamma_\varepsilon (u-v) 
- \int_{\Gamma_1} g(x,u) \gamma_\varepsilon (u-v) \leq 0, 
\end{align*}
and combining the two latter assertions, we get
\begin{align*}
\int_D f(x,u) \gamma_\varepsilon (u-v) + \int_{\Gamma_1} g(x,u) \gamma_\varepsilon (u-v) \geq \int_D uv \left( \frac{f(x,v)}{v} - \frac{f(x,u)}{u} \right) \theta_\varepsilon (u-v). 
\end{align*}
Since $\gamma_\varepsilon (t)\leq \varepsilon$, there exists a constant $C>0$ such that 
\begin{align} \label{ineq:Cvarep}
C\varepsilon \geq \int_D uv \left( \frac{f(x,v)}{v} - \frac{f(x,u)}{u} \right) 
\theta_\varepsilon (u-v). 
\end{align}
Since $t \mapsto \frac{f(x,t)}{t}$ is decreasing in $(0,\infty)$, 
we use Fatou's lemma to deduce from \eqref{ineq:Cvarep} that  
\begin{align*}
\int_D \liminf_{\varepsilon \to 0^+}\, uv \left( \frac{f(x,v)}{v} - \frac{f(x,u)}{u} \right) \theta_\varepsilon (u-v) \leq 0. 
\end{align*}
Note that
\begin{align*}
\lim_{\varepsilon \to 0^+} \theta_\varepsilon (u-v) = \left\{ 
\begin{array}{ll}
1, & u > v, \\
0, & u\leq v, 
\end{array} \right. 
\end{align*}
so that 
\begin{align*}
\int_{\{ u>v \}} uv \left( \frac{f(x,v)}{v} - \frac{f(x,u)}{u} \right) \leq 0. 
\end{align*}
Using again that $t \mapsto \frac{f(x,t)}{t}$ is decreasing in $(0,\infty)$, we conclude $|\{ u>v \}|=0$, which implies $u\leq v$ a.e.\ in $D$. By continuity, the desired conclusion follows. 
\end{proof}

\section{Positivity of nontrivial non-negative weak solutions in the one-dimensional case}  \label{sec:posi}

In this Appendix, we show the positivity of nontrivial non-negative weak solutions for the one-dimensional case of $(P_\lambda)$.   We take $\Omega = I = (0,1)$ and show that  under some regularity assumptions on $m$ and $a$ a nontrivial non-negative solution satisfies $u > 0$ on $\overline{I}$. More precisely, we consider nontrivial non-negative weak solutions of the problem 
\begin{align} \label{p:1dim}
\begin{cases}
-u'' = \lambda (m(x)u + a(x) u^{p-1}) & \mbox{in $I$}, \\ 
-u'(0)=\lambda b_0 u(0)^{q-1}, & \\
u'(1) = \lambda b_1 u(1)^{q-1}, & 
\end{cases}
\end{align}
where $1<q<2<p$, $m, a \in \mathcal{C}^1(\overline{I})$, and $b_0, b_1 \in \R$. A non-negative function $u \in H^1(I)$ is a non-negative weak solution of \eqref{p:1dim} if it satisfies 
\begin{align*}
\int_I u'\phi' = \lambda \left( b_0 u(0)^{q-1}\phi (0) + b_1 u(1)^{q-1}\phi (1) \right) + \lambda \int_I (mu + au^{p-1})\phi , \quad \forall \phi \in H^1(I). 
\end{align*}

We prove here the following:

\begin{prop} \label{p:positivedef}
Let $b_0, b_1 \in \R$ be arbitrary. Then any nontrivial non-negative weak solution $u$ of \eqref{p:1dim} satisfies $u>0$ in $\overline{I}$. 
\end{prop}

\begin{proof}
If $u$ is a nonnegative weak solution of \eqref{p:1dim} then, thanks to the inclusion $H^1(I)\subset \mathcal{C}(\overline{I})$ (see \cite{Bre11}) we have $u \in \mathcal{C}(\overline{I})$. Moreover, we claim that $u \in H^2(I)$, so that $u \in \mathcal{C}^1(\overline{I})$. Indeed, from the definition we derive 
\begin{align*}
\int_I u'\phi' = \lambda \int_I (mu + au^{p-1}) \phi, \quad\mbox{$\forall \phi \in \mathcal{C}^1_c(I)$}.  
\end{align*}
This implies that $(u')' = -au^{p-1}$ in $I$ in the distribution sense. By the chain rule we obtain $mu + au^{p-1} \in H^1(I)$, since $m,a \in \mathcal{C}^1(\overline{I})$. By definition we infer that $u \in H^2(I)$. From the inclusion $H^2(I)\subset \mathcal{C}^1(\overline{I})$, it follows that $u \in \mathcal{C}^1(\overline{I})$.

In fact, by a bootstrap argument and elliptic regularity, we have $u \in \mathcal{C}^2(I)$. Hence, it follows that $u \in \mathcal{C}^1(\overline{I}) \cap \mathcal{C}^2(I)$, and we infer that $u > 0$ in $I$ by the strong maximum principle. In order to show that $u(0)>0$, we assume by contradiction that $u(0)=0$. Then the boundary point lemma tells us that $-u'(0)<0$. However, the boundary condition in \eqref{p:1dim} is understood in the classical sense under the condition $u \in \mathcal{C}^1(\overline{I}) \cap \mathcal{C}^2(I)$, and thus, $u'(0)=0$, which is a contradiction. Likewise we can show that $u(1)>0$. 
\end{proof}

\begin{rem}{\rm 
Using the same argument as in Proposition \ref{p:positivedef}, we infer that in the case $N=1$ nontrivial non-negative solutions of \eqref{w0}  satisfy $w_0 > 0$ on $\overline{\Omega}$. 
}\end{rem}



\bigskip 


\begin{thebibliography}{1000}
%

\bibitem{ADN}
S. Agmon, A. Douglis, and L. Nirenberg, Estimates near the boundary for solutions of elliptic partial differential equations satisfying general boundary conditions. I, Comm.\ Pure Appl.\ Math.\ {\bf 12}, (1959), 623--727. 



\bibitem{AT} S. Alama, G. Tarantello, On semilinear elliptic equations
with indefinite nonlinearities, Calc.\ Var.\ Partial Differential Equations {\bf 1} (1993), 439--475.


\bibitem{ABC}
A. Ambrosetti, H. Brezis, and G. Cerami, Combined effects of concave and convex nonlinearities in some elliptic problems, J.\ Funct.\ Anal.\ {\bf 122}, (1994), 519--543. 


\bibitem{BPT} C. Bandle, C, M. A. Pozio, A. Tesei, Existence and uniqueness of solutions of nonlinear Neumann problems, Math.\ Z.\ {\bf 199}, (1988), 257--278.


\bibitem{BCN94} H. Berestycki, I. Capuzzo-Dolcetta, and L. Nirenberg, Superlinear indefinite elliptic problems and nonlinear Liouville theorems, Topol.\ Methods Nonlinear Anal.\ {\bf 4}, (1994), 59--78. 


\bibitem{BCN95} H. Berestycki, I. Capuzzo-Dolcetta, and L. Nirenberg, Variational methods for indefinite superlinear homogeneous elliptic problems, NoDEA Nonlinear Differential Equations Appl.\ {\bf 2}, (1995), 553--572. 


\bibitem{Bre11} H. Brezis, Functional analysis, Sobolev spaces and partial differential equations, Universitext.\ Springer, New York, 2011. 


\bibitem{B07}  K. J. Brown, Local and global bifurcation results for a semilinear boundary value problem, J.\ Differential Equations {\bf 239}, (2007), 296--310.

\bibitem{BH} K. J. Brown and P. Hess, Stability and uniqueness of positive solutions for a semi-linear elliptic boundary value problem, Differential Integral Equations {\bf 3}, (1990), 201--207. 


\bibitem{BL} 
K.J. Brown, S.S. Lin, On the existence of positive eigenfunctions for an eigenvalue problem with indefinite weight function, J.\ Math.\ Anal.\ Appl.\ {\bf 75}, (1980), 112--120. 





\bibitem{BZ03} K. J. Brown and Y. Zhang, The Nehari manifold for a semilinear elliptic equation with a sign-changing weight function, J.\ Differential Equations {\bf 193}, (2003), 481--499. 
%


\bibitem{CC03}
R. S. Cantrell and C. Cosner, Spatial ecology via reaction-diffusion equations, Wiley Series in Mathematical and Computational Biology, John Wiley \& Sons, Ltd., Chichester, 2003.  


\bibitem{CFQ91}
M. Chipot, M. Fila and P. Quittner, Stationary solutions, blow up and convergence to stationary solutions for semilinear parabolic equations with nonlinear boundary conditions, Acta.\ Math.\ Univ.\ Comenianae {\bf 60}, (1991), 35--103.
%

\bibitem{F} W. H. Fleming, A selection-migration model in population genetics,  J.\ Math.\ Biol.\ {\bf 2}, (1975), 219--233. 


\bibitem{G-MM-RRS08} J. Garc\'ia-Meli\'an, C. Morales-Rodrigo, J. D. Rossi, and A. Su\'arez, Nonnegative solutions to an elliptic problem with nonlinear absorption and a nonlinear incoming flux on the boundary, Ann.\ Mat.\ Pura Appl.\ (4) {\bf 187}, (2008), 459--486. 


\bibitem{GRS09} J. Garc\'ia-Meli\'an, J. D. Rossi, and J. C. Sabina de Lis, Existence and uniqueness of positive solutions to elliptic problems with sublinear mixed boundary conditions, Commun.\ Contemp.\ Math.\ {\bf 11}, (2009), 585--613. 


\bibitem{GT} D. Gilbarg and N. S. Trudinger, Elliptic partial differential equations of second order, Second edition, Springer-Verlag, Berlin, 1983. 


\bibitem{GRLG00} R. G\'omez-Re\~nasco and J. L\'opez-G\'omez, The effect of varying coefficients on the dynamics of a class of superlinear indefinite reaction-diffusion equations, J.\ Differential Equations {\bf 167}, (2000), 36--72. 







%
\bibitem{LG} J. L\'opez-G\'omez, On the existence of positive solutions for some indefinite superlinear elliptic problems, Comm.\ Partial Differential Equations {\bf 22}, (1997), 1787--1804.


\bibitem{L-GMW93}
J. L\'opez-G\'omez, V. M\'arquez, and N. Wolanski, Dynamic behavior of positive solutions to reaction-diffusion problems with nonlinear absorption through the boundary, Rev.\ Un.\ Mat.\ Argentina {\bf 38}, (1993), 196--209. 


\bibitem{M-RS05} 
C. Morales-Rodrigo and A. Su\'arez, Some elliptic problems with nonlinear boundary conditions, Spectral theory and nonlinear analysis with applications to spatial ecology, 175--199, World Sci. Publ., Hackensack, NJ, 2005. 


\bibitem{Ou} T. Ouyang, On the positive solutions of semilinear equations $\Delta u+\lambda u+hu^p=0$ on compact manifolds. II, Indiana Univ.\ Math.\ J.\ {\bf 40}, (1991),  1083--1141. 

%
%


\bibitem{RQU} H. Ramos Quoirin and K. Umezu, The effects of indefinite nonlinear boundary conditions on the structure of the positive solutions set of a logistic equation, J.\ Differential Equations {\bf 257}, (2014), 3935--3977. 
 

\bibitem{RQU3} H. Ramos Quoirin and K. Umezu, Bifurcation for a logistic elliptic equation with nonlinear boundary conditions: A limiting case, J.\ Math.\ Anal.\ Appl.\ {\bf 428}, (2015), 1265--1285. 


\bibitem{R} J. D. Rossi, Elliptic problems with nonlinear boundary conditions and the Sobolev trace theorem, Stationary partial differential equations, Vol.II, 311--406, Handb. Differ. Equ., Elsevier/North-Holland, Amsterdam, 2005. 


\bibitem{Stam}
G. Stampacchia, Problemi al contorno ellitici, con dati discontinui, dotati di soluzionie h\"olderiane, Ann.\ Mat.\ Pura Appl.\ (4) {\bf 51}, (1960), 1--37. 


\bibitem{S} Nikos M. Stavrakakis, Global bifurcation results for semilinear elliptic equations on ${\bf R}\sp N$: the Fredholm case, J.\ Differential Equations {\bf 142}, (1998), 97--122.


%

\bibitem{U12} K. Umezu, Bifurcation approach to a logistic elliptic equation with a homogeneous incoming flux boundary condition, J.\ Differential Equations {\bf 252}, (2012), 1146--1168. 


\bibitem{U13} K. Umezu, Global structure of supercritical bifurcation with turning points for the logistic elliptic equation with nonlinear boundary conditions, Nonlinear Anal.\ {\bf 89}, (2013), 250--266. 


\bibitem{V} J. L. V\'azquez, A strong maximum principle for some quasilinear elliptic equations. Appl.\ Math.\ Optim.\ {\bf 12}, (1984), 191--202.


\bibitem{W06} 
T.-F. Wu, A semilinear elliptic problem involving nonlinear boundary condition and sign-changing potential, Electron.\ J.\ Differential Equations {\bf 2006}, No.131, 15 pp. 


\end{thebibliography}
\end{document}